\documentclass[12pt,twoside]{article}
\usepackage{graphicx}
\usepackage{styleset}
\usepackage{macros}
\let\phi\varphi
\title  {$\SU(3)$ instanton homology for webs and foams}

\author {P. B. Kronheimer and T. S. Mrowka}

\address {Harvard University, Cambridge MA 02138 \\
	Massachusetts Institute of Technology, Cambridge MA 02139}       

\begin{document}
	
\maketitle

\tableofcontents

\section{Introduction}

\subsection{Statement of a result}

In \cite{KM-Tait} and \cite{KM-triangles}, the authors introduced an
instanton homology for webs and foams, based on orbifold connections
with structure group $\SO(3)$. In this context, a \emph{web} $K$ in an
oriented 3-manifold $Y$ is an embedded trivalent graph, and the pair
$(Y,K)$ is interpreted as describing a 3-dimensional orbifold whose
local groups are elementary abelian of order $2$ or $4$ at the edges
and vertices of $K$ respectively. Similarly a \emph{foam} in an
oriented $4$-manifold $X$ is a singular 2-complex $\Sigma\subset X$
with particularly restricted singular structure, such that the pair
$(X,\Sigma)$ determines a $4$-dimensional orbifold with local groups
that now include the elementary abelian group of order $8$ at isolated
points of $\Sigma$. We refer to these orbifolds as 3- and
4-dimensional \emph{bifolds} respectively. They give rise to a
cobordism category in which the objects are closed, oriented
3-dimensional bifolds, and the morphisms are oriented 4-dimensional
bifolds with boundary. See \cite{KM-Tait}.

The instanton homology of \cite{KM-Tait} is constructed using the
spaces of orbifold $\SO(3)$ connections over these bifolds, with the
restriction that the action of the local group at an edge of a web, or
at a facet of a foam, must be the nontrivial action of the cyclic group
$\{\pm 1\}$ on the fiber $\R^{3}$. We refer to these as \emph{bifold}
$\SO(3)$ connections.

The present paper explores how the constructions and results of
\cite{KM-Tait} and \cite{KM-triangles} evolve when the structure group
$\SO(3)$ in the gauge theory is replaced by $\SU(3)$. The local models
for orbifold $\SU(3)$ connections that we consider arise naturally
from the $\SO(3)$ models via the inclusion
\begin{equation}\label{eq:r}
            \r : \SO(3)\to \SU(3),
\end{equation}
and we refer to the $\SU(3)$ connections with these models again as
bifold connections. See section~\ref{sec:bifolds} for details.

To construct the $\SO(3)$ instanton homology in \cite{KM-Tait}, an
additional connected-sum construction was used, as a device to avoid
reducible connections. There is more than one version of such a
construction, but with this understood, what is obtained in
\cite{KM-Tait} is a functor $\Jsharp$ from the category of closed
3-dimensional bifolds and 4-dimensional bifold cobordisms to the
category of finite-dimensional $\F$-vector spaces, where $\F$ is the
2-element field. Using $\SU(3)$ in place of $\SO(3)$, we will
similarly define an instanton homology group $\Lsharp(Y,K)$ for webs
$K$ in closed 3-dimensional manifolds $Y$, or equivalently for
$3$-dimensional bifolds. This $\SU(3)$ instanton homology group is
functorial for foams, or $4$-dimensional bifolds, just as in the
$\SO(3)$ case. The connected-sum construction that we will use to
avoid reducible connections will involve summing with an orbifold
whose singular locus is a Hopf link in $S^{3}$, but this will not be
of ``bifold'' type: the orbifold stabilizer along the singular locus
will be cyclic of order $3$, not $2$. This use of a ``trifold'' will
be confined to this particular role only. The details of this
construction are given in section~\ref{subsec:atoms}.

Among the results we obtain is the following theorem.

\begin{theorem}\label{thm:planar-tait}
    If $K\subset \R^{2} \subset S^{3}$ is a planar web, then the dimension
    of the $\SU(3)$ instanton homology $\Lsharp(S^{3},K)$, as
    a vector space over the field $\F$ of two elements, is equal to the
    number of Tait colorings of $K$.
\end{theorem}

\begin{remark}
The corresponding statement for planar webs in $\SO(3)$ homology
$\Jsharp$ is stated as a conjecture in \cite{KM-Tait}. 
\end{remark}

The proof of Theorem~\ref{thm:planar-tait} depends in an important way
on the specific connected sum construction (the trifold) that is used
to avoid reducible connections in the construction of $\Lsharp(K)$.
A significant proportion of the proof, however, rests on properties
of this bifold instanton homology that are less sensitive to this
particular choice. These are the skein exact triangles and the
octahedral diagram described in
Theorem~\ref{thm:octahedron-statement}.

The second main ingredient in the proof of
Theorem~\ref{thm:planar-tait} is an application of the decomposition
of the $\SU(3)$ instanton homology that arises from the eigenspace
decompositions for commuting operators associated to the edges of a
web. It is here that the choice of the trifold to avoid reducibles is
relevant. The argument here is very similar to the argument used by the
authors in \cite{KM-deformation}.

\begin{remark}
Although we work exclusively over the field $\F$ of two elements in
this paper, the authors have no evidence that one cannot define
$\Lsharp$ also over the integers, at least as a projective functor. In
particular, all the moduli spaces that are used to define $\Lsharp(K)$
itself (rather than the maps arising from cobordisms) are orientable.
The authors hope to return to this in a subsequent paper. 
\end{remark}

\subsection{Outline}
\label{subsec:outline}

Section~\ref{sec:bifolds} sets up the gauge theory that is needed to
define the $\SU(3)$ instanton homology. The construction of $\Lsharp$
itself is then given in section~\ref{sec:floer}.

The skein exact triangles and the octahedral diagram of
Theorem~\ref{thm:octahedron-statement} are discussed and proved in
section~\ref{sec:exact-triangle}. This is exactly parallel to a
corresponding result for the $\SO(3)$ case \cite{KM-triangles} and the
proof is essentially unchanged here. It rests, in particular, on the
identification of the $\SU(3)$ moduli spaces for some particular
closed foams, discussed in section~\ref{subsec:RP2-foams}. Although
the descriptions of these moduli spaces in
section~\ref{subsec:RP2-foams} mirror the results from the $\SO(3)$
case closely, it is interesting that the details look rather different
(for example because the dimension formula for the moduli spaces has
changed).

The eigenspace decomposition of the $\SU(3)$ instanton homology is
introduced in section~\ref{sec:edge-decomp}, and it is used together
with the skein exact triangles to complete the proof of
Theorem~\ref{thm:planar-tait}.

Unlike the $\SO(3)$ homology $\Jsharp(K)$, the $\SU(3)$ homology
$\Lsharp(K)$ admits a relative $\Z/2$ grading. This is discussed in
section~\ref{sec:absolute}, where it is shown that this relative
grading can be made an absolute $\Z/2$ grading by choosing an extra
piece of framing data for $K$. With this understood, one can consider
the Euler number of $\Lsharp$ as an integer invariant of webs. This
invariant is described and calculated in section~\ref{subsec:euler}.

The remaining sections of the paper include the calculation of further
examples, and a discussion of related questions.

\subparagraph{Acknowledgements.} The work of the first author was
supported by the National Science Foundation through NSF grants
DMS-2005310 and DMS-2304877. The work of the second author was
supported by NSF grant DMS-2105512 as well as by the Institute of
Theoretical Studies of ETH, the Department of Mathematics and the
Mathematics Research Center at Stanford University, and the Department
of Mathematics and the Minerva Foundation's Visitor Program at
Princeton University during the second author's sabbatical. Both
authors were supported by a Simons Foundation Award \#994330,
\emph{Simons Collaboration on New Structures in Low-Dimensional
Topology}.

\section{Bifolds and \texorpdfstring{$\SU(3)$}{SU(3)} instantons}
\label{sec:bifolds}

\subsection{Orbifolds and bifolds}

As in \cite{KM-Tait}, we consider a restricted class of oriented $3$-
and $4$-dimensional orbifolds which we call \emph{bifolds}. We will
typically use the notation $Y$ or $X$ for a manifold of dimension $3$
or $4$, and write $\check Y$ or $\check X$ for a bifold. We ask that
at each point $x$ in a bifold there is a neighborhood $U$ and an
orbifold chart
\begin{equation}\label{eq:orbichart}
 \phi: \tilde U \to U
\end{equation}
identifying $U$ with the quotient $\tilde U / H_{x}$, where the local
group $H_{x}$ is elementary abelian of order $2$, $4$, or (in
dimension 4 only) order $8$. This condition determines the
3-dimensional models completely and means that the singular set of the
bifold is a trivalent graph. In dimension $4$, the models allowed are
the product of the $3$-dimensional models with $\R$ if the order of
$H_{x}$ is $2$ or $4$. If the order is $8$, then the group $H_{x}$
should act on the $4$-dimensional tangent space $T_{x}\tilde U$ as the
group of diagonal matrices with entries $\pm 1$ and determinant $1$.

With these restrictions understood, our bifolds in dimensions $3$ and
$4$ have an underlying topological space which is also a manifold. The
singular set (the set of points with non-trivial local group $H_{x}$)
is a singular subcomplex of codimension $2$, referred to in this
context as a \emph{web} or \emph{foam} respectively. Note that the
edges of a web are not oriented, and the 2-dimensional facets of a
foam may be non-orientable. In a foam, the set of points where the
local stabilizer $H_{x}$ has order $4$ form arcs, which we call
\emph{seams}. Three facets of the foam locally meet along a seam.
At points $x$ where $H_{x}$ has order $8$, the singular set of the
orbifold locally has the structure of a cone on the $1$-skeleton of a
tetrahedron, and we call such points \emph{tetrahedral points}.

Our bifolds (of dimension either three or four) will always be
oriented, and will always be equipped with an orbifold Riemannian
metric: these will both be tacitly implied sometimes in the
exposition.

\subsection{Bifold bundles and connections}

An orbifold Riemannian metric can be conveniently described as as a
smooth Riemannian metric on the non-singular part of the bifold with
the requirement that the pull-back of the metric extends to a smooth
metric on the domain $\tilde U$ of each orbifold chart
\eqref{eq:orbichart}. In the same spirit, we define an orbifold bundle
with connection on a bifold $\check X$ to be a smooth bundle with
connection, $(E,A)$, on the smooth part of $\check X$ with the
requirement that the pull-back $(\tilde E, \tilde A)$ via $\phi$
should admit an extension as a smooth bundle with connection on the
entire domain $\tilde U$. Once the chart \eqref{eq:orbichart} is
given, the extension $(\tilde E, \tilde A)$ is unique up to canonical
isomorphism in this setting, and it therefore does not need to be
included as part of the data. Phrasing the definition this way, we are
required to equip every orbifold vector bundle with a connection,
something which we will often omit from mentioning. Orbifold
connections in any Sobolev class can be defined this way, by reference
to an underlying $C^{\infty}$ orbifold connection.

Given an orbifold bundle, there is a well-defined action of the
local group $H_{x}$ on the fiber $\tilde E_{x}$ at every point. In
this paper we will be concerned with the case that $E$ is a complex
hermitian bundle of rank $3$, and we restrict the local models by
requiring that, at each point $x$ where the the local group $H_{x}$
has order $2$, the action of the non-trivial element on the fiber
$\tilde E_{x} \cong \C^{3}$ is by the element
\[
\begin{pmatrix}
    1 & 0 & 0\\
    0 & -1 &0\\
    0 & 0 &-1
\end{pmatrix}
\]
in some basis of $\tilde E_{x}$. This requirement determines what the
group action on the fiber must be at points where the local group has
order $4$ or $8$. These actions on $\tilde E_{x} \cong \C^{3}$ are
the complexifications of the required actions in the $\SO(3)$ case
treated in \cite{KM-Tait}. We will refer to a hermitian orbifold
bundle satisfying these conditions as a \emph{bifold} bundle with a
bifold connection. When necessary, we will refer to real, or $\SO(3)$,
bifold bundles to distinguish the case considered in \cite{KM-Tait}.

\subsection{The configuration space}
\label{subsec:config}

Given a bifold $\check Y$ or $\check X$, closed for now and of
dimension $3$ or $4$, we consider the space of all pairs $(E,A)$
defining bifold bundles with connection. We allow $A$ to have Sobolev
class $L^{2}_{l}$ for a suitable choice of $l$, and we write
$\bonf_{l}(\check Y)$ or $\bonf_{l}(\check X)$ for the space of
isomorphism classes of such pairs.

As usual, the local model for $\bonf_{l}$ at an element $[E,A]$ has
the form $\mathcal{S}/\Gamma$, where $\mathcal{S}$ is a Hilbert space
and $\Gamma$ is the automorphism group of $(E,A)$. The group $\Gamma$
can be identified with the group of parallel sections of the bundle of
groups $\SU(E)$, which always contains a subgroup $\Z/3$ arising from
the parallel sections with values in the center of $\SU(E)$. If the
bifold is connected, then $\Gamma$ is isomorphic to a subgroup of
$\SU(3)$ and the possibilities for $\Gamma$ are:
\begin{enumerate}
    \item $\Gamma = Z(\SU(3))$, the center, a cyclic group of order
    $3$;
    \item $\Gamma \cong S^{1}$, which occurs only if $E$ has a parallel
    direct sum decomposition as $E_{1}\oplus E_{2}$ where $E_{1}$ has
    rank $1$ and $E_{2}$ is irreducible;
    \item $\Gamma \cong T^{2}$, the maximal torus of $\SU(3)$, which
    occurs only if $E=L_{1}\oplus L_{2}\oplus L_{3}$, a sum of three
    orbifold line bundles, preserved by the connection;
   \item $\Gamma  = S(U(1)\times U(2))\cong U(2)$ which occurs only if
   $E=L_{1}\oplus L_{2}\oplus L_{3}$ and $L_{2}\cong L_{3}$ as bundles
   with connection; or
   \item $\Gamma=\SU(3)$, which occurs only if the singular
   locus of the bifold is empty and $(E,A)$ is either trivial or has
   holonomy contained in the center, $\Z/3$.
    \end{enumerate}
In the first case, we say that $(E,A)$ is \emph{irreducible}. In all other
cases, $(E,A)$ is \emph{reducible}.
Dropping the Sobolev subscript, we will usually write $\bonf$ for the
space of bifold connections and $\bonf^{*}\subset\bonf$ for the
subspace of irreducibles.

For a bifold connection $(E,A)$ on a 4-dimensional bifold, we
write $\kappa$ as usual for the 4-dimensional Chern-Weil integral
\[
        \kappa = \frac{1}{8\pi^{2}}\int_{\check X}\tr (F\wedge F),
\]
where $F$ is the curvature. The constant is normalized so that the
integral is $$\kappa = c_{2}(E)[\check X]$$ in the closed case. The
characteristic classes here are defined in an orbifold sense and
$\kappa$ is not necessarily an integer.

\subsection{Classification of bifold bundles}
\label{subsec:classify}

Let a closed, oriented bifold $\check X$ of dimension four be given,
and let $(E,A)$ represent an element of $\bonf(\check X)$. If
$x\in\check X$ is a point which is not an orbifold point, then we can
modify $E$ by forming a connected sum at $x$ with a bundle $I$ on
$S^{4}$, with non-zero Chern class. In this way (with inevitable
choices for extending the connection $A$) we obtain a new bifold
bundle $(E',A')$ whose action $\kappa'$ differs from
$\kappa=\kappa(E)$ by an integer, the second Chern class
$c_{2}(I)(S^{4})$. We refer to this topological change as ``adding (or
subtracting) instantons''. If $f\in \check X$ is a point on a
two-dimensional facet of the singular set, where the local orbifold
group has order $2$, then there is a similar construction, forming an
orbifold connected sum with a bifold bundle $I$ on $S^{4}/(\Z/2)$.
This changes $\kappa$ by a multiple of $1/2$. We refer to this as
``adding (or subtracting) half-instantons'' to $E$. (Note that a
bifold bundle on $S^{4}/(\Z/2)$ has two characteristic numbers,
referred to in \cite{KM-gtes1} for example as the instanton and
monopole numbers.)

\begin{proposition}\label{prop:classify}
    Given 
    two bifold connections $(E,A)$ and $(E',A')$ in
    $\bonf(\check X)$,  we can obtain $E'$ from $E$ by adding
    or subtracting instantons and half-instantons (and we require the
    former only if the bifold is a manifold).
\end{proposition}

\begin{proof}
    This is an obstruction theory argument. The local requirements of
    our $\SU(3)$ bifold bundles determine the local structure at all
    the tetrahedral points, so we may choose an isomorphism between
    $E$ and $E'$ in the neighborhood of these points. Along the seams,
    the action of the orbifold monodromy on the fibers of the vector
    bundles has commutant the maximal torus of $\SU(3)$, and because
    this is a connected group, the isomorphism between the two bifold
    bundles can be extended from the vertices along the seams. After a
    further extension along the 1-skeletons of the facets of the
    foams, the two bundles have been identified in a neighborhood of
    the singular set, except on 2-cells in the facets, where the
    difference between them is the addition of half-instantons. The
    remaining difference between them will be on the interiors of the
    4-cells: that is, by the addition of instantons. Furthermore, the
    addition of two half-instantons on the same facet can be done in
    such a way that the effect is the same as adding an instanton
    nearby. (See \cite{KM-gtes1}.) So the only the addition of
    half-instantons is eventually needed if the singular set is
    non-empty.
\end{proof}

\begin{corollary}\label{cor:kappa-half}
    If $[E,A]$ and $[E', A']$ are two elements of
    $\bonf(\check X)$ on a closed oriented bifold $\check X$, then
    $\kappa(E')-\kappa(E)$ is a multiple of $1/2$.
\end{corollary}

\begin{remark}
     The corollary is in contrast to the situation with $\SO(3)$
     described in \cite{KM-Tait}, where the difference can be an odd
     multiple of $1/8$, as occurs for example when the singular set
     (the foam) is the suspension of the 1-skeleton of a tetrahedron.
\end{remark}

\subsection{The inclusion of \texorpdfstring{$\SO(3)$}{SO(3)}}
\label{sec:includeSO3}

Let $(E_{r}, A_{r})$ be an $\SO(3)$ bifold connection on $\check X$,
in the sense of \cite{KM-Tait}, and let $\kappa_{r}$ be its action,
normalized as usual to yield the second Chern class of a lift to
$\SU(2)$, if it exists. From the inclusion $\r: \SO(3)\to \SU(3)$, we
obtain an $\SU(3)$ bifold connection $(E, A)$ with trivial
determinant.

\begin{lemma}\label{lem:times-4}
    The action $\kappa(E,A)$ is related to $\kappa_{r}(E_{r}, A_{r})$
    by $\kappa=4\kappa_{r}$.
\end{lemma}

\begin{proof}
    For the $\SO(3)$ bundle we have $\kappa_{r} =
    -4p_{1}(E_{r})[\check X]$. On
    the other hand $p_{1}(E_{r})$ is (by definition)
    $-c_{2}(E_{r}\otimes \C)$, which is $-c_{2}(E)$. 
\end{proof}

    Note that the topological classification of $\SO(3)$ bifold
    bundles is more complicated than the $\SU(3)$ case, just as the
    Stiefel-Whitney class $w_{2}(E_{r})$ (in the case of a manifold)
    disappears on passing to $E = E_{r}\otimes \C$. Let us write
    $\bonf_{r}(\check X)$ for the space of $\SO(3)$ bifold connections
    and $\r : \bonf_{r}(\check X) \to \bonf(\check X)$ for the map
    given by complexification.

\begin{lemma}\label{lem:involution-fixed}
    On the space of irreducible $\SU(3)$ bifold connections, the
    involution $\sigma : \bonf^{*}(\check X) \to \bonf^{*}(\check X)$
    given by complex conjugation has fixed point set which is the
    injective image of\/ $\r: \bonf^{*}_{r}(\check X) \to
    \bonf^{*}(\check X)$.
\end{lemma}

\begin{remark}
    In line with the comment above, the map $\r$ is not injective on
    $\pi_{0}$.
\end{remark}

\begin{proof}[Proof of the lemma.]
     If $[E,A]\in \bonf(\check X)$ is fixed by $\sigma$, then there is
     a special-unitary isomorphism $u : E \to \bar E$ respecting the
     connections $A$ and $\bar A$. Regarding $u$ as a conjugate-linear
     map $E\to E$, we consider $u^{2}$, which must belong to the
     automorphism group $\Gamma$ for $(E,A)$. Because $(E,A)$ is
     irreducible, the element $u^{2}$ is a scalar automorphism: it is
     multiplication by an element of the center of $\SU(3)$, the
     cyclic group of order $3$. On the other hand, the fact that
     $u^{2}$ commutes with the conjugate-linear map $u$ means that the
     set of eigenvalues of $u^{2}$ is invariant under complex
     conjugation. We must therefore have $u^{2}=1$. This means that
     $u$ is a real structure on the complex vector bundle $E$, and $E$
     together with its connection are the complexification of the
     fixed subbundle.

     To establish that $\r$ is injective, suppose that $E$ and $E'$
     are two $\SO(3)$ irreducible bifold bundles with connection, and
     that the complexifications are isomorphic as $\SU(3)$ bundles
     with connection by a map $v: E\otimes \C \to E' \otimes \C$. If
     $J$ denotes complex conjugation then $ v^{-1} J v J$ is an
     automorphism of $E\otimes \C$, which must be multiplication by a
     cube root of unity $\omega$ (possibly $1$). Replacing $v$ by
     $\omega^{2}v$ we obtain an isomorphism $E\otimes \C \to E'
     \otimes \C$ which commutes with complex conjugation and is
     therefore an isomorphism of the $\SO(3)$ bifold bundles.
\end{proof}

\begin{remark}
    Without the hypothesis that $(E,A)$ is irreducible, it is possible
    that $[E,A]$ is fixed by the involution even though $(E,A)$ is not
    the complexification of an $\SO(3)$ bifold connection: the
    remaining possibility is that the structure group of $(E,A)$
    reduces to $\{1\}\times \mathrm{Sp}(1)=\{1\}\times \SU(2)\subset
    \SU(3)$, or a subgroup thereof. This can happen only when the
    singular set has no seams, because the $V_{4}$ monodromy at the
    seams is not a subgroup of $\{1\}\times \SU(2)$. In this case,
    when the singular set is a surface, such a bifold bundle has the
    form $\C \oplus E'$, where $E'$ has structure group $\SU(2)$ and
    has holonomy $-1$ along the link of the singular set. 
\end{remark}

\subsection{Anti-self-dual bifold connections}

Let $\check X$ be a closed, oriented bifold of dimension $4$.
Inside the space of bifold connections $\bonf(\check X)$, there is
the moduli space
\[
                M(\check X) \subset \bonf(\check X)
\]
consisting of those connections with $F^{+}=0$, where $F$ as above is
the curvature. At a solution $[E,A]$, there is the usual deformation
complex describing the local structure of the moduli space, and the
index of that complex is the \emph{formal dimension} of $M(\check X)$
at $[E,A]$.

In the following proposition, we use the notation and definitions from
\cite{KM-Tait}. In particular $X$ denotes the underlying 4-manifold of
$\check X$ and $\Sigma$ denotes the singular set of the orbifold. The
self-intersection number $\Sigma\ccdot \Sigma$ is defined as in
\cite[Definition 2.5]{KM-Tait}, and may be a half-integer. The term
$t$ is the number of tetrahedral points (the 4-valent vertices of the
graph formed by the seams).

\begin{proposition}\label{prop:dimension}
    The formal dimension of the moduli space $M(\check X)$ at $[E,A]$ is given
    by the formula
    \[
               d = 12 \kappa(E) - 8 (b^{+}(X) - b^{1}(X)+1)
                + \Sigma\ccdot \Sigma + 2\chi(\Sigma) - t.                                                          
    \]
\end{proposition}

\begin{proof}
   We use an excision argument, which can be closely modeled on the
   proof of the corresponding result for the $\SO(3)$ case in
   \cite[Proposition~2.6]{KM-Tait}. If the singular set is empty, this
   is the formula for the dimension of the instanton moduli space from
   \cite{AHS}. If the singular set is a non-empty orientable
   2-manifold (i.e.~has no seams), then the formula is the same as
   that in \cite{KM-yaft}. For the case that the singular set is a
   non-orientable surface, it is sufficient to verify one case for
   each possible Euler number of a connected non-orientable surface.
   For this we can take $\check X$ to be the total space of a 2-sphere
   bundle with orientable total space over a non-orientable surface
   $\Sigma$. We may construct this 2-sphere bundle as the fiberwise
   compactification of a real $2$-plane bundle, and the singular set
   will be taken to be a copy of $\Sigma$ arising as the zero section.
   The terms on both sides of the dimension formula are multiplicative
   under (unbranched) finite covers, so we can pass to the double
   cover of this bifold in which the singular set $\Sigma$ lifts to
   its orientable double cover, so reducing to the case that the
   singular surface is orientable.

    It remains to consider the case that $\Sigma$ has seams and
    possibly tetrahedral points. The argument in
    \cite[Proposition~2.6]{KM-yaft} needs essentially no modification
    except for a final step. As in the earlier paper, we must verify
    the dimension formula explicitly for the case that $\check X$ is
    $S^{4}/V_{8}$, where $V_{8}\subset \SO(4)\subset \SO(5)$ is the
    elementary abelian 2-group, acting so that the singular set
    $\Sigma$ is the suspension of the tetrahedral web, with $t=2$. We
    can take $(E,A)$ to be a flat bifold bundle with monodromy group
    $V_{4}$. The automorphism group $\Gamma$ for this bifold bundle is
    the maximal torus $T^{2}$ and index the $d$ of the deformation
    complex is $-2$. On the other hand, we have
    \[
             \begin{aligned}
                12 \kappa(E) - 8 (b^{+}(X) - b^{1}(X)+1)
                + \Sigma\ccdot \Sigma + 2\chi(\Sigma) - t &= 0  - 8 + 0
                + 2\times 4 - 2\\
                &=-2
             \end{aligned}
    \]
    also, so the result is proved in this remaining case.
\end{proof}

\begin{corollary}
    If $[E, A]$ and $[E',A']$ are two elements in $M(\check X)$ for a
    closed 4-dimensional bifold $\check X$, then the formal dimensions
    $d$ and $d'$ of the moduli space at these two points differ by a
    multiple of $6$.
\end{corollary}

\begin{proof}
    By Corollary~\ref{cor:kappa-half}, the actions $\kappa$ and
    $\kappa'$ differ by a multiple of $1/2$. The other terms in the
    dimension formula are the same.
\end{proof}

We can be more precise about the dimension mod $6$ in the above
corollary. The next proposition gives a formula for the formal
dimension mod $6$ in terms only of the topology of the bifold,
without referring directly to $\kappa$.

\begin{proposition}\label{prop:dim-mod-6}
    Let $\check X$ be a closed, oriented $4$-dimensional bifold with
    underlying 4-manifold $X$ and foam $\Sigma \subset X$. Then at any
    $[E,A]\in M(\check X)$, the formal dimension of the moduli space
    mod $6$ is given by
    \[
            \dim M(\check X) = - 2\Bigl( b^{+}(X) - b^{1}(X) + 1\Bigr)
            + 2\Bigl( - \Sigma
            \ccdot \Sigma
            + \chi(\Sigma) + t\Bigr), \pmod 6.
    \]
\end{proposition}

Before proving the proposition, we recall that $\Sigma 
\ccdot  \Sigma$
may be a half-integer for a foam. In particular, the above formula
does not imply that the dimension is even. However, we do have the
following immediate corollary.

\begin{corollary}\label{cor:dimension-parity}
    In the above situation, we have
    \[
              \dim M(\check X) = 2 \Sigma
           \ccdot \Sigma,
             \pmod 2.
    \]\qed
\end{corollary}

\begin{proof}[Proof of Proposition~\ref{prop:dim-mod-6}]
    We exploit the fact that there exist $\SO(3)$ bifold connections
    on $\check X$, so there is a non-empty inclusion $\bonf_{r}(\check
    X) \to \bonf(\check X)$. Let $[E_{r}, A_{r}]$ be an element of
    $\bonf_{r}(\check X)$, and let $d_{r}$ be the formal dimension of
    the moduli space for this component of $\bonf_{r}(\check X)$: that
    is, the index of the orbifold operator $d^{*}  + d^{+}$ coupled to
    $(E_{r}, A_{r})$. From \cite{KM-Tait}, we have the formula
    \[
            d_{r} = 8 \kappa_{r}  - 3(b^{+}-b^{1} + 1) +
            (1/2)\Sigma\ccdot \Sigma + \chi(\Sigma) - t/2.
    \]
    In particular, the right-hand side is an integer. Now let $[E,A]$
    be the image of $[E_{r}, A_{r}]$ in $\bonf(\check X)$ and let
    $\kappa$ be its topological energy. By
    Lemma~\ref{lem:times-4}, we can express $8 \kappa_{r}$ as
    $2\kappa$, so the fact that $d_{r}$ is an integer tells us 
    \[
          12 \kappa = -6\Bigl(  - 3(b^{+}-b^{1} + 1) +
            (1/2)\Sigma\ccdot \Sigma + \chi(\Sigma) - t/2\Bigr) , \pmod
            {6\Z},
    \]
    or just,
      \[
          12 \kappa =  
           -3\, \Sigma\ccdot \Sigma  + 3 t, \pmod
            {6\Z}.
    \]
    If we substitute this formula for $12\kappa$ into the dimension
    formula in Proposition~\ref{prop:dimension} and simplify the
    result modulo $6$, we obtain the result stated in the proposition:
    \[
    \begin{aligned}
    \dim M_{\kappa} &= 12\kappa   - 8 (b^{+}(X) - b^{1}(X)+1) +
            \Sigma\ccdot\Sigma + 
            2 \chi(\Sigma) - t \\
                   &=   (-3 \,\Sigma\ccdot \Sigma  + 3 t)
                - 8 (b^{+}(X) - b^{1}(X)+1)    + \Sigma\ccdot\Sigma + 2
            \chi(\Sigma) - t, \pmod {6}\\
            &=  - 2 (b^{+}(X) - b^{1}(X)+1) - 2 \, \Sigma\ccdot\Sigma
            + 2 \chi(\Sigma) + 2t , \pmod {6}.
            \end{aligned}
    \]
\end{proof}

\subsection{Uhlenbeck compactness}

Let $\check X$ be a 4-dimensional bifold (possibly non-compact) and
let $A_{n}$ be a sequence of anti-self-dual connections in a bifold
bundle $E\to \check X$. We have the usual statement of Uhlenbeck's
compactness theorem in this setting: provided only that there is a
uniform bound on the Chern-Weil integrals $\kappa(A_{n})$, there is a
subsequence $A_{n'}$ which, after applying gauge transformations,
converges on compact subsets of $\check X\setminus \{x_{1}, \dots,
x_{l}\}$ to a connection $A^{*}$ in a bifold bundle $E^{*}$.
Furthermore, the limit $(E^{*},A^{*})$ has removable singularities
after gauge transformation, at each of the points $x_{i}$.

The extra information we need to add to this general form of the
compactness theorem is a statement of how much action $\kappa$ is lost
in the bubbles at the points $x_{i}$. That is, if $U_{i}$ is
neighborhood of $x_{i}$ whose closure is compact and disjoint from the
other $x_{j}$, we must consider the difference
\[
    \delta_{i}=\lim \kappa(A_{n'}|_{U_{i}}) - \kappa(A^{*}|_{U_{i}}).
\]

\begin{lemma}
    In the above situation, the loss of action $\delta_{i}$ in the
    bubble $x_{i}$ is a multiple of $1/2$ if $x_{i}$ belongs to the
    singular set of the bifold, and is an integer otherwise.
\end{lemma}

\begin{proof}
    This follows from the results of section~\ref{subsec:classify}.
\end{proof}

\begin{corollary}
    We have an inequality
    \[
            \kappa(A^{*}) \le \limsup \kappa_{A_{n'}} - \sum_{1}^{l}
            \delta_{i}
    \]
    where $\delta_{i}$ is a non-negative multiple of $1/2$ or of $1$,
    according as $x_{i}$ is or is not on the singular set of the
    bifold. In case $\check X$ is compact, this is an equality.
\end{corollary}

\begin{remark}
    Note that the result here is different from the results in the
    case of $\SO(3)$ bifold bundles \cite{KM-Tait}. In the $\SO(3)$
    case, the value of the $\SO(3)$ instanton number $\kappa_{r}$ can
    drop by $1/4$ or $1/8$ when bubbles occur at points $x_{i}$ on the
    seam of the foam or a tetrahedral point respectively. Note that
    this is consistent with the inclusion $\r$ of the $\SO(3)$
    solutions in the space of $\SU(3)$ solutions, because of the
    factor of 4 in Lemma~\ref{lem:times-4}.
\end{remark}

\section{Instanton homology for bifold connections}
\label{sec:floer}

\subsection{Atoms and a trifold}
\label{subsec:atoms}

The space of bifold connections $\bonf(\check Y)$ on a
three-dimensional bifold $\check Y$ may contain flat \emph{reducible}
connections, which obstruct a straightforward construction of
instanton homology for $\check Y$. To remedy this situation, as in
similar situations in \cite{KM-yaft,KM-unknot,KM-Tait}, we modify the
construction by forming a connected sum of the bifold $\check Y$ with
another orbifold (an \emph{atom}) on which there are no reducible
connections. It is convenient further if the atom can be chosen so
that there is exactly one (irreducible) flat connection modulo gauge.

In the current context, there are several possible choices for what to
use as the atom, but we choose one which is convenient for our
purposes. Our particularly choice requires modifying our framework
slightly in two ways.

First, the orbifold we choose for the atom is not a bifold, because
the local stabilizers will have order $3$, not order $2$. Consider
$S^{3}$ with coordinates $(z_{1},z_{2})$ as the unit sphere in
$\C^{2}$ and let the group $V_{9}=C_{3} \times C_{3}$ act on $S^{3}$
by multiplying $z_{1}$ and $z_{2}$ by cube roots of unity. The
quotient orbifold is topologically $S^{3}$ and the singular locus is a
pair of circles forming a Hopf link in $S^{3}$. Since the stabilizers
have order $3$, we refer to $S^{3}/V_{9}$ as a \emph{trifold}. For
future reference, we give this trifold a name and write
\[
       H_{3} = S^{3}/V_{9}.
\]

The second modification is that our orbifold bundle will not be quite
an orbifold $\SU(3)$ bundle over $H_{3}$. We start instead by
describing an orbifold $\PSU(3)$ bundle. Let $\Gg$, and $\Gh$ denote
the standard generators of $C_{3}\times C_{3}$. There is a central
extension,
\[
          \Z \to G_{\Z} \to V_{9}
\]
described by generators
$\Gg, \Gh, \gamma$, with $\gamma$ central,
$[\Gg, \Gh]=\gamma$ and
$\Gg^{3}=\Gh^{3}=1$. There is a
representation $\rho:G_{\Z}\to \SU(3)$ given on generators by
\begin{equation}\label{eq:rep-rho}
     \rho:\Gg\mapsto \begin{bmatrix}
                           \omega &  0  & 0\\
                            0 &  1 & 0 \\
                            0 & 0 & \omega^{-1}
                            \end{bmatrix} \quad
     \rho:\Gh\mapsto \begin{bmatrix}
                           0 &  1  & 0\\
                            0 & 0  & 1 \\
                            1 &  0 & 0
                            \end{bmatrix}\quad
         \rho: \gamma \mapsto \begin{bmatrix}
                           \omega &  0  & 0\\
                            0 & \omega  & 0 \\
                            0 &  0 & \omega
                            \end{bmatrix}   
\end{equation}
where $\omega=e^{2\pi i /3}$. Since $\gamma$ is central, this
representation descends to a homomorphism $V_{9} \to \PSU(3)$. The
quotient of the trivial b$\PSU(3)$-bundle $\PSU(3)\times S^{3}$ by
$V_{9}$ is an orbifold principal bundle over the trifold $H_{3}$. It
comes with a flat orbifold connection from the trivial connection over
$S^{3}$.

Now let $w$ be an arc in $H_{3}$ joining the two circles of the Hopf
link. For example we can take the path $(z_{1}, z_{2})= (\cos(\theta),
\sin(\theta))$ for $\theta$ in $[0,\pi/2]$. The complement of
$w\subset H_{3}$ has orbifold fundamental group isomorphic to $G_{\Z}$
in such a way that the map $\pi_{1}(H_{3}\setminus w )\to
\pi_{1}(H_{3})$ is the quotient map $G\to V_{9}$, so our flat
$\PSU(3)$ bundle over $H_{3}$ lifts to a flat $\SU(3)$ bundle on the
complement of $w$. This flat connection has monodromy $\rho(\gamma) =
\omega \mathbf{1}$ on the link of the arc $w$. Since $\rho(\gamma)$
has order $3$, the monodromy of this flat connection is the subgroup
$G\subset \SU(3)$ which is a central extension
\begin{equation}\label{eq:central-extension-3}
          \Z/3 \to G \to V_{9}
\end{equation}
whose center is generated by an element $\bar\gamma = [\Gg,
\Gh]$ of order $3$.

We can now consider the space $\bonf_{l}^{w}(H_{3})$, elements of
which consist of data of the following sort:
\begin{itemize}
    \item  an orbifold $\PSU(3)$ connection $(E',A')$ on $H_{3}$ of
    Sobolev class $L^{2}_{l}$;
   \item a lift of $(E,A)$ to an $\SU(3)$ orbifold connection on the
   complement of the arc $w$;
   \item local models for $(E,A)$ at orbifold points of
   $H_{3}\setminus w$ are required to have the local stabilizer
   $C_{3}$ acting with three distinct eigenvalues $\{\omega, 1,
    \omega^{-1}\}$ on the fibers $\tilde E_{x}$ in the orbifold
    charts;
    \item and the asymptotic monodromy of $A$ around the
    oriented link of the arc $w$ is $\omega \mathbf{1}$. 
\end{itemize}

\begin{lemma}\label{lem:atomic}
    In the configuration space $\bonf_{l}^{w}(H_{3})$, there is a
    unique flat connection $(E,A)$. It is non-degenerate and it has
    trivial automorphism group.
\end{lemma}

\begin{proof}
     The complement of the singular set in $H_{3}$ is the product of a
     2-torus with an open interval, and $w$ meets the torus in a
     point. On the torus, our $\PSU(3)$ bundle is familiar as the
     adjoint bundle of the unique projectively flat $U(3)$ connection
     with degree $1$ on $T^{2}$. See, for example, \cite{KM-yaft}.
 \end{proof}

\begin{remarks}
(1) There is a completely analogous construction of an orbifold
$\PSU(N)$ bundle over an orbifold $H_{N} = S^{3}/(C_{N} \times
C_{N})$, having an $\SU(N)$ lift on the complement of an arc joining
the two circles in the singular set. The case $N=2$ appears in
\cite{KM-unknot}.

(2) As in \cite{KM-unknot}, one can consider a more general cobordism
category whose objects are orbifolds with bifold and trifold
singularities, adorned with arcs and circles $w$. The locus $w$ is
allowed to have endpoints on the trifold loci, but not on the bifold
loci. Supplying $w$ is equivalent to supplying a line bundle $\Lambda$
with $c_{1}(\Lambda)$ dual to $w$, and we can regard the construction
as studying connections in the adjoint bundle of a $U(3)$ bundle with
determinant $\Lambda$. From this point of view, taking an explicit
representative $w$ aids in the discussion of functoriality. We will
not need to consider this general situation further in this paper,
because the trifold locus and the arc $w$ are appear only in the atom.
\end{remarks}

Given now a 3-dimensional \emph{bifold} $\check Y$, we form the
connected sum $\check Y\# H_{3}$ at a smooth point (not on the
orbifold locus). To do this canonically, we need to choose a basepoint
$y_{0} \in Y$, and a framing of $T_{y_{0}}Y$. We also need a one-off
choice of basepoint in $H_{3}$, which we ask not to be on $w$. On the
connected sum $\check Y\# (H_{3} \setminus w)$ we then consider
orbifold bundles obtained by summing a bifold bundle on $\check Y$ to
our model $\SU(3)$ trifold bundle on $H_{3}\setminus w$. We then have
a space of Sobolev connections $\bonf_{l}^{w}(\check Y \# H_{3})$,
with the same local models as before. Note that the $\SU(3)$ orbifold
connection is still defined only on the complement of $w$.

We introduce the abbreviations
\[
            \check Y^{\sharp} = \check Y \# H_{3},
\]
and
\[
        \bonf^{\sharp}(\check Y^{\sharp}) = \bonf_{l}^{w}(\check Y \#
        H_{3}).
\]

\begin{lemma}
    This space of bifold connections contains no reducible flat
    connections.
\end{lemma}

\begin{proof}
   This follows from Lemma~\ref{lem:atomic}, because a flat connection
   must already be reducible on the atom $H_{3}$.
\end{proof}

\subsection{Defining instanton homology}

We now have what we need to set up the definition of the $\SU(3)$
instanton homology $\Lsharp$. The models for this construction are in
\cite{KM-yaft} and \cite{KM-Tait}. The first of those papers deals
with structure group $\SU(N)$ in general, but the singular locus there
was always a submanifold, not a web or foam. The generalization to
webs and foams was done in \cite{KM-Tait}, though only for $\SO(3)$.

In outline, given an oriented closed bifold $\check Y$ of dimension
$3$, we form the connected sum with the atom $H_{3}$ at a basepoint,
and consider the Chern-Simons functional on the space
$\bonf^{\sharp}(\check Y^{\sharp})$, perturbed by a holonomy
perturbation so that the critical points are non-degenerate and the
intersection of the unstable and stable manifolds of critical points
$\alpha$ and $\beta$ (the trajectory spaces $M(\alpha,\beta)$) are
transverse. We then define an $\F$-vector space $CL^{\sharp}(\check
Y)$ whose basis is the set of critical points, and define a
differential $\partial$ on $CL^{\sharp}(\check Y)$ whose matrix
entries count the number of components in trajectory spaces
$M(\alpha,\beta)$ of dimension $1$. The instanton homology
$\Lsharp(\check Y)$ is the homology of this complex.

In order for the moduli spaces used in the construction of the
differential to have the necessary compactness properties, a
monotonicity condition is needed. This is discussed in detail in
\cite{KM-yaft} for arbitrary compact structure group. Given critical
points $\alpha$ and $\beta$, the trajectory space $M(\alpha,\beta)$
has components of different dimensions, depending on the action
$\kappa$. If we write $M_{\kappa}(\alpha,\beta)$ for the components of
action $\kappa$, then the monotonicity condition states that the
dimension of a component of $M(\alpha,\beta)$ depends only on the
action, and is therefore constant on each $M_{\kappa}(\alpha,\beta)$.
The monotonicity condition is a constraint on the admissible local
models for orbifold connections in general. In the case of $\SU(3)$,
it can be stated as the condition
\[
        \dim M_{\kappa}(\alpha, \beta) - \dim M_{\kappa'}(\alpha,
        \beta) = 12(\kappa-\kappa').
\]
Connections in these two moduli spaces differ topologically by gluing
in instantons and monopoles, and in the case that the this happens on
the bifold locus of $H_{3}$, the required monotonicity is a
consequence of Proposition~\ref{prop:dimension}. For the case of
gluing monopoles on the trifold locus in $H_{3}$, the monotonicity
condition holds as an a particular case of the classification in
\cite{KM-yaft}. (See \cite[section 2.5]{KM-yaft} for the case of the
special unitary groups.)

\begin{remark}
    The case of a bifold singularity where the action the order-2
    local stabilizer on the $\SU(3)$ fiber is $\mathrm{diag}(1,-1,-1)$
    does not appear in \cite{KM-yaft}. A closely related case does
    appear, and this is the case that asymptotic monodromy is given by
    $\mathrm{diag}(e^{2\pi i /3},e^{-\pi i /3}, e^{-\pi i /6} )$,
    which differs from the $\mathrm{diag}(1,-1,-1)$ by an element of
    the center of $\SU(3)$. Since they are the same in the adjoint
    group, the local analysis of these two cases are essentially the
    same. In the setting of that \cite{KM-yaft}, in the bifold case,
    the element $\mathrm{diag}(-1,1,-1)$ would be written as
    \[
         \exp 2\pi  i \,\mathrm{diag}(\lambda_{1}, \lambda_{2},
         \lambda_{3}),
    \]
    where $(\lambda_{1}, \lambda_{2}, \lambda_{3}) = (1/2, 0, -1/2)$.
    These eigenvalues do not satisfy a constraint which is required in
    the setting of \cite{KM-yaft}, namely that the eigenvalues
    $\lambda_{k}$ lie in an interval of length strictly less than $1$.
    We have an interval of length $1$ exactly, which can be
    interpreted as placing this diagonal matrix on the far wall of the
    Weyl alcove (see \cite[section 2.7]{KM-yaft}).
\end{remark}

In general, two trajectory spaces $ M_{\kappa}(\alpha, \beta)$ and
$M_{\kappa'}(\alpha, \beta)$ will have action $\kappa-\kappa'$ which
is a multiple of $1/6$, because gluing monopoles on the bi- and
trifold loci contribute multiples of $1/2$ and $1/3$ respectively. The
dimensions of the trajectory spaces therefore differ by a multiple of
$2$.

\begin{corollary}
    The complex defining the $\SU(3)$ instanton homology
    $\Lsharp(\check Y)$
    has a relative $\Z/2$ grading. \qed
\end{corollary}

\begin{remark}
    In section~\ref{sec:absolute} we will examine what choices needs
    to be made to specify an absolute $\Z/2$ grading, at least for
    webs in $\R^{3}$. 
\end{remark}

The following is a consequence of Lemma~\ref{lem:atomic} and the
definitions.

\begin{lemma}\label{lem:S3}
    For the $3$-sphere (as a bifold whose singular locus is empty), we
    have $\Lsharp(S^{3})=\F$. \qed
\end{lemma}

As usual, we often regard $\Lsharp$ as an invariant of knots, links
and webs in $\R^{3}$:

\begin{notation}\label{notn:webs}
    If $K\subset\R^{3}$ is a spatial web, we write $\Lsharp(K)$ for
    the the $\SU(3)$ bifold Floer homology $\Lsharp(S^{3}, K)$, where
    $S^{3}$ is taken to have its framed basepoint at infinity.
\end{notation}
\subsection{Functoriality}

The extension of the definition of $\Lsharp(\check Y)$ to a functor on
a suitable cobordism category of webs and foams is now quite standard.
Just as we require a framed basepoint $y_{0}\in \check Y$ at which to
form the connected sum $\check Y \# H_{3}$ with the atom, so our
cobordisms are required to be oriented 4-dimensional bifolds equipped
a framed arc joining the basepoints at the two ends. (See
\cite{KM-yaft} again, or \cite{KM-unknot} for example.) In this way, a
bifold cobordism $\check X$ from $\check Y_{0}$ to $\check Y_{1}$,
equipped with such an arc, gives rise to an orbifold cobordism $\check
X^{\sharp}$ from $\check Y_{0}^{\sharp}$ to $\check Y_{1}^{\sharp}$.
Now attach cylindrical ends to $\check X^{\sharp}$, choose a
cylindrical-end metric, and a holonomy perturbation $\pi$ for the
anti-self-duality equations, as in \cite{KM-yaft, K-higherrank}. Given
non-degenerate critical points for the perturbed Chern-Simons
functional $\alpha_{0}$, $\alpha_{1}$ on $\check Y^{\sharp}_{0}$ and
$\check Y^{\sharp}_{1}$, we have moduli spaces of perturbed ASD
connections,
\begin{equation}\label{eq:moduli-space-X}
                 M(\alpha_{0}, \check X^{\sharp}, \alpha_{1})
\end{equation}
as in \cite{KM-yaft}. They are cut out transversely by the equations
for generic choice of perturbation $\pi$.
There is then a linear
map
\[
        \Lsharp (\check X) : \Lsharp ( \check
Y_{0})\to  \Lsharp (\check Y_{1}),
\]
 defined by counting solutions of the perturbed equations in
 zero-dimensional components of these moduli spaces.

With this construction understood, we obtain a functor $\Lsharp$ to
the category of $\F$-vector spaces, whose source category has objects
the closed, connected, 3-dimensional bifolds with framed basepoint and
whose morphisms are isomorphism classes of connected, 4-dimensional
bifold cobordisms containing framed arcs connecting the basepoints on
the two boundary components. We refer to this category sometimes as
$\Cat_{0}^{\sharp}$, following \cite{KM-Tait}, so we have a functor,
\[
       \Lsharp : \Cat_{0}^{\sharp} \to \mathrm{Vect}(\F).
\]

\begin{notation} \label{not:rel-invariant}
The empty 3-dimensional bifold is not an object in the category
$\Cat_{0}^{\sharp}$, so if $\check X$ is a 4-dimensional bifold with a
single oriented boundary component $\check Y$, then it does not
directly define a morphism. However, given a framed basepoint $y_{0}$
in the non-singular part of $\check Y$, we may remove a ball from a
collar neighborhood of $\check Y$, adjacent to $y_{0}$, to obtain a
cobordism $\check X'$ from $S^{3}$ to $\check Y$. Join $y_{0}$ to a
point on the 3-sphere by a standard framed arc in the collar, and
$\check X'$ becomes a morphism in $\Cat^{\sharp}_{0}$ from $S^{3}$ to
$\check Y$. As notation, we allow ourselves to define $\Lsharp(\check
X)$ as
\begin{equation}\label{eq:rel-invariant}
         \begin{aligned}
            \Lsharp(\check X) &= \Lsharp(\check X')(1) \\
                              &\in \Lsharp(\check Y),
         \end{aligned}
\end{equation}
where the element $1$ on the right is the generator of
$\Lsharp(S^{3})=\F$ (Lemma~\ref{lem:S3}).
\end{notation}

\subsection{Extending functoriality with dots}

We can extend the category $\Cat_{0}^{\sharp}$ to a category
$\Cat^{\sharp}$ by decorating our bifold cobordisms $\check X$ with
``dots''. We outline the construction in this subsection, following
standard models.

Given a 4-dimensional orbifold $\check X$, let $\bonf^{*}(\check X)
\subset \bonf(\check X)$ be the space of irreducible $\SU(N)$ orbifold
bundles modulo equivalence, of Sobolev class $L^{2}_{l}$ for suitable
$l$, with any specified local models at the orbifold points. There is
a universal $\PSU(N)$ bundle $\Ad\PP \to \bonf^{*}(\check X)\times
\check X$ (which may or may not lift to an $\SU(N)$ bundle $\PP \to
\bonf^{*}(\check X)$).

Given a point $x\in \check X$, we obtain by restriction a $\PSU(N)$
bundle $\Ad\PP_{x}\to\bonf^{*}(\check X)$. If $c \in
H^{*}(B\PSU(N);R)$ is a characteristic class of $\PSU(N)$ bundles,
then we obtain a cohomology class $c(\Ad\PP_{x}) \in
H^{*}(\bonf^{*}(\check X)$. If $x$ lies on the orbifold locus of
$\check X$, then the structure group of $\Ad\PP_{x}$ is reduced. This
is because we can identify $\Ad\PP_{x}$ as the basepoint bundle coming
from the basepoint $\tilde x$ above $x$ in the orbifold chart, on
which the local stabilizer group $H_{x}$ acts, so there is a reduction
of structure group from $\PSU(N)$ to the subgroup which is the
centralizer of this representation of $H_{x}$. In this situation, we
can use for $c$ a characteristic class of this subgroup. For our
bifold $\SU(3)$ case, if $x$ lies in a facet of the orbifold locus
where $H_{x}$ has order $2$, the corresponding reduction is to the
subgroup
\begin{equation}\label{eq:Q-reduction}
         Q =  P(S(U(1)\times U(2)))\subset \PSU(3),
\end{equation}
which is a group isomorphic to $U(2)/(\Z/3)$. (We use $P$ in this
context to mean the quotient by the center of $\SU(3)$ and $S$ to
denote the elements of determinant $1$.) If $x$ lies on a seam of the
bifold, then the reduction is to
\[
         P(\mathbb{T})\subset \PSU(3),
\]
where $\mathbb{T}$ is the maximal torus of $\SU(3)$.

Because we are working with mod $2$ coefficients, we are interested
primarily in characteristic classes $c$ with coefficients in $\F$, the
field of $2$ elements. The Chern classes of $\SU(N)$ bundles with mod
$2$ coefficients are pulled back from characteristic classes of
$\PSU(N)$ bundles when $N$ is odd, because the fiber of the map
$B\SU(N)\to B\PSU(N)$ is $B(\Z/N)$, which has trivial mod 2
cohomology. The classes that concern us are, when $x$ is not on the
orbifold locus, the mod $2$ Chern class
\begin{equation}\label{eq:point-class-c2}
           c_{2}(\Ad \PP_{x}) \in H^{4}(\bonf^{*}(\check X) ; \F),
\end{equation}
and when $x$ is on a facet, the classes
\begin{equation}\label{eq:bifold-class-ci}
           c_{i}(Q_{x}) \in H^{2i}(\bonf^{*}(\check X) ; \F)
\end{equation}
for $q=1,2$. (Here $Q_{x}$ is the reduction of $\Ad \PP_{x}$ to the
subgroup $Q \cong U(2)/(\Z/3)$, and $c_{i}$ are the mod $2$ Chern
classes of $U(2)$ bundles, regarded as pulled back from $BQ$.)

More concrete descriptions of the 4-dimensional characteristic classes
can be given as follows. Given a principal $\PSU(3)$ bundle $\Ad\PP$,
let $V$ be the associated real vector bundle with fiber $\su(3)$. The
class \eqref{eq:point-class-c2} can then be interpreted via the
equality
\[
            c_{2}(\Ad \PP_{x}) = w_{4}(V)\in H^{4}(\bonf^{*}(\check X) ; \F),
\]
which can be verified using the splitting principal. If $x$ is a
bifold point, then the action of the element of order $2$ in $H_{x}$
decomposes the adjoint bundle into the $\pm 1$ eigenspaces,
\[
                V = V_{+} \oplus V_{-}
\]
where $V_{+}$ is the bundle of Lie algebras $\u(2)$ associated to the
reduction $Q_{x}$ and $V_{-}$ is its complement. We can then interpret
the mod 2 Chern classes \eqref{eq:bifold-class-ci} as
\[
\begin{aligned}
c_{1}(Q_{x}) &= w_{2}(V_{-})\in H^{2}(\bonf^{*}(\check X) ; \F),\\
c_{2}(Q_{x}) &= w_{4}(V_{-}) +w_{2}(V_{-})^{2} \in
H^{4}(\bonf^{*}(\check X) ; \F).
\end{aligned}
\]
We give these classes names,
\[
\begin{aligned}
    \nu=\nu(x)  &= c_{2}(\Ad \PP_{x}) \\
    \sigma_{1}(x) &= c_{1}(Q_{x}) \\
    \sigma_{2}(x) &= c_{2}(Q_{x})
\end{aligned}
\]
in mod $2$ cohomology.

As usual, if $U\subset \check X$ is an open set containing $x$, and if
\begin{equation}\label{eq:restriction-r}
           r : \bonf^{**}(\check X) \to \bonf^{*}(U)
\end{equation}
is the restriction map from the set $\bonf^{**}(\check X)$ of
connections whose restriction is irreducible, then the classes
$\nu(x)$ or $\sigma_{i}(x)$ on $\bonf^{**}(\check X)$ are pulled back
from $\bonf^{*}(U)$.

We apply these constructions after summing with the atom $H_{3}$, so
we consider the cobordism $\check X^{\sharp}$ from $\check
Y_{0}^{\sharp}$ to $\check Y_{1}^{\sharp}$ and the moduli spaces
\eqref{eq:moduli-space-X} on the cylindrical-end manifold. Given
points $x_{1},\dots, x_{m}$ in $\check X$, away from the arc of
basepoints where the atom is attached, and given any bound $\Delta$,
we can choose disjoint neighborhoods $U_{1}$, \dots, $U_{m}$ of these
points such that all of the components of the moduli spaces
\eqref{eq:moduli-space-X} of dimension at most $\Delta$ are in the
domain $\bonf^{**}$ of the restriction maps \eqref{eq:restriction-r}
to each of the $U_{k}$. For each $x_{k}$, let there be given one of
the above mod 2 classes, $\mu_{k}$, equal to either $\nu(x_{k})$ or
$\sigma_{i}(x_{k})$. We take closed subsets
$\cV_{k}\subset\bonf^{*}(U_{k})$ stratified by submanifolds of the
Hilbert manifold, of finite codimension, and representing the dual of
the classes $\mu_{i}$. We have two requirements of these:
\begin{enumerate}
\item
    We wish all strata in the intersections $\cV_{k_{1}}\cap \dots \cap
    \cV_{k_{p}}$ to be transverse, and transverse to the restriction
    map $r$ from the moduli spaces
    $M(\alpha_{0},\check X^{\sharp}, \alpha_{1})$.
\item
    We need the above intersections to be closed under suitable
    limits, as arise in the Uhlenbeck compactness theorem.
\end{enumerate}

To elaborate on the second condition, consider a sequence of solutions
\[
          [E_{i}, A_{i}] \in M(\alpha_{0},\check X^{\sharp}, \alpha_{1})
\]
in which bubbling occurs. This means that there are finitely many
points $b_{j}$ such that the connections converge, after gauge
transformation, on compact subsets of the cylindrical end manifold
disjoint from $\{b_{j}\}$. If none of the bubble points $b_{j}$ lie in
the neighborhood $U_{k}$, and if all $[E_{i}, A_{i}]$ belong to
$\cV_{k}$ (on restriction to $U_{k}$), then we require the weak limit
$[E,A]$ also to belong to $\cV_{k}$. In the presence of holonomy
perturbations (whose effects are not local), the convergence can only
be assumed to be in the topology of $L^{p}_{1}$ connections, for all
$p$ \cite{KM-yaft}. So to achieve the second condition, we require
that $\cV_{k}$ be closed in the $L^{p}_{1}$ topology for some $p$.

To achieve the first condition, we simply need a sufficiently large
supply of sections of the vector bundles associated to the principal
bundle $\Ad \PP_{x}$ or $Q_{x}$ over $\bonf^{*}(U)$. These are
constructed in the standard way using local trivializations and
cut-off functions. Our definition of $\bonf^{*}(U)$ used to Sobolev
class $L^{2}_{l}$ for expediency, but there is a smooth map of Banach
manifolds $\bonf^{*}(U) \to \bonf^{*}_{p,1}(U)$ where the latter is
defined using $L^{p}_{1}$ connections for $p > 2$. If $p$ is even,
then radial cut-off functions defined using the $L^{p}_{1}$ norm are
smooth, so we can construct our stratified subsets $\cV$ in
$\bonf^{*}_{p,1}(U)$ and then pull back to $\bonf^{*}(U)$, allowing us
to fulfill both of the above requirements.

We are now able to define the decorated category $\Cat^{\sharp}$ and the
extension of the functor $\Lsharp$. A morphism in $\Cat^{\sharp}$ will be
a morphism in $\Cat_{0}^{\sharp}$ (an oriented bifold cobordism together
with an arc joining the basepoints), enriched with a finite collection
of points $x_{1}, \dots x_{m}$ and for each point a choice of
corresponding mod 2 cohomology class: either $\nu(x_{k})$ if $x_{k}$
is not on the singular set, or $\sigma_{i}(x_{k})$ for $i=1$ or $i=2$
if $x_{k}$ belongs to a facet of the foam. The distinguished points
are required to be disjoint from the arc. Writing $\mu_{k}$ as a
generic symbol for either $\nu(x_{k})$ or $\sigma_{i}(x_{k})$, we will
write our morphism as
\[
             \bX = (\check X, \mu_{1}, \dots, \mu_{m}).
\]

To extend the definition of $\Lsharp$ to such morphisms, we choose
representatives $\cV_{k}$ for the classes $\mu_{k}$ satisfying the
transversality and compactness requirements above, and define the
matrix entries of
\begin{equation}\label{eq:Lsharp-deco}
 \Lsharp(\bX) : \Lsharp(\check Y_{0}) \to \Lsharp(\check Y_{1})
\end{equation}
at the chain level by counting
elements of the zero-dimensional components of the transverse
intersection
\[
  M(\alpha_{0}, \check X^{\sharp}, \alpha_{1}) \cap \cV_{1}\cap \dots \cap
    \cV_{m}
\] 
(where the restriction maps are implied).

As a standard special case, we can consider the case that $\check
X^{\sharp} = \R\times \check Y^{\sharp}$. In this case, the same
construction gives us operators on $\Lsharp(Y)$: we have an operator
\begin{equation}\label{eq:nu-op}
             \nu : \Lsharp(\check Y) \to \Lsharp(\check Y)
\end{equation}
corresponding to $\nu(x)$ for $x$ not on the singular set, and we
have, for each edge $e$ of the web $K\subset \check Y$, operators
\begin{equation}\label{eq:sigma-ops}
        \sigma_{1}(e), \sigma_{2}(e):
            \Lsharp(\check Y) \to \Lsharp(\check Y)
\end{equation}
corresponding to any chosen points on the facet $\R\times e$ of the
product foam. These operators commute.

\subsection{The excision property}

We will use an excision property of our $\SU(3)$ instanton homology
groups for webs and foams. In the $\SU(2)$ or $\SO(3)$ case, this is
essentially Floer's excision theorem \cite{Braam-Donaldson}, and was
applied to $\Isharp$ in \cite{KM-unknot}. (See \cite[Corollary
5.9]{KM-unknot} for the closest parallel to the version stated here.)
The proof adapts to $\SU(3)$ or $\SU(N)$, as discussed in
\cite{Daemi-Xie}.

In our context, let $\check Y_{1}$ and $\check Y_{2}$ be two
3-dimensional bifolds, and let $\check Y$ be their connected sum,
formed at a non-singular point. There are standard bifold cobordisms
from $\check Y_{1}^{\sharp} \sqcup \check Y_{2}^{\sharp}$ to $\check
Y^{\sharp} \sqcup (S^{3})^{\sharp}$, and vice versa. We have:

\begin{proposition}\label{prop:excision}
    The excision cobordisms give mutually inverse isomorphisms,
    \[
    \begin{aligned}
         \Lsharp(\check Y_{1})\otimes \Lsharp(\check Y_{2}) \to\null
          &\Lsharp(\check Y)\otimes \Lsharp(S^{3}) \\
          =\null  &\Lsharp(\check Y),
    \end{aligned}
    \]
    where the last equality results from Lemma~\ref{lem:S3}. These
    isomorphisms are natural in the following sense. Given morphisms
    $\bX_{1}$ and $\bX_{2}$ in $\Cat^{\sharp}$ and the morphism $\bX$
    obtained by summing the two manifolds along the basepoint arcs,
    the excision isomorphism intertwines $\Lsharp(\bX_{1}) \otimes
    \Lsharp(\bX_{1})$ with $\Lsharp(\bX)$. \qed
\end{proposition}

The excision property is often used to understand t he relationship
between morphisms $\Lsharp(\bX)$ and $\Lsharp(\bX')$ when $\bX'$ is
obtained from $\bX$ by removing a closed subset of the interior and
replacing it with something different. A general version is the
following. We consider morphisms
\[
        \bX_{1}, \dots , \bX_{r}
\]
in $\Cat^{\sharp}$ from $\check Y_{0}$ to $\check Y_{1}$. We suppose
that these have the form
\[
                 \bX_{i} = \bX' \cup_{\check Q} \mathbf{P}_{i},
\]
where each $\mathbf{P}_{i}$ is a decorated bifold with boundary
$\check Q$, and $\bX'$ is a decorated bifold cobordism from $\check
Y_{0}$ to $\check Y_{1}$ having an additional boundary component
$-\check Q$.

\begin{proposition}\label{prop:excision-local}
    Let $\alpha_{i}\in \Lsharp(\check Q)$ be the element defined by
    $\mathbf{P}_{i}$, and suppose that
     \[
               \sum_{i=1}^{r} \alpha_{i} = 0 \in \Lsharp(\check Q).
    \]
    Then
    \[
                \sum_{i=1}^{r}\Lsharp(\bX_{i})=0
    \]
    as linear maps from $\Lsharp(\check Y_{0})$ to $\Lsharp(\check Y_{1})$.
\end{proposition}

\subsection{Dot relations}

We collect here the relations which the various point-classes satisfy.
As a general principle, a relation between cohomology classes
$\mu(x_{i})$ in $H^{d}(\bonf(\check X); \F)$ gives rise directly to a
relation between the corresponding operators on $\Lsharp$, provided
that the dimensions of the moduli spaces which are involved are small
enough that bubbling does not interfere with the compactness
arguments. In practice, this applies when $d\le 4$, because bubbles
will occur in open sets containing $x_{i}$ for moduli spaces of
dimension $6$ or more. When $d$ is larger, an analysis of the
contributions from bubbling is required.

\paragraph{The 4-dimensional point class.}

We begin with with the class $\nu= c_{2}(\Ad \PP_{x})$ for $x$ in the
non-singular locus of the bifold.

\begin{lemma}\label{lem:nu-is-1}
    For any $\check Y$, we have $\nu=1$ for the operator
    \eqref{eq:nu-op} on $\Lsharp(\check Y)$.
\end{lemma}

\begin{proof}
    With rational coefficients and a different atom, the relation
    $\nu=3$ was proved by Xie in \cite{Xie}, and the lemma can be
    deduced from that result by using and excision argument to show
    independence of the choice of atom. (A reminder here that our
    coefficients are mod 2, so $3=1$.) Alternatively, and more
    directly, the excision theorem in the form of
    Proposition~\ref{prop:excision} with $\check Y_{2}=S^{3}$, shows
    that it is enough to verify the lemma for the case $\check
    Y=S^{3}$, which we will do in an appendix by examining the
    relevant moduli space directly using the ADHM construction.
\end{proof}

\begin{corollary}
    Let $\bX_{0}=(\check X, \mu_{1}, \dots,\mu_{m})$ be any morphism
    in $\Cat^{\sharp}$, and let \[\bX_{1}=(\check X, \mu_{1},\dots,
    \mu_{m}, \nu(x))\] be obtained from $\bX_{0}$ by adding a single
    class $\nu(x)$, where $x$ is a non-singular point of the bifold.
    Then $\Lsharp(\bX_{0}) = \Lsharp(\bX_{1})$. \qed
\end{corollary}

\paragraph{Relations for dot-migration.}

\begin{lemma}\label{lem:dot-migration}
    Let $e_{1}$, $e_{2}$, $e_{3}$ be three edges of the web $K\subset
    \check Y$ incident at a single vertex. (The edges need not be
    distinct.) Then the corresponding operators $\sigma_{1}(e_{i})$
    satisfy the following relations.
\begin{enumerate}
\item
    $\sigma_{1}(e_{1}) + \sigma_{1}(e_{2}) + \sigma_{1}(e_{3})=0$;
 \item
    $\sigma_{1}(e_{1}) \sigma_{1}(e_{2}) +
    \sigma_{1}(e_{2})\sigma_{1}(e_{3})
    +\sigma_{1}(e_{3})\sigma_{1}(e_{1})  = 1$;
    \item
    $\sigma_{1}(e_{1}) \sigma_{1}(e_{2})\sigma(e_{3})=0$.
\end{enumerate}
    
\end{lemma}

\begin{proof}
    At a vertex of the web, the $\SU(3)$ structure group is reduced to
    the maximal torus $S(U(1)\times U(1) \times U(1))$, and the three
    associated line bundles $L_{1}$, $L_{2}$, $L_{3}$ are the same
    line bundles associated to the three reductions to $S(U(1) \times
    U(2))$ at the three edges. So, for the ordinary cohomology classes
    $\sigma_{1}(e_{i}) = c_{1}(Q_{x_{i}})$ we have the relations
\begin{enumerate}
\item
    $\sigma_{1}(e_{1}) + \sigma_{1}(e_{2}) +
    \sigma_{1}(e_{3})=c_{1}(\Ad \PP_{x})$;
 \item
    $\sigma_{1}(e_{1}) \sigma_{1}(e_{2}) +
    \sigma_{1}(e_{2})\sigma_{1}(e_{3})
    +\sigma_{1}(e_{3})\sigma_{1}(e_{1})  = c_{2}(\Ad \PP_{x})$;
    \item
    $\sigma_{1}(e_{1}) \sigma_{1}(e_{2})\sigma(e_{3})=c_{3}(\Ad \PP_{x})$
\end{enumerate}
as classes in $H^{*}(\bonf^{\sharp}(\check Y); \F)$. As explained in
the remarks at the beginning of this section, the first two of these
relations in ordinary cohomology become relations for the operators
directly, because the cohomology classes here have degree $2$ and $4$.
Since $c_{2}(\Ad \PP_{x}) = \nu$ and $\nu=1$ by
Lemma~\ref{lem:nu-is-1}, this proves the first two formulae. A direct
approach to the third formula requires consideration of a moduli space
with non-compactness due to bubbles. We postpone the proof of this
last formula until after the proof of
Proposition~\ref{prop:theta-evaluation}, where an indirect argument is
given.
\end{proof}

\paragraph{The Xie relation.}

The following relation is central to the proof of
Theorem~\ref{thm:planar-tait}. It should be contrasted to the
situation with the $\SO(3)$ homology, $\Jsharp$. For the latter, there
is an operator $u$ associated to each edge of web, and these operators
satisfy $u^{3}=0$ \cite{KM-Tait}. By introducing a deformation of
$\Jsharp$ using a system of local coefficients, this relation gets
altered and takes the form $u^{3} + Pu=0$ for a certain element $P$ in
the coefficient ring. (See \cite{KM-deformation}.) In the $\SU(3)$
homology $\Lsharp$, with coefficients $\F$, we have a parallel result.

\begin{lemma}\label{lem:Xie-rel-mod-2}
    For any bifold $\check Y$ and edge $e$ of the embedded web, we
    have
    \[
           \sigma_{1}(e)^{3} + \sigma_{1}(e) = 0,
    \]
    as operators on $\Lsharp(\check Y)$. 
\end{lemma}

\begin{proof}
    When the characteristic classes are interpreted in rational
    cohomology, a relation
    \begin{equation}\label{eq:Xie-relation}
    \sigma_{1}(e)^{3} + \nu(x)
    \sigma_{1}(e)=0
    \end{equation} is
    established in \cite{Xie} for the case that the foam is a smooth
    2-manifold without seams and the bifold $\check Y$ is
    ``admissible'' (without the need to introduce an atom). The proof
    adapts to coefficients mod 2 and is local, so the argument remains
    valid for foams in our present context. We indicate the argument,
    for future use.

    When an $\SU(3)$ bundle $P$ has a reduction of structure group to
    $S(U(1)\times U(2))$, the first Chern class $\sigma$ of the $U(1)$
    bundle satisfies a relation
    \begin{equation}\label{eq:general-Xie}
             \sigma^{3} - c_{1}(P) \sigma^{2} + c_{2}(P)\sigma -
             c_{3}(P)=0. 
    \end{equation}
    When these classes are interpreted as operators on the Floer
    homology of an admissible bifold, there is \emph{a priori} an
    additional term coming from a bubbling phenomenon, but in
    \cite{Xie} it is shown that this term is zero. (It comprises two
    canceling contributions $1$ and $-1$.) So the above relation holds
    for the operators $\sigma=\sigma_{1}(e)$ and the operators arising
    as $c_{i}(\Ad\PP_{x})$ for a point $x$ not in the singular set.

    The operator coming from $c_{2}(\Ad\PP_{x})$ is the operator $\nu$
    in Lemma~\ref{lem:nu-is-1}, which is $1$. The operators from
    $c_{1}$ and $c_{3}$ are both zero in our setting. One can see this
    (as in the previous lemma) by using excision to reduce to the case
    of $S^{3}$. In this special case, the bifold locus is empty and
    the Floer complex for $\Lsharp(S^{3})$ therefore has a relative
    mod-4 grading because the only monopole bubbles to consider are on
    the trifold locus of the atom $H_{3}$. Since $\Lsharp(S^{3})$ is
    non-zero in only one grading mod 4, the operators from $c_{1}$ and
    $c_{3}$, which have degree $2$ mod $4$, must be zero. So the
    relation \eqref{eq:general-Xie} reduces to the one in the lemma.
\end{proof}

The next lemma is similar but simpler.

\begin{lemma}
    For any bifold $\check Y$ and edge $e$ of the embedded web, we
    have
    \[
           \sigma_{2}(e)  = \sigma_{1}(e)^{2} + 1,
    \]
    as operators on $\Lsharp(\check Y)$. 
\end{lemma}

\begin{proof}
    As in the proof of the previous lemma, we consider the classes
    $c_{i}(P)$ for an $\SU(3)$ bundle with reduction to
    $Q=S(U(1)\times U(2)) \cong U(2)$, and the classes
    $\sigma_{1}=c_{1}(Q)$ and $\sigma_{2}=c_{2}(Q)$ which can be seen
    to satisfy the relation
    \[
            \sigma_{2} = \sigma_{1}^{2} + c_{2}(P)
    \]
    mod $2$. This continues to hold for the corresponding operators.
    (There is no bubbling to be considered for this relationship
    between 4-dimensional classes.) Since $c_{2}(P)=1$ again as
    operators on $\Lsharp$, this proves the lemma.
\end{proof}

Returning to the category $\Cat^{\sharp}$, we see from these lemmas that
the relations allow us to dispense with dots decorated by
$\sigma_{2}(x)$
or by $\nu(x)$. We can more simply consider a slight modification of
our definition of $\Cat^{\sharp}$ in which there is only one sort of dot,
always lying on facets, corresponding the classes $\sigma_{1}(x)$:

\begin{notation}\label{notation:dots}
    By a foam \emph{with dots} we shall mean, unless the context
    requires otherwise, a foam carrying dots $x$ on the facets, each
    label led with the class $\sigma_{1}(x)$.
\end{notation}

\section{Calculations for some closed foams}

\subsection{The setup for evaluation of closed foams}

If $\Sigma \subset \R^{4}$ is a closed foam decorated with dots, then
we may regard it as a decorated cobordism from $S^{3}$ to $S^{3}$,
which is a morphism in $\Cat^\sharp$. The functor $\Lsharp$ then
assigns to $\Sigma$ a linear map from $\Lsharp(S^{3})$ to itself,
i.e~simply a value in $\F=\{0,1\}$ because of Lemma~\ref{lem:S3}:
\[
     \Lsharp(\Sigma) \in \{0,1\}.
\]

To unwrap the this a little, let us first look at the morphism $\bX$
in $\Cat^{\sharp}$ corresponding to $\Sigma$. From the definition, the
cylinder $\R\times (S^{3})^{\sharp}$ is $\R\times H^{3}$, where
$H^{3}$ is the trifold whose singular set is the Hopf link. The
morphism $\bX$ is obtained by placing the foam $\Sigma$, with its
decoration of dots, in a 4-ball in this cylinder $\R\times H^{3}$. We
write
\[
           \bX = (\check X, \mu_{1}, \dots, \mu_{l})
\]
as before, where $\check X$ is the underlying orbifold and the
$\mu_{k}$ are the decorations. In keeping with
Notation~\ref{notation:dots}, we will have $\mu_{k}=\sigma_{1}(x_{k})$
for some $x_{k}$ on a facet.

In $\bonf(H_{3})$, there is a unique (non-degenerate and irreducible)
critical point $\alpha_{0}$ for the Chern-Simons functional and we
have moduli spaces
\begin{equation}\label{eq:closed-foam-moduli}
                M_{\kappa}(\alpha_{0}, \check X^{\sharp}, \alpha_{0})
\end{equation}
on the cylindrical-end manifold, where $\kappa$ is the action. The
dimension formula for this moduli space can be deduced from the
formula in the closed case, Proposition~\ref{prop:dimension}, and is
\begin{equation}\label{eq:dim-closed-foam}
            \dim M_{\kappa} = 12\kappa + \Sigma\ccdot\Sigma + 2 \chi(\Sigma) - t,
\end{equation}
where $t$ is the number of tetrahedral points. The action $\kappa$ is
non-negative, and is an integer linear combination of $1/2$ and $1/3$.
So $\kappa\in (1/6)\Z$. When bubbling occurs for a sequence of
connections in this moduli space, the change in the action is an
integer linear combination of $1/2$ and $1/3$ \emph{with non-negative
coefficients}, because that is the minimum charge for a bubble on the
bifold locus or trifold locus respectively. The following lemma is a
consequence.

\begin{lemma}
    If $\kappa$ is less than $1/3$, then the moduli space
    \eqref{eq:closed-foam-moduli} is compact.  \qed
\end{lemma}

For the evaluation $\Lsharp(\bX)$ to be non-zero, it is necessary that
the dimension of the transverse intersection
\[
     M_{\kappa}(\alpha_{0}, \check X^{\sharp}, \alpha_{0}) \cap \cV_{1} \cap
     \dots \cap \cV_{l}   
\]
is zero for some $\kappa$, where $\cV_{k}$ represents the class $\mu_{k}$. This means
that
\begin{equation}\label{eq:dim-zero}
   12\kappa + \Sigma\ccdot\Sigma + 2 \chi(\Sigma) - t = \sum_{1}^{l}
   \deg \mu_{k}.
\end{equation}
When this occurs, $\Lsharp(\Sigma)$ is defined by counting the points
in this transverse intersection, provided that the moduli space is
regular (cut out transversely by the equations). We are concerned with
the case that $\mu_{k}=\sigma_{1}(x_{k})$ is the 2-dimensional class,
so the above condition is
\begin{equation}\label{eq:dim-zero-l}
   12\kappa + \Sigma\ccdot\Sigma + 2 \chi(\Sigma) - t = 2l.
\end{equation}

When the moduli space $M_{\kappa}$ is compact, the number of points in
the transverse intersection is simply the ordinary evaluation of the
mod 2 cohomology class \[ \mu_{1} \cupprod \dots \cupprod \mu_{l} \in
H^{d}(\bonf(\check X); \F)\] on the fundamental class of the moduli
space.

When $\kappa=0$, the moduli space $M_{\kappa}(\alpha_{0}, \check
X^{\sharp}, \alpha_{0})$, without perturbation, is the moduli space of
flat connections, and in this case regularity of the moduli space (as
a moduli space of ASD connection) is equivalent to regularity of the
flat connections, which in turn is easily verified in any particular
case. The moduli space of flat connections is also the space of
homomorphisms from the orbifold fundamental group of $\check X$ to
$\SU(3)$.

\subsection{The 2-sphere}

\begin{proposition}\label{prop:sphere-with-dots}
    Let $S(l)$ denote the unknotted $2$-sphere with $l$ dots, as a foam
    in $\R^{4}$ or $S^4$. Then
    \[
    \Lsharp(S(l))=
    \begin{cases}
        0,&\text{$l=0$ or $l$ odd},\\
        1,&\text{$l\ge 2$ and even.}
    \end{cases}
\]    
\end{proposition}

\begin{proof}
    The dimension constraint \eqref{eq:dim-zero-l} requires $12\kappa
    + 4 = 2l$, so the evaluation is zero for $l=0$ and $l=1$, and for
    $l=2$ it coincides with the evaluation of the class
    $\sigma_{1}^{2}$ on the 4-dimensional moduli space of flat
    connections. This moduli space is regular, and is a copy of
    $\CP^{2}$. The class $\sigma_{1}$ is the first Chern class of the
    tautological line bundle, so the evaluation of $\sigma_{1}^{2}$ is
    $1$. This deals with the cases where $l < 3$. For larger $l$, we
    note that Lemma~\ref{lem:Xie-rel-mod-2} gives
    \[
                    \Lsharp(S(l + 3)) = \Lsharp(S(l+1))                    
    \]
    for all $l\ge 0$. This is sufficient to complete the proof.
\end{proof}

\subsection{The theta foam}

We consider next the theta foam, consisting of three standard disks in
$S^{4}$ meeting in a circular seam. We write $\Theta(l_{1}, l_{2},
l_{3})$ for this foam decorated with $l_{i}$ dots on the $i$'oh disk.

\begin{proposition}\label{prop:theta-evaluation}
    For the theta foam with dots, $\Theta(l_{1}, l_{2}, l_{3})$, we
    have
    \[
                 \Lsharp( \Theta(l_{1}, l_{2}, l_{3})) = 1
     \]
    if $(l_{1}, l_{2},l_{3}) = (0,1,2)$ or some permutation
    thereof, or more generally if
    \[
                (l_{1}, l_{2}, l_{3}) = (0, 1 + 2m, 2+2n)
    \]
    for non-negative integers $n$ and $m$. 
\end{proposition}

\begin{proof}
    The dimension constraint \eqref{eq:dim-zero-l} imposes the
    condition $12\kappa + 6 = 2(l_{1}+l+2 + l_{3})$, so there is no
    contribution if $l_{1} + l_{2} + l_{3}<3$. When $l_{1} + l_{2} +
    l_{3}=3$, we are again evaluating ordinary cohomology classes on
    the fundamental class of the moduli space of flat connections. In
    this case, the moduli space $M_{0}$ is the flag manifold of
    $\C^{3}$ and the three cohomology classes $\sigma_{1}(x_{i})$ are
    the first Chern classes $\mu_{1}$, $\mu_{2}$, $\mu_{3}$ of the
    three tautological line bundles. From the known cohomology ring of
    the flag manifold, we have
    \[
\begin{aligned}
    \Lsharp(\Theta(0,1,2)) &= (\mu_{2}\cupprod\mu_{3}^{2})[M_{0}] \\
                            &= 1,
\end{aligned}
    \]
   with the same answer mod 2 for any permutation of the three
   classes.

To proceed further, we note that the first of the dot-migration rules
of Lemma~\ref{lem:dot-migration} gives
\[
\Lsharp(\Theta(l_{1}+1, l_{2}, l_{3})) + \Lsharp(\Theta(l_{1}, l_{2}+1,
l_{3})) +  \Lsharp(\Theta(l_{1}, l_{2}, l_{3}+1)) =0,
\]
whenever the $l_{i}$ are non-negative. In particular, if
$l_{1}=l_{2}$, we obtain
\[
\begin{aligned}
\Lsharp(\Theta(l_{1}, l_{1}, l_{3}+1)) 
&= 0.
\end{aligned}
\]
In particular, we obtain zero for the evaluation in the cases
$(1,1,1)$, $(1,1,2)$, and $(2,2,1)$.

The relation of Lemma~\ref{lem:Xie-rel-mod-2} gives
\[
\Lsharp(\Theta(l_{1}+3, l_{2}, l_{3}))  =
       \Lsharp(\Theta(l_{1}+1, l_{2}, l_{3})) 
\]
whenever $l_{1}\ge 0$, and this allows the calculation for any
$(l_{1}, l_{2}, l_{3}))$ to be reduced to cases where each $l_{i}$ is
$0$, $1$, or $2$. We have dealt with all such cases already, with the
exception of $(2,2,0)$. But using dot-migration and the above relation
one more time, we have
\[
\begin{aligned}
    (2,2,0)  &= (3,1,0) + (2,1,1) \\
             &= (1,1,0) + 0 \\
             &= 0,
\end{aligned}
\]
in the obvious shorthand. This completes the calculation in all cases
and verifies the proposition.
\end{proof}

\subsection{The suspension of the tetrahedron}
\label{subsec:tetrahedron-foam}

Let $K$ be the 1-skeleton of the tetrahedron, as a web in $S^{3}$. Let
$\Sigma_{-} \subset B^{4}$ be a cone on $K$, a foam with one
tetrahedral point, and let $\Sigma = \Sigma_{-}\cup\Sigma_{+}\subset
S^{4}$ be the double of $\Sigma)=_{-}$, the suspension of $K$. The web
$K$ has 6 edges, and we label them as shown in
Figure~\ref{fig:tetrahedron-web}. We write $E_{1}$, $E_{2}$, $E_{3}$
for the facets of $\Sigma$ corresponding to the edges $e_{i}$ and
$F_{1}$, $F_{2}$, $F_{3}$ for those corresponding to $f_{i}$. We write
$\Sigma(k_{1}, k_{2}, k_{3}; l_{1}, l_{2}, l_{3})$ for the foam
decorated with dots on these three facets respectively.

\begin{figure}
    \begin{center}
        \includegraphics[scale=0.50]{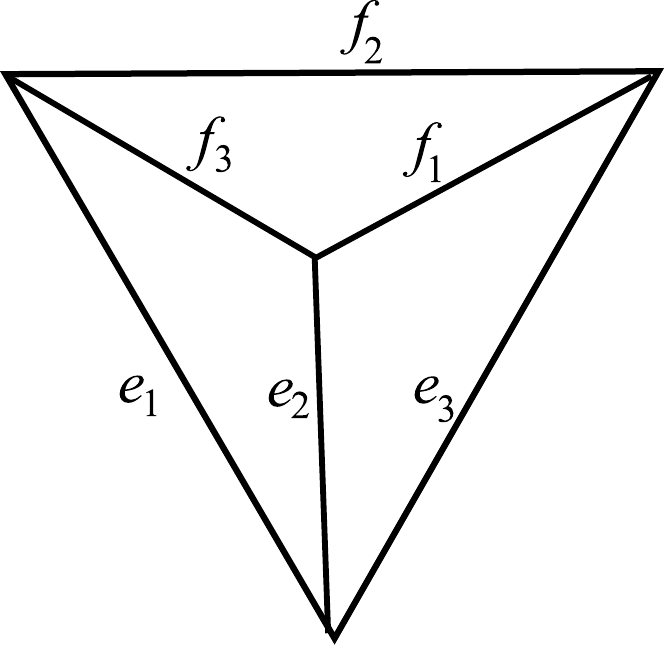}
    \end{center}
    \caption{\label{fig:tetrahedron-web}
   The tetrahedron.}
\end{figure}

\begin{proposition}\label{prop:tetr-evaluation}
    For the suspension of the tetrahedron with dots, we have
    \[
                \Sigma(k_{1}, k_{2}, k_{3}; 0,0,0) = 1
     \]
    if $(k_{1}, k_{2}, k_{3}) = (0,1,2)$ or some permutation
    thereof, or more generally if
    \[
                (k_{1}, k_{2}, k_{3}) = (0, 1 + 2m, 2+2n)
    \]
    for non-negative integers $n$ and $m$. For other values of
    $(k_{1}, k_{2}, k_{3})$, the evaluation is zero.
\end{proposition}

\begin{proof}
    The proof is identical to that of
    Proposition~\ref{prop:theta-evaluation}, because moduli space of
    flat bifold connections is again the flag manifold $\SU(3)/T$.
\end{proof}

With a little further examination below, we will deal also with the
case that the $l_{i}$ are non-zero. We state the result here and prove
it later:

\begin{proposition}
        For the suspension of the tetrahedron with dots, we have
    \[
                \Sigma(k_{1}, k_{2}, k_{3}; l_{1}, l_{2}, l_{3}) =
                \Sigma(k_{1}+l_{1}, k_{2}+l_{2}, k_{3}+l_{3}; 0,0,0)
     \]
     for all $k_{i}$ and $l_{i}$. \qed
\end{proposition}

\subsection{Some foams based on \texorpdfstring{$\RP^{2}$}{RP2}}
\label{subsec:RP2-foams}

For the proof of the exact triangles in
section~\ref{sec:exact-triangle} we will need to understand the
smallest-energy moduli spaces for some particular foams in $S^{4}$.
This material closely follows the case of the $\SO(3)$ gauge group as
presented in \cite[section 4]{KM-triangles}. As in that previous
paper, we denote by $R\subset S^{4}$ a standard copy of $\RP^{2}$ with
self-intersection number $+2$, arising as the branch locus of the
quotient map $S^{4}=(-\CP^{2})/c$, where $c$ is complex conjugation.
For $n=0,1,2$ or $3$, we consider $n$ lines in general position in
$-\CP^{2}$ defined by real equations,
\[
           L_{1},\dots,L_{n} \subset -\CP^{2}.
\]
Their images in the orbifold $(S^{4}, R)$ are $n$ disks
$D_{1},\dots,D_{n}$ with disjoint interiors. When $n\ge 2$, these
disks meet in pairs at their boundaries in $R$. The union
\begin{equation}\label{eq:Psi-n-foam}
            \Psi_{n} = R\cup D_{1} \cup \dots \cup D_{n}
\end{equation}
is a foam in $S^{4}$ whose seams lie along lines in $R$ and whose
tetrahedral points are where the disks meet.

The next lemma is the counterpart of \cite[Lemma
4.1]{KM-triangles}, which was the $\SO(3)$ case.

\begin{lemma}
    \label{lem:Psi-n-dim} The formal dimension of the moduli space of
    anti-self-dual bifold $\SU(3)$ connections of action $\kappa$ on
    $(S^{4}, \Psi_{n})$ is given by
    \[
                12 \kappa -4 +2n -n^{2}/2
    \]
     In particular, we have
    \begin{enumerate}
    \item $12\kappa - 4$ for $\Psi_{0}$;
    \item $12\kappa - 5/2$ for $\Psi_{1}$;
    \item $12\kappa -2 $ for $\Psi_{2}$;
    \item $12\kappa - 5/2$ for $\Psi_{3}$.
    \end{enumerate}
\end{lemma}

\begin{proof}
    As in \cite{KM-triangles}, we have $\Psi_{n} \ccdot \Psi_{n} = 2 -
    n/2$ and $\chi(\Psi_{n})=n+1$, while the number of tetrahedral
    points is $n(n-1)/2$. With this information, the formulae are
    derived directly from the general dimension formula in
    Proposition~\ref{prop:dimension}.
\end{proof}

With the dimension formula above, we examine the smallest-energy
moduli spaces in each case. The results and the proofs exactly
parallel the results from the $\SO(3)$ case,
\cite[Lemma~4.2]{KM-triangles}. The stabilizers of the various
reducible solutions in the $\SO(3)$ case were $O(2)$, the Klein
4-group $V_{4}$, or the group of order $2$, while in the $\SU(3)$ case
the corresponding groups are $U(2) = S(U(1)\times U(2))$, the maximal
torus, and a circle subgroup. (See the discussion of reducible
solutions in section~\ref{subsec:config}.) The formal dimensions of
the moduli spaces reflects the dimension of the stabilizers.

\begin{lemma}
\label{lem:smallest-action-R}
On the bifolds corresponding to $(S^{4}, \Psi_{n})$, the smallest-action
    non-empty moduli spaces of anti-self-dual bifold connections  are
    as follows, for $n\le 3$.
    \begin{enumerate}
    \item For $n=0$ or $n=2$, there is a unique solution with $\kappa=0$: 
        a flat connection whose holonomy group has order $2$ for $n=0$
        and is the Klein $4$-group, $V_{4}$, for $n=2$. The
        automorphism group of the connection is $U(2)$ (respectively,
        the maximal torus $T^{2}$),
         and it is an
        unobstructed solution in a moduli space of formal dimension
        $-4$ (respectively, dimension $-2$).
    \item For $n=1$ and $n=3$,
        the smallest non-empty moduli spaces have
        $\kappa=1/8$ and formal dimension $-1$. In both cases,  with
        suitable choices of bifold metrics, 
        the moduli space consists of a unique unobstructed
        solution with holonomy group $O(2)\subset U(2)$ and stabilizer
        $S^{1}$.
    \end{enumerate}
\end{lemma}

\begin{proof}
    The corresponding result for the $\SO(3)$ case is Lemma~4.2 from
    \cite{KM-triangles}, where it was shown that the smallest
    non-empty moduli spaces consist in each case of a unique solution.
    The inclusion of $\SO(3)$ in $\SU(3)$ gives the map $\r$ of
    section~\ref{sec:includeSO3} on the spaces of connections. The
    statement of the present lemma amounts to the assertion that by
    applying $\r$ to the $\SO(3)$ solutions we obtain $\SU(3)$
    solutions which are unique in their moduli spaces and are
    unobstructed. We illustrate how this goes.  
    The complement $S^{4}\sminus R$ deformation-retracts onto another copy of
    $\RP^{2}$ denoted $R'$ and $S^{4}\sminus
    \Psi_{n}$ has the homotopy type of $R'$  with $n$ punctures.
    The fundamental group is $\Z/2$ for $n=0$ and $\Z *
    \Z$ for $n=2$. For $n=0$ and $2$, the smallest possible action is
    $\kappa=0$, and elements of $M_{0}$ are representations of $\Z/2$
    or $\Z * \Z$ in $\SU(3)$ sending the
    standard generators to involutions, modulo conjugation.
    In the second case, the two
    involutions must be distinct and commuting because of the presence
    of the tetrahedral point where the disks meet. In each of these
    cases $n=0,2$, there is therefore exactly one solution with
    $\kappa=0$, and the stabilizers are $U(2)$ and the maximal torus
    respectively, as claimed. For a moduli space of flat bifold connections,
    the space of infinitesimal deformations $H^{1}_{A}$ as ASD
    bifold connections coincides with the deformation space as flat
    bifold connections, allowing us to read off that $H^{1}_{A}=0$ in
    both of these cases.  The dimension of $H^{0}_{A}$ is the
    dimension of the stabilizer, which is $4$ or $2$ respectively. From
    the dimension formula, which gives the index of the deformation
    complex, we can then obtain the dimension of $H^{2}_{A}=0$, and we
    see that it is zero in both cases. The solutions are therefore
    unobstructed.

    For the case $n=1$, we exploit as in \cite{KM-triangles} the
    existence of a conformally anti-self-dual orbifold metric on
    $(S^{4}, \Psi_{1})$ to see that the solutions are unobstructed.
    The formal dimension being $-1$, the solutions for $\kappa=1/8$
    must be reducible. There is no possibility of having stabilizer
    larger than $S^{1}$ as there are no abelian solutions here, so the
    stabilizers must be $S^{1}$, and the dimension formula tells us
    that the solutions are isolated. It remains to show that there is
    exactly one. If $A$ is a solution, consider the pull-back $\tilde
    A$ on the double cover along $R$, namely the bifold $(-\CP^{2},
    L_{1})$ whose orbifold locus is $L_{1}$. We have $\kappa(\tilde
    A)=1/4$ (twice that of $A$), and the dimension formula tells us
    that the formal dimension of the moduli space containing $\tilde
    A$ is $-2$. Again, $\tilde A$ is unobstructed for the same reason
    that $A$ is. So $\tilde A$ must have stabilizer the maximal torus.
    Thus the bifold bundle with connection $(\tilde E, \tilde A)$ is a
    sum of bifold line bundles
    \[
           \tilde E = M_{0}\oplus M_{1}\oplus M_{2}
    \]
    where the local orbifold group $\Z/2$ acts as $+1$ on the first
    factor and $-1$ on the other two. Let $m_{0}$, $m_{1}$, $m_{2}$ be
    their orbifold degrees. These sum to zero because the structure
    group is $\SU(3)$. We have $m_{0}\in \Z$ and $\m_{1}$, $\m_{2}$ in
    $1/2 + \Z$. Furthermore the solution is invariant under complex
    conjugation, which forces $m_{2}=-m_{1}$ and $m_{0}=0$. We want
    $\kappa=1/4$, and this implies that $m_{1}=1/2$ and $m_{2}=-1/2$
    (or vice versa). This determines $\tilde E$ uniquely. To pass back
    to $A$, we need to consider how the involution $c$ on $(-\CP^{2},
    L_{1})$ is acting on $\tilde E$. At a fixed point of the action,
    the involution $c$ on the fiber is a complex-linear involution on
    the fiber of $M_{1}\oplus M_{2}$ interchanging the two factors.
    The eigenvalues here are $1$ and $-1$ therefore. To make the
    action of bifold type, the involution $c$ therefore has to act as
    $-1$ on the trivial line bundle $M_{0}$. This determines the
    action of $c$ uniquely, and completes the argument for $n=1$.  

    In the case $n=3$, as in \cite{KM-triangles}, the solutions must
    again be unobstructed with stabilizer $U(1)$. The issue is again
    the uniqueness. The bifold corresponding to $\Psi_{3}\subset
    S^{4}$ is a quotient of $-\CP^{2}$ by an elementary abelian group
    of order $8$: this is the Klein 4-group acting as $\pm 1$ on the
    coordinates, complex conjugation. A solution $A$ of action $1/8$
    pulls back to a solution $\tilde A$ of action $1$ on the $8$-fold
    cover. The stabilizer of $\tilde A$ must be at least as large as
    the stabilizer of $A$, so $\tilde A$ reduces to the subgroup
    $S(U(1)\times U(2))\subset \SU(3)$.

    We claim that $\tilde A$ must actually reduce to the maximal
    torus, \[T^{2}=S(U(1)\times U(1)\times U(1)).\] If it did not,
    then the associated $U(1)$ bundle to the $S(U(1)\times U(2))$
    reduction would have to be invariant under complex conjugation,
    and would therefore have to be trivial. So $\tilde A$ would arise
    from the inclusion of an irreducible $\SU(2)$ instanton in
    $\SU(3)$. The irreducible $\SU(2)$ instantons with $\kappa=1$ are
    an open cone on $-\CP^{2}$ and the action of the group of order
    $8$ does not have isolated fixed points. This contradicts the fact
    that $A$ is isolated on the quotient orbifold.

    It follows that $(\tilde E, \tilde A)$ is a sum of line bundles
    again, say $N_{0}\oplus N_{1} \oplus N_{2}$. The invariance under
    complex conjugation forces their degrees to be $0$, $k$ and $-k$
    respectively, and since $c_{2}=1$ we must have $k=\pm1$. This
    uniquely determines $\tilde A$. The action of complex conjugation
    $c$ on $-\CP^{2}$ must be as $-1$ on the trivial summand $N_{0}$
    and must interchange the other two, by the same argument as in the
    $n=1$ case above. The group of order $8$ is $\langle c \rangle
    \times V_{4}$. The $V_{4}$ subgroup must preserve the lines
    $N_{i}$ separately, because they have three different degrees. At
    a fixed point of $\langle c \rangle \times V_{4}$, i.e.~at a point
    $x$ where two of the lines, say $L_{1}$ and $L_{2}$, meet
    $\RP^{2}$, the action of $c$ on $(N_{0}\oplus N_{1}\oplus
    N_{2})_{x}$ is by
    \[
               c =  
               \begin{pmatrix}
                -1 & 0 & 0 \\
                0 & 0 & 1 \\
                0 & 1 &0
               \end{pmatrix}.
    \]
    The two generating involutions $a$ and $b$ in $V_{4}$ whose fixed
    sets near $x$ are $L_{1}$ and $L_{2}$ must act on $(N_{0}\oplus
    N_{1}\oplus N_{2})_{x}$ by non-trivial diagonal matrices of bifold
    type, and since they must both commute with $c$, the only
    possibility is
    \[
             a = b = \begin{pmatrix}
                1 & 0 & 0 \\
                0 & -1 & 0 \\
                0 & 0 & -1
               \end{pmatrix}.
    \]
    This shows that the solution on $(S^{4}, \Psi_{3})$ is unique. A
    comparison with results of the $\SO(3)$ case \cite{KM-triangles}
    shows that it must be obtained from the unique $\SO(3)$ bifold
    solution (which has action $1/32$) by the inclusion $\r$.
\end{proof}

\section{Consequences of the closed foam calculations}

\subsection{Unknots and unlinks}

Let $U_{n}$ denote a standard $n$-component unlink in $\R^{3} \subset
S^{3}$. We examine $\Lsharp(U_{n})$ in the notation of
\eqref{notn:webs}. By the excision property,
Proposition~\ref{prop:excision}, this is the tensor product of $n$
copies of $\Lsharp(U_{1})$ in a natural way. The next proposition
identifies $\Lsharp(U_{1})$. The unknot $U_{1}$ has a single edge $e$,
from which is obtained a linear operator $\sigma_{1}(e)$ acting on
$\Lsharp(U_{1})$. From Lemma~\ref{lem:Xie-rel-mod-2}, we know that
this satisfies the relation $\sigma_{1}(e)^{3} + \sigma_{1}(e)=0$.

\begin{proposition}
    \label{prop:unknot} The $\SU(3)$ bifold homology of the unknot
    $U_{1}$ is a 3-dimensional vector space over $\F$. As a module for
    the polynomial ring $\F[u]$, where $u$ acts by the operator
    $\sigma_{1}(e)$, the three-dimensional vector space
    $\Lsharp(U_{1})$ is isomorphic to the cyclic module
    \[
            \Lsharp(U_{1}) \cong \F[u] / (u^{3}+u).
    \]
    A cyclic generator for the module is the element $\Lsharp(D)$,
    where $D$ is a standard disk in the ball, with boundary $U_{1}$.
    (See Notation~\ref{not:rel-invariant}.)
\end{proposition}

\begin{proof}
    The representation variety of the orbifold $(S^{3},
    U_{1})^{\sharp}$ is $\CP^{2}$ and is Morse-Bott, which shows as
    usual that the dimension of $\Lsharp$ is at most $3$. Let $D(l)$
    be the disk with $l$ dots, as a cobordism from the empty set to
    $U_{1}$ and let $D'(l)$ be the opposite morphism. The elements
    $\Lsharp(D(l)) \in \Lsharp(U_{1})$ for $l=0,1,2$ are linearly
    independent, because we can compute their pairings with
    $\Lsharp(D'(l'))$ for $l'=0,1,2$ and verify that the determinant
    of the pairing matrix is non-zero: the pairings are the
    evaluations of the closed foam $S^{2}(l+l')$ which are given by
    Proposition~\ref{prop:sphere-with-dots}, yielding the matrix,
    \[
\begin{pmatrix}
    0 & 0 & 1\\
    0 & 1 & 0\\
    1 & 0 & 1
\end{pmatrix}.
    \]
    As well as establishing independence, the matrix shows that $D(0)$
    provides a cyclic generator. 
\end{proof}

From the proposition, we can pull out a corollary.

\begin{corollary}
    \label{cor:basis-dual-basis}
    A basis for $\Lsharp(U_{1})$ is given by 
    \[
                \beta_{l} = \Lsharp(D(l)), \qquad l =0,1,2,
    \]
    which are given by $\beta_{l}=u^{l}\beta_{0}$.
    The dual basis are the elements $(\beta'_{0}, \beta'_{1},
    \beta'_{2})$ in $\Hom(\Lsharp(U_{1}), \F)$ given by
    \[
    \begin{aligned}
    \beta'_{0} &= \Lsharp(D'(2)) + \Lsharp(D'(0)) \\
                  \beta'_{1} &= \Lsharp(D'(1))  \\
                   \beta'_{2} &= \Lsharp(D'(0)) .
                   \end{aligned}
    \]\qed
\end{corollary}

Having a basis and dual basis for $\Lsharp(\check Q)$ when $\check Q$
is the bifold corresponding to an unlink in $S^{3}$, we are able to
test whether an element of $\Lsharp(\check Q)$ is zero by pairing it
with the dual basis elements. This allows us to draw the following
particular corollary from Proposition~\ref{prop:excision-local}.

\begin{corollary}\label{cor:excision-unlink}
    As in Proposition~\ref{prop:excision-local}, let $\bX_{1}, \dots,
    \bX_{r}$ be morphisms from $\check Y_{0}$ to $\check Y_{1}$ of the
    form $\bX' \cup \mathbf{P}_{i}$, where the $\mathbf{P_{i}}$ are
    decorated bifolds corresponding to dotted foams $\Sigma_{i}\subset
    B^{4}$ with common boundary $\check Q = (S^{3}, U_{n})$, where
    $U_{n}$ is the unlink of $n$ components. Let $D(l_{1}, \dots,
    l_{n})$ be the foam consisting of $n$ standard disks, with
    boundary $U_{n}$, carrying $l_{i}$ dots on the $i$oh disk, with
    each $l_{i}$ at most $2$. Suppose that for all such $(l_{1},
    \dots, l_{n})$, we have
    \[
               \sum_{i=1}^{r} \Lsharp ( \mathbf{P}_{i} \cup_{\check Q}
               D(l_{1}, \dots, l_{n})) = 0.
    \]
    as evaluations of closed foams in $\R^{4}$. Then
    \[
                \sum_{i=1}^{r} \Lsharp(\bX_{i}) = 0
    \]
    as linear maps  $\Lsharp(\check Y_{0}) \to \Lsharp(\check
    Y_{1})$. \qed
    \end{corollary}

\subsection{The theta web}

Just as the $\Lsharp$ of the unknot is computed from the evaluation of
the $2$-sphere with dots, so $\Lsharp$ of the theta web is computed
from the evaluation of the theta foam with dots.

\begin{proposition}
\label{prop:theta-web}
Let $K$ denote the theta web, and regard $\Lsharp(K)$ as a module over
the polynomial in $\F[u_{1}, u_{2}, u_{3}]$, where $u_{i}$ acts by the
dot operators $\sigma_{1}(e_{i})$ corresponding to the the three
edges. Let $\iota \in \Lsharp(K)$ be the element obtained from
regarding $K$ as the boundary of half of the theta foam,
$\Theta=\Theta_{-} \cup \Theta_{+}$. Then $\Lsharp(K)$ is a
6-dimensional vector space over $\F$ and is a cyclic module over
polynomial ring $\F[u_{1}, u_{2}, u_{3}]$ with cyclic generator
$\iota=\Lsharp(\Theta_{-})$ and relations
    \[
\begin{gathered}
    u_{1} + u_{2} + u_{3} = 0 \\
         u_{1}u_{2} + u_{2}u_{3} + u_{3}u_{1}=1 \\
                u_{1}u_{2}u_{3}=0.
\end{gathered}
\] \qed
\end{proposition}

Concretely, this tells us that a basis for $\Lsharp(K)$ is given by
$u_{1}^{l_{1}}u_{2}^{l_{2}}\iota$ with $l_{1}\le 2$ and $l_{2}\le 1$.

\subsection{The tetrahedral web}

As in section~\ref{subsec:tetrahedron-foam}, let $K$ be the 1-skeleton
of the tetrahedron, with edges labeled as in
Figure~\ref{fig:tetrahedron-web}. Let $\Sigma_{-}$ be the cone on $K$
and let $\Sigma$ be the suspension of $K$. Let
\[
\begin{aligned}
\iota &= \Lsharp(\Sigma_{-}) \\
                     &\in \Lsharp(K).
                     \end{aligned}
\]
Let $u_{i}$ be the three operators $\sigma_{1}(e_{i})$ and let $v_{i}$
be the operators $\sigma_{1}(f_{i})$ acting on $\Lsharp(K)$.

\begin{proposition}
The vector space $\Lsharp(K)$ for the tetrahedral web
has the following description.
    \label{prop:tetrahedron-web}
\begin{enumerate}
    \item It has
    dimension $6$ over $\F$ and is a cyclic module over $\F[u_{1},
    u_{2}, u_{3}]$ with generator $\iota$ and the same relations as
    for the theta web:
        \[
\begin{gathered}
    u_{1} + u_{2} + u_{3} = 0 \\
         u_{1}u_{2} + u_{2}u_{3} + u_{3}u_{1}=1 \\
                u_{1}u_{2}u_{3}=0.
\end{gathered}
\]
\item For $i=1,2,3$, the operator $v_{i}=\sigma_{1}(f_{i})$ on
$\Lsharp(K)$ is equal to the operator $u_{i}=\sigma_{1}(e_{i})$.
\end{enumerate}   \qed
\end{proposition}

\begin{proof}
    The first part follows from
    Proposition~\ref{prop:tetr-evaluation}, in just the same way that
    Proposition~\ref{prop:theta-web} (for $\Lsharp$ of the theta web)
    followed from Proposition~\ref{prop:theta-evaluation} (the
    evaluations of the the theta-foam). This is because the
    representation variety is again the flag manifold $\SU(3)/T$ and
    the cohomology classes $\sigma_{1}(e_{i})$ are the three
    tautological classes.

    For the second part, it is enough to show that $v_{1}\iota =
    u_{1}\iota$. On general grounds we have
    \[
                 v_{1}\iota = p(u_{1}, u_{2}, u_{3})\iota
    \]
    for some polynomial $p$; and because of the relations, we may take
    it that $p$ has the form
    \[
                p = A + Bu_{2} + C u_{3} + D u_{2}^{2} + E u_{2}u_{3}
                + F u_{2} u_{3}^{2}.
    \]
    We must show that $p=u_{2}+u_{3}$ (because this is $u_{1}$). From
    the symmetries of the tetrahedron, this expression (modulo the
    relations) must be invariant under interchanging $u_{2}$ and
    $u_{3}$. So $B=C$ and $D=0$. The evaluations
    $\Lsharp(\Sigma(k_{1}, k_{2}, k_{3} ; l_{1}, l_{2}, l_{3})$ can be
    computed as ordinary evaluations of cohomology classes on the flag
    manifold when the total number of dots is 3, and vanish when the
    number of dots is less than 3. Furthermore, as ordinary cohomology
    classes, $\sigma_{1}(f_{i}) = \sigma_{1}(e_{i})$. This provides
    the evaluations (in abbreviated notation):
    \[
\begin{aligned}
    (0,1,1;1,0,0)&=0;
    (0,0,0,;1,0,0)&=0.
\end{aligned}
    \]
    It follows that $E=0$ and $F=0$. So $p=A + B u_{1}$ for some $A$
    and $B$ in $\F$. This gives four possibilities to check, but only
    one of these has minimal polynomial $v^{3}+v$, namely the operator
    $u_{1}$. So $v_{1}=u_{1}$ as claimed.
\end{proof}

\subsection{Some computations for connected sums}

For the application to the proof skein exact triangle later, we will
need to examine particular sums involving the foams $\Psi_{n}$ from
\eqref{eq:Psi-n-foam} in section~\ref{subsec:RP2-foams}. This follows
\cite{KM-triangles} very closely, drawing on the description of the
smallest-action moduli spaces for these foams,
Lemma~\ref{lem:smallest-action-R}.

If $\Sigma
\subset X$ and $\Sigma'\subset X'$ with tetrahedral points $t$, $t'$ in
each, there is a connected sum
\begin{equation}\label{eq:tet-sum}
                 (X,\Sigma) \#_{t,t'} (X',\Sigma')
\end{equation}
performed by removing standard neighborhoods and gluing together the
resulting foams-with-boundary. The result is not unique, because the
gluing is performed along a copy of the $(S^{3}, K_{T})$, where
$K_{T}$ is the tetrahedral web, which has automorphisms that are not
isotopic to the identity. For uniqueness, we need to specify which
edges are glued to which.

With similar ambiguity, if $s$ and $s'$ are points on seams of
$\Sigma$ and $\Sigma'$, there is a connected sum
\[
                (X,\Sigma) \#_{s,s'} (X',\Sigma')
\]
along a theta web in $S^{3}$, and then there is the usual connect sum
of pairs,
\[
                (X,\Sigma) \#_{f,f'} (X',\Sigma'),
\]
formed at points $f$, and $f'$ in facets of the foams.

We consider a connected sum at a tetrahedral point in the case that
$(X',\Sigma')$ is either $(S^{4}, \Psi_{2})$ or $(S^{4}, \Psi_{3})$.

\begin{proposition}
 \label{prop:sum-with-Psi-2-3}
 Let $\check X=(X,\Sigma)$ be a morphism in $\Cat^{\sharp}$ (possibly
 decorated with dots). Let $t$ be a tetrahedral point in $\Sigma$.
    \begin{enumerate}
    \item If a new foam $\tilde\Sigma$ is constructed from $\Sigma$ as
        a connected sum \[(X,\Sigma) \#_{t, t_{2}} (S^{4},\Psi_{2}), \]
        where $t_{2}$ is the unique tetrahedral point in $\Psi_{2}$,
        then the new linear map  $\Lsharp(X,\tilde \Sigma)$ is equal to
        the old one, $\Lsharp(X,\Sigma)$.
    \item If a new foam $\tilde\Sigma$ is constructed from $\Sigma$ as
        a connected sum \[ (X,\Sigma) \#_{t, t_{3}} (S^{4},\Psi_{3}),\]
        where $t_{3}$ is any of the three tetrahedral points in $\Psi_{3}$,
        then the new linear map  $\Lsharp(X,\tilde \Sigma)$ is zero.   
    \end{enumerate}
\end{proposition}

\begin{proof}
    The proof of this proposition and the following two are almost
    identical to the proofs of the corresponding results
    \cite[Propositions~4.3--4.5]{KM-triangles} in the $\SO(3)$ case.
    We point out how the proofs get modified in the present case, and
    leave the remaining two.

    Consider a general connected sum at tetrahedral points, as in
    equation~\eqref{eq:tet-sum}. Let $A$ and $A'$ be unobstructed
    solutions on $(X,\Sigma)$ and $(X',\Sigma')$. Let $U_{A}$ and
    $U_{A'}$ be neighborhoods of $[A]$ and $[A']$ in their respective
    moduli spaces. The limiting holonomy of a bifold connection at a
    tetrahedral point is the Klein $4$-group $V$, whose commutant in
    $\SU(3)$ is the maximal torus $T$. By comparison, in the $\SO(3)$
    case, the commutant was also $V$. Moduli spaces of solutions with
    framing at $t$ and $t'$ contain neighborhoods of $[A]$ and $[A']$,
    say $\tilde U_{A}$, $\tilde U_{A'}$, such that $U_{A} = \tilde
    U_{A}/ T$ and $U_{A'} = \tilde U_{A'}/T$ are the unframed moduli
    spaces. The model for the moduli space on the connected sum with a
    long neck, has the form
       \[ 
            \tilde U_{A} \times_{T} \tilde U_{A'}.
        \]
    If the action of $T$ on $\tilde U_{A}$ is free and $U_{A'}$
    consists of the single point $[A']$, then this local model is a
    bundle over $U_{A}$ with fiber $T/\Gamma_{A'}$, where
    $\Gamma_{A'}\subset T$ is the automorphism group of the solution
    $A'$. 

    When $A'$ is the smallest energy solution on $(S^{4},\Psi_{2})$,
    then the fiber is a single point, while for $(S^{4}, \Psi_{3})$
    the fiber is a circle. For the case of compact, zero-dimensional
    moduli spaces on the connected sum, these local models become
    global descriptions when the neck is long, and we conclude that
    the moduli space whose point-count defines the map
    $\Lsharp(X,\Sigma)$ is unchanged in the first case and becomes
    empty in the second case, by dimension counting.
\end{proof}

Next we cover connected sums at seam points.

\begin{proposition} \label{prop:seam-sums}
 Let $\check X=(X,\Sigma)$ be a morphism in $\Cat^{\sharp}$ (possibly
 decorated with dots), as in the previous proposition. Let $s$ be a
 point in a seam of $\Sigma$. For $n=1,2,3$, let $s_{n}$ be a point on
 a seam of $\Psi_{n}$. If a new foam $\tilde\Sigma_{n}$ is constructed
 from $\Sigma$ as the connected sum
 \[(X,\Sigma) \#_{s, s_{n}} (S^{4},\Psi_{n}), \] then the new linear
 map $\Lsharp(X,\tilde \Sigma_{n})$ is equal to the old one in the
 case $n=2$, and is zero in the case that $n=1$ or $n=3$. \qed
\end{proposition}

Finally we have a proposition about connected sums at points
interior to faces of the foams. The case $n=0$ here is already stated
in Proposition~\ref{prop:R-sum} above.

\begin{proposition}\label{prop:face-sums}
    Let $\check X = (X,\Sigma)$ be a foam cobordism, as in the
    previous propositions. Let $f$ be a point in the interior of a
    face of $\Sigma$. Let $f_{n}$ be a point in a face of $\Psi_{n}$.
    If a new foam $\tilde\Sigma_{n}$ is constructed from $\Sigma$ as
    the connected sum \[(X,\Sigma) \#_{f, f_{n}} (S^{4},\Psi_{n}), \]
    then the new linear map $\Lsharp(X,\tilde \Sigma_{n})$ is equal to
    the old one in the case $n=0$, and is zero when $n=1$, $2$ or $3$.
    \qed
\end{proposition}

\section{The exact triangles and the octahedron}
\label{sec:exact-triangle}

\subsection{The set-up}

We now turn to the skein exact triangles which hold for $\Lsharp$.
These are essentially identical in their statement (and even their
proof) to the corresponding results for the $\SO(3)$ case, which are
stated as Theorems~1.1 and 1.2 in \cite{KM-triangles}. As with the
$\SO(3)$ case however, a more complete statement puts both triangles
together in a larger octahedral diagram. The following theorem is the
result. It exactly mirrors Theorem~9.1 in \cite{KM-triangles}, which
was the $\Jsharp$ case, and is summarized in
Figure~\ref{fig:L-octahedron}. Here it is understood that the webs
$K_{i}$ and $L_{i}$ all lie in the same $3$-manifold $Y$ and that they
are identical outside a ball, inside which they are as shown. Each
arrow in the diagram represents a standard foam in $[0,1]\times Y$.
See for example \cite{KM-triangles}.

\begin{figure}
    \begin{center}
        \includegraphics[scale=.46]{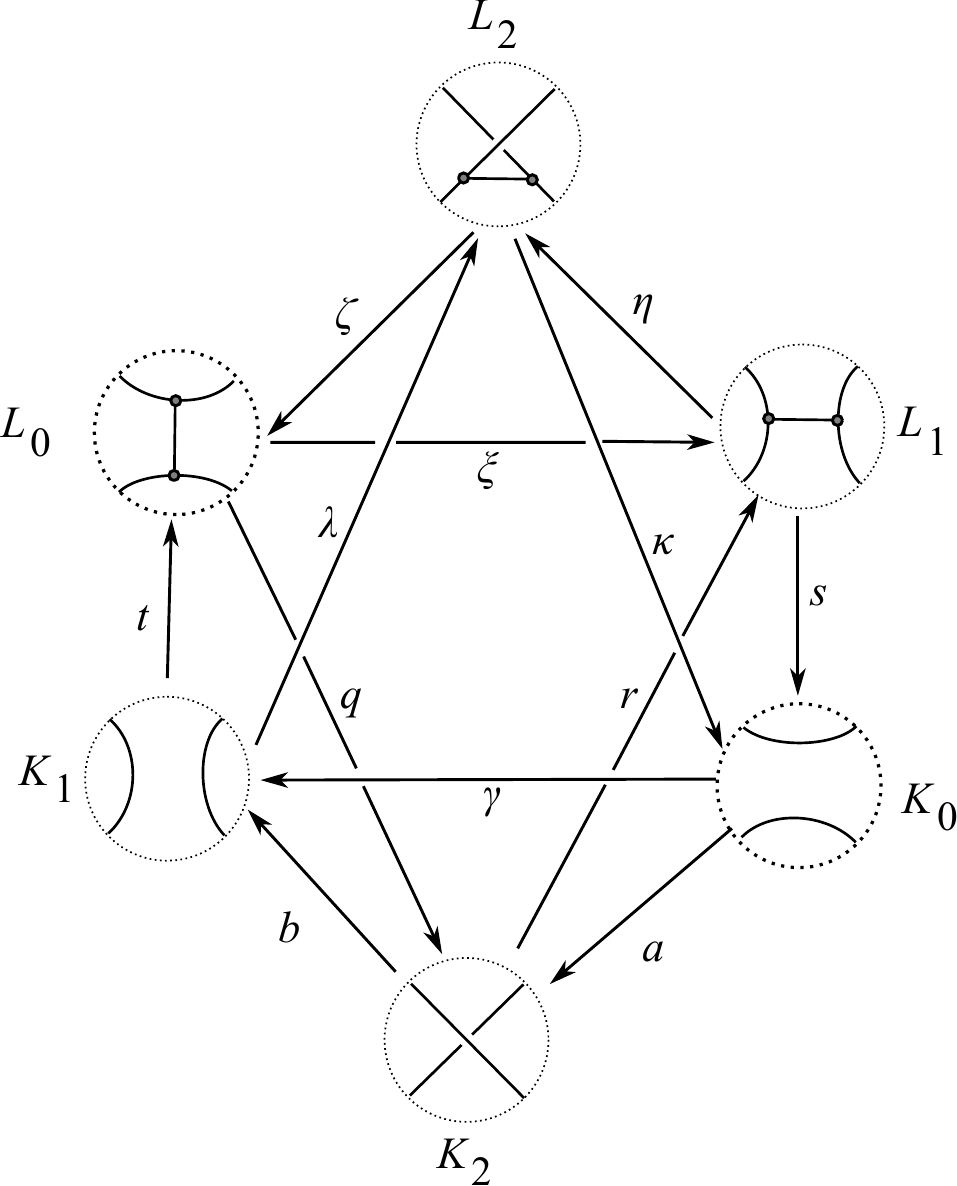}
    \end{center}
    \caption{\label{fig:L-octahedron}
    The octahedral diagram.}
\end{figure}

\begin{theorem}\label{thm:octahedron-statement}
In the diagram of standard cobordisms pictured in
Figure~\ref{fig:L-octahedron}, the triangles involving
\begin{enumerate}
\setlength{\itemsep}{0pt}
\item  \label{item:LKK-exact} $L_{0}$, $K_{2}$, $K_{1}$,
\item $L_{1}$, $K_{0}$, $K_{2}$, 
\item \label{item:LKK3-exact} $L_{2}$, $K_{0}$, $K_{1}$, and
\item \label{item:LLL-exact} $L_{2}$, $L_{0}$, $L_{1}$
\end{enumerate}
become exact triangles on applying $\Lsharp$. The faces
\begin{enumerate}
\setcounter{enumi}{4}
\setlength{\itemsep}{0pt}
\item \label{item:first-comm-face} 
      $K_{0}$, $K_{2}$, $K_{1}$,
\item  \label{item:second-comm-face}
      $L_{0}$, $K_{2}$, $L_{1}$,
\item  \label{item:third-comm-face}
      $K_{1}$, $L_{2}$, $L_{0}$, and
\item  \label{item:fourth-comm-face}
      $L_{1}$, $L_{2}$, $K_{0}$
\end{enumerate}
become commutative diagrams. And finally,
\begin{enumerate}
\setcounter{enumi}{8}
\setlength{\itemsep}{0pt}
\item \label{item:bottom-to-top} the composites $K_{2}\to K_{1} \to
    L_{2}$  and $K_{2} \to
    L_{1}\to L_{2}$ give the same map on $\Lsharp$, and
\item \label{item:top-to-bottom} the composites $L_{2}\to K_{0}\to
    K_{2}$  and $L_{2}\to
    L_{0}\to K_{2}$ give the same map on $\Lsharp$.
\end{enumerate}
\end{theorem}

\begin{remark}
    In the corresponding diagram in \cite{KM-triangles}, the web
    $L_{2}$ was named $L'_{2}$, because in that paper the notation
    $L_{2}$ had been reserved earlier for the mirror image of this
    local web.
\end{remark}

\begin{proof}
    We begin with the commutativity statements,
    \ref{item:first-comm-face}--\ref{item:fourth-comm-face}. The
    arguments are identical to the $\Jsharp$ case. The composite
    cobordism $K_{0}\to K_{2}\to K_{1}$ is equal to the connect sum
    $\Sigma\# \Psi_{0}$, where $\Sigma$ is the standard cobordism from
    $K_{0}$ to $K_{1}$. So the commutativity in case
    \ref{item:first-comm-face} follows from
    Proposition~\ref{prop:face-sums}. For \ref{item:second-comm-face},
    the argument is the same, except that the sum is with $\Psi_{2}$
    at a tetrahedral point, so Proposition~\ref{prop:sum-with-Psi-2-3}
    establishes the commutativity in this case. and with $\Psi_{2}$.
    In cases \ref{item:fourth-comm-face} and
    \ref{item:fourth-comm-face}, the composite has the form of a sum
    with $\Psi_{2}$ at a seam point, so
    Propositions~\ref{prop:seam-sums} deals with these cases.

    In each of the final two statements \ref{item:bottom-to-top} and
    \ref{item:top-to-bottom}, the first composite cobordism is
    obtained from the second composite by forming a connect sum with
    $\Psi_{2}$ at a tetrahedral point. So these cases also follow from
    Proposition~\ref{prop:sum-with-Psi-2-3}.

    The most interesting parts of the theorem concern the exactness of
    the triangles in the first four statements. Note that case
    \ref{item:LLL-exact} is different from the other three: the
    remaining ones are essentially all the same. In
    \cite{KM-triangles}, the proof of \ref{item:LLL-exact} was
    presented directly, while the remaining triangles,
    \ref{item:LKK-exact}--\ref{item:LKK3-exact}, were deduced from
    \ref{item:LLL-exact} by additional arguments. In order to slightly
    vary the approach, we will outline the direct approach to the
    proof of \ref{item:LKK-exact} instead, which is the exactness of
    the sequence
    \begin{equation}\label{eq:exact-Lsharp-eg}
           \cdots \stackrel{q}{\longrightarrow} \Lsharp(K_{2})
           \stackrel{b}{\longrightarrow}
                 \Lsharp(K_{1}) \stackrel{t}{\longrightarrow}
                  \Lsharp(L_{0}) \stackrel{q}{\longrightarrow}
                    \Lsharp(K_{2}) \stackrel{b}{\longrightarrow}
                    \cdots .
     \end{equation}                    
     (In what follows, the letters $b$, $t$, $q$ etc.~will refer
     variously to the actual cobordisms between the webs, or the
     induced maps on $\Lsharp$ or at the chain level on $\CLsharp$.)
     We will not dwell on the proof of exactness in the triangle
     \ref{item:LLL-exact}, because the proof is so similar to the
     $\SO(3)$ case \cite{KM-triangles}. Alternatively, the exactness
     of \ref{item:LLL-exact} can be deduced from the exactness of
     \ref{item:LKK-exact}, by the same sort of auxiliary arguments
     that were used for $\SO(3)$.

The argument for the exactness of \eqref{eq:exact-Lsharp-eg} follows a
standard layout. It begins by showing that the composite maps $b\comp
q$, $t\comp b$ and $q\comp t$ are zero at the level of homology on
$\Lsharp$, and in so doing we also construct explicit chain homotopies
to zero for the corresponding maps at the chain level of Floer
homology. So we have chain homotopies $j_{0}$, $j_{1}$ and $j_{2}$
with
     \begin{equation}\label{eq:first-chain-homotopy}
     \begin{aligned}
     b \, q  + d j_{2} + j_{2} d &= 0 \\
          q\, t  + d j_{0} + j_{0} d &= 0 \\
                     t \,b  + d j_{1} + j_{1} d &= 0 \\
 \end{aligned}
     \end{equation}
     where in each case $d$ denotes the differential on the singular
     instanton chain complex $\CLsharp$ for $\Lsharp$. So for example
     the first equation of these three expresses the vanishing of a
     chain map,
      \[
               b \,q  + d j_{1} + j_{1} d :
               \CLsharp(L_{0})\to \CLsharp(K_{1}).
      \]
     Note that, unlike other similar situations, all three cases here
     are slightly different, because of the lack of symmetry between
     $L_{0}$ and the others.

     Next one constructs second chain homotopies,
     \begin{equation}\label{eq:second-chain-homotopies}
\begin{aligned}
            k_{0} : \CLsharp(L_{0}) &\to \CLsharp(L_{0}) \\
            k_{1} : \CLsharp(K_{1}) &\to \CLsharp(K_{1}) \\
            k_{2} : \CLsharp(K_{2}) &\to \CLsharp(K_{2}) \\
\end{aligned}
     \end{equation}
     so that the following chain maps are \emph{isomorphisms} at the chain
     level:
\begin{equation}\label{eq:second-CH-fmla}
     \begin{aligned}
     t j_{1} + j_{0} q + d k_{0} + k_{0} d : \CLsharp(L_{0}) & \to
        \CLsharp(L_{0}) \\
                b j_{2} + j_{1} t + d k_{1} + k_{1} d : \CLsharp(K_{1}) & \to
        \CLsharp(K_{1}) \\
                q j_{0} + j_{2} b + d k_{2} + k_{2} d : \CLsharp(K_{2}) & \to
        \CLsharp(K_{2}) .
     \end{aligned}
\end{equation}
Given such chain homotopies, an algebraic argument used in
\cite{OS-double-covers} establishes the exactness of
\eqref{eq:exact-Lsharp-eg}, following model arguments in \cite{KMOS,
KM-unknot, KM-triangles}, for example.

We outline each of the steps for the argument: the vanishing of the
composites in \eqref{eq:exact-Lsharp-eg}, the construction of the
first chain homotopies \eqref{eq:first-chain-homotopy} and the second
chain homotopies \eqref{eq:second-chain-homotopies}. For reference in
what follows, the cobordisms $b$, $t$ and $q$ are depicted somewhat
schematically in Figure~\ref{fig:mobiusshadedwithfeetfoamonedisk}. The
cobordisms are trivial outside a region $[0,1]\times B^{3}$, and the
non-trivial parts are drawn. It consists of plumbed twisted bands, and
the composite $b\cup t\cup q$ contains two M\"obius bands. The core of
one the M\"obius bands bounds a disk $\Delta_{0}$ which is part of the
foam. It is the union $\Delta^{-}_{0} \cup \Delta^{+}_{0}$ which lie
in the cobordisms $t$ and $q$ respectively. The indicated disk
$\Delta_{1}$ is similar, but is not part of the foam, and is included
for reference.

\begin{figure}
    \begin{center}
 \includegraphics[scale=.36]{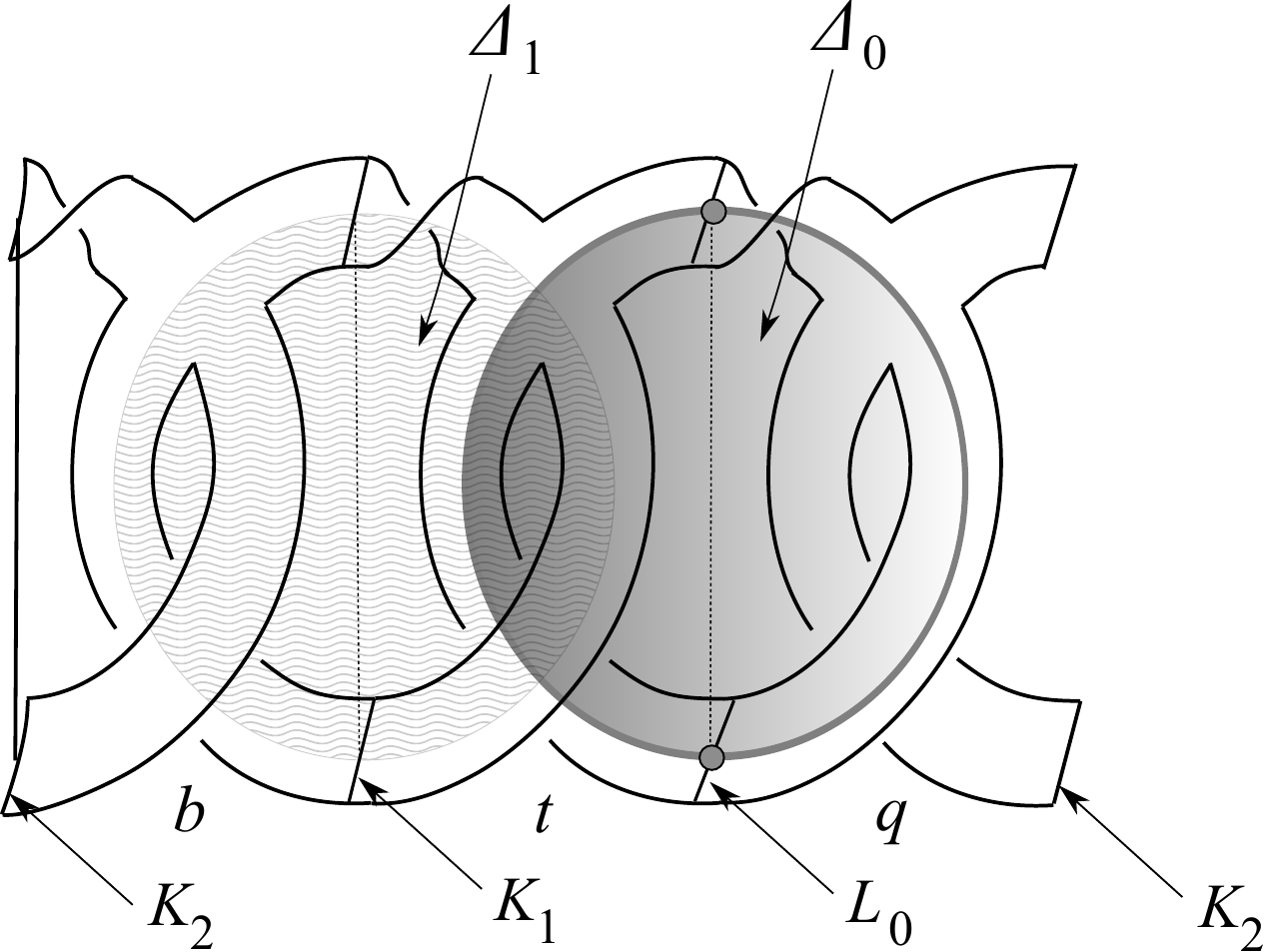}
    \end{center}
    \caption{\label{fig:mobiusshadedwithfeetfoamonedisk} The
    non-trivial part of the composite cobordism $b \cup t \cup q$
    (portrayed schematically because it is not embedded in $\R^{3}$).
    The gray dots are vertices of the web $L_{0}$. The shaded disk
    $\Delta_{0}$ is part of the foam. The hatched disk $\Delta_{1}$ is
    not. The picture continues periodically with period $3$ in both
    directions.}
\end{figure}

\subsection{The vanishing of the composites}

Consider the composite map $q\comp t : \Lsharp(K_{1}) \to
\Lsharp(K_{2})$. It is induced by the cobordism $t \cup q$ in
Figure~\ref{fig:mobiusshadedwithfeetfoamonedisk}. In the interior of
the cobordism, a regular neighborhood of the disk $\Delta_{0}$ is a
4-ball $B_{0}$ which meets the foam in the union of a M\"obius band
the disk $\Delta_{0}$ itself. The boundary $S_{0}=\partial B_{0}$
meets the foam in an unknotted circle, so this describes $t\cup q$ as
a connected sum: one summand is the union of an $\RP^{2}$ and a disk.
The $\RP^{2}$ has self-intersection $+2$, so this summand is a copy of
the foam $\Psi_{1}$ from section~\ref{subsec:RP2-foams}. The fact that
composite map is zero on $\Lsharp$ is therefore a corollary of
Proposition~\ref{prop:face-sums} in the case $n=1$.

For the composite map $t \comp b : \Lsharp(K_{2}) \to \Lsharp(L_{0})$
induced by the cobordism $b\cup t$, the argument is similar. This
time, a regular neighborhood of the disk $\Delta_{1}$ is a ball
$B_{1}$ meeting the foam in the union of a M\"obius band and a
half-disk $\Delta_{0}^{-}$. (The disk $\Delta_{1}$ is not part of the
foam.) This describes $b\cup t$ as a sum where one summand is again
$\Psi_{1}$, but the sum is now formed at a seam point of the foams.
The vanishing of the map on $\Lsharp$ follows now from
Proposition~\ref{prop:seam-sums} (in the case $n=1$ again).

The composite map $b \comp q$ induced by the cobordism $q\cup b$
vanishes for essentially the same reason as $t\comp b$, using the
evident disk $\Delta_{2}$ (not shown in the figure).

\subsection{The first chain homotopies \texorpdfstring{\boldmath
$j_{i}$}{j\_i}}

The proof the vanishing of the composite maps on $\Lsharp$ above, when
picked apart, provides the necessary chain homotopies $j_{i}$. At the
chain level, the map induced by a composite cobordism such as $t\cup
q$ is not the composite of the chain maps, but is chain-homotopic to
it. They become equal when the cobordism is stretched, in this case
along a cylindrical neighborhood of the intermediate bifold $(S^{3},
L_{0})$. So a chain homotopy is provided by counting instantons in
moduli spaces of total dimension $0$ over a 1-parameter family of
metrics parametrized by $(-\infty, 0]$, where the end at $-\infty$ is
the limit where the neck is stretched. The vanishing of the map
$\Lsharp(t\cup q)$ from the connected-sum argument above is also not
at the chain level, but becomes so when the connected sum is stretched
along the sphere $S_{0}=\partial B_{0}$. Joining these two families of
metrics, we obtain a family parametrized by $\R = (-\infty,0] \cup
[0,\infty)$. The chain homotopy $j_{0}$ is defined by counting
instantons in moduli spaces of total dimension $0$ over this
one-parameter family:
\[
          \CLsharp(q)\comp \CLsharp(t) = d j_{0} + j_{0} d.
\]

The construction of the chain homotopy $j_{1}$ uses a similar
1-parameter family of metrics, stretching along $K_{1}$ and
$S_{1}=\partial B_{1}$, and so on with $j_{2}$. 

\subsection{The second chain homotopy  \texorpdfstring{%
$k_{i}$ for $i=2$}{k\_i for i=2}}
\label{subsec:second-chain}

We continue to follow \cite{KM-triangles} closely. We will construct
the chain homotopy $k_{2}$ needed in the formula
\eqref{eq:second-chain-homotopies}. In interior of the triple
composite cobordism $b \cup t \cup q$ in the figure, we can identify
five codimension-1 bifolds. These are:
\begin{itemize}
    \item
the bifolds $K_{1}$ and $L_{0}$ (we use this
notation as short-hand for the corresponding 3-dimensional bifolds);
\item
the bifold $\check S_{1}$ arising from the 3-sphere which is the
boundary of the regular neighborhood of $\Delta_{1}$, whose singular
set is a theta graph (the union of the boundary of a M\"obius band and
half the boundary of the half-disk $\Delta_{0}^{-}$;
\item
the bifold $\check S_{0}$ arising from the 3-sphere which is the
boundary of the regular neighborhood of $\Delta_{0}$, whose singular
set is an unknot;
\item
a bifold $\check S_{10}$ arising from the boundary of the ball $\check
B_{10}$
which is a regular neighborhood of $\Delta_{1}\cup \Delta_{0}$. The
singular set, where $\partial \check B_{10}$ meets the foam,
is a 2-component unlink.
\end{itemize}

If we list these five in a suitable cyclic order,
\[
       K_{1}, \check S_{0}, \check S_{10},  \check S_{1}, L_{0},
\]
then adjacent bifolds are disjoint (and the last is disjoint from the
first). For each such disjoint pair, we form a family of metrics
parametrized by a quadrant $[0,\infty)^{2}$, stretching along both.
These five quadrants have common edges, and their union is the
interior of an open 2-parameter family of metrics that can be
visualized as a pentagon $P$. See \cite{KMOS}, and equation (11) in
\cite{KM-triangles} for example.

The second chain homotopy $k_{2}$, as a chain map from
$\CLsharp(K_{2})$ to $\CLsharp(K_{2}))$, has two parts,
$k_{2}=k'_{2}+k''_{2}$. The first part is defined by counting points
in zero-dimensional moduli spaces over $P$, on the composite cobordism
$(b\cup t\cup q)^{+}$ equipped with cylindrical ends. As in
\cite{KM-triangles}, the terms in the expression
\[
             q j_{0} + j_{2} b + d k'_{2} + k'_{2} d 
\]
have the following interpretation in terms of 1-dimensional moduli
spaces over the same family $P$. The terms $d k'_{2}$ and $k'_{2} d$
count the ends of such 1-dimensional moduli spaces arising from
trajectories sliding off one of the two ends of $(b\cup t\cup q)^{+}$.
The terms $q j_{0}$ and $j_{2} b$ count ends which limit to two of the
5 edges of the pentagon. There are three other edges of in the
compactification of the pentagon $P$, two of which contribute $0$ to
the count of the ends. Since the number of ends is zero mod $2$, we
therefore have a relation at the chain level,
\begin{equation}\label{eq:U}
             q j_{0} + j_{2} b + d k'_{2} + k'_{2} d = U_{2},
\end{equation}
where $U_{2}$ counts the ends of 1-dimensional moduli spaces which
limit to the fifth and final edge of $P$. The key step is to
understand $U_{2}$, and to show that it is chain-homotopic to the
identity. Writing $k''_{2}$ for the latter chain-homotopy, we will
complete our task of constructing $k_{2}$ as $k'_{2}+k''_{2}$.

The edge of $P$ which corresponds to $U_{2}$ is where the cylindrical
neighborhood of $\check S_{10}$ has been stretched to infinity,
pulling out the orbifold 4-ball $\check B_{10}$, which carries a
1-parameter family of metrics $G$. This 1-parameter family is the
union of two half-lines, where $\check B_{10}$ is stretched either
along $\check S_{1}$ or along $\check S_{0}$.

For this one-parameter family of metrics $G$, let $M_{G}$ denote the
parametrized moduli space on the bifold $(\check B_{10})^{+}$ with
cylindrical end $\R^{+}\times \check S_{10}$, and let
\[
            r : M_{G} \to \Rep(\check S_{10})
\]
be the restriction map to the space of flat bifold connections on the
end. The following proposition is the main non-formal ingredient, and
is the counterpart of \cite[Proposition~7.1]{KM-triangles}.

\begin{lemma}\label{lem:push-me-pull-you}
The representation variety $\Rep(\check S_{10})$ is a closed interval,
and for generic choice of perturbations, $M_{G}$ has an open subset
$M^{1}_{G}$ of dimension $1$ consisting of connections with $S^{1}$
stabilizer. Furthermore, the restriction map $r$ maps $M^{1}_{G}$
properly and surjectively to the interior of the interval $\Rep(\check
S_{10})$ with degree $1$ mod $2$. The remainder of $M_{G}$ consists of
components of dimension $6$ or more, together with possibly a finite
number of irreducible solutions mapping to the interior of the
interval.
\end{lemma}

\begin{figure}
    \begin{center}
        \includegraphics[scale=.30]{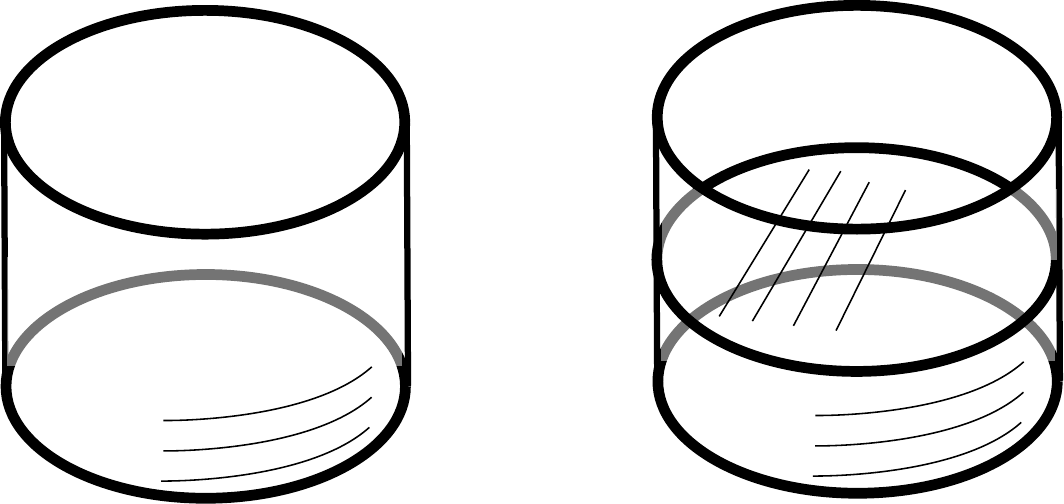}
    \end{center}
    \caption{\label{fig:foams-for-K-SU3} Two isomorphic foams $W'_{2}$
    (left) and $W''_{2}$ (right), each with boundary the unlink
    $U_{2}$, resulting from the two decompositions of $W$.}
\end{figure}

\begin{proof}[Proof of the Lemma]
    The orbifold $\check S_{10}$ is the 3-sphere with singular set a
    2-compo\-nent unlink. The bifold fundamental group is a free
    product of two cyclic groups of order $2$, and an element of
    $\Rep(\check S_{10})$ assigns to each generator an element of
    $\CP^{2}$, to be considered up to the action of $\PSU(3)$. The one
    invariant is the distance between the points on $\CP^{2}$. So the
    representation variety is a closed interval, which we choose to
    write as
    \begin{equation}\label{eq:Rep-interval}
           \Rep(\check
    S_{10}) = [0,\pi/2]
    \end{equation}
    for consistency with \cite{KM-triangles}.

    To describe the singular set of the foam $\Sigma$ which forms the
    singular set of $\check B_{10}$, we start with the foams
    $\Psi_{1}\subset \Psi_{2}$ of \eqref{eq:Psi-n-foam}, but renaming
    the disks for the current context, so
    \[
\begin{aligned}
     \Psi_{2} &= R\cup \Delta_{1} \cup \Delta_{0} \\
     \Psi_{1} &= R\cup \Delta_{0} .
\end{aligned}
    \]
    The boundaries of the two disks divide $R$ into two connected
    components. Let $x_{0}$ and $x_{1}$ belong to these two connected
    components, let $a$ be an arc joining $x_{0}$ to $x_{1}$ which is
    otherwise disjoint from $\Psi_{2}$, and let $\beta$ be a regular
    neighborhood of $a$. So $\beta$ is a 4-ball whose boundary meets
    $\Psi_{2}$ in an unlink $U_{2}$. The complement $\beta^{c}$ is
    also a 4-ball, and the bifold $\check B_{10}$ can be described as
    the pair
    \begin{equation}\label{eq:B10}
            \check B_{10} = (\beta^{c}, \beta^{c} \cap \Psi_{1})
    \end{equation}
    whose singular set is a twice-punctured copy of $\Psi_{1}$, which
    we write as
    \begin{equation}\label{eq:W2-foam}
        W_{2} = \beta^{c} \cap \Psi_{1}.
    \end{equation}
    This
    description also displays the spheres $S_{1}$ and $S_{0}$ which
    are the boundaries of regular neighborhoods of $\Delta_{1}$ and
    $\Delta_{0}$, though only the latter disk is part of the foam.

     The two limit points of the $1$-parameter family of metrics $G$
     correspond to pulling out a neighborhood of either $\Delta_{1}$
     or $\Delta_{0}$ from the bifold $\check B_{10}$. As in the proof
     of the vanishing of the composites, this is a sum decomposition
     of the foam $W_{2}$, in which the summand that is being pulled
     off is a copy of $\Psi_{1}$ and the sum is either at facet (in
     the case of $\Delta_{0}$) or a seam (in the case of
     $\Delta_{1}$). So we have decompositions,
     \begin{equation}\label{eq:push-me-pull-you}
     \begin{aligned}
         W_{2} &= \Psi_{1} \#_{f,f'} W'_{2} \\
                  W_{2} &= \Psi_{1} \#_{s,s'} W''_{2} \\
     \end{aligned}
     \end{equation}
     corresponding to pulling out a neighborhood of $\Delta_{0}$ or
     $\Delta_{1}$ respectively. The foams $W'_{2}$ and $W''_{2}$ are
     both easy to describe. The former is an annulus standardly
     embedded in the 4-ball with boundary the unlink $U_{2}$. The
     latter is the union of an annulus and a disk whose boundary lies
     on the interior of the annulus. See
     Figure~\ref{fig:foams-for-K-SU3}. 

     The sum decompositions~\eqref{eq:push-me-pull-you} allow
     descriptions of the ends of the moduli space $M_{G}$ over the two
     ends of the family of metrics $G$. Consider first the case
     $\Psi_{1} \#_{f,f'} W'_{2}$. The representation variety for
     $W'_{2}$ consists of a single point with stabilizer $U(2)$, and
     in the description \eqref{eq:Rep-interval}, the image of this
     single point under $r$ is the endpoint $0$. (The monodromy of the
     flat bifold connection is the same at the two boundary
     components.) The smallest non-empty moduli space on $\Psi_{1}$
     has $\kappa=1/8$, and is a single point with stabilizer $S^{1}$
     as stated in Lemma~\ref{lem:smallest-action-R}. We consider
     gluing this to the flat connection using a metric $g$ near the
     end of $G$ (so with a fixed, large parameter for the length of
     the neck). The gluing is unobstructed, and there is no effective
     gluing parameter because of the $U(2)$ stabilizer. It follows
     that the moduli space $M_{g,1/8}$ of solutions with $\kappa=1/8$
     is a single point when $g$ is close to this end. This single
     point is a solution with $S^{1}$ stabilizer.

     The analysis of the other end of the family $G$, corresponding to
     the decomposition $W_{2} = \Psi_{1} \#_{s,s'} W''_{2}$, is
     similar. The moduli space of flat connections for the foam
     $W''_{2}$ is again a single point, this time a connection with
     stabilizer $S^{1}$. Under $r$, it maps to the end $\pi/2$ of the
     interval $[0,\pi/2]$ because the monodromies of the connection
     around links of the two boundary components determine orthogonal
     points in $\CP^{2}$. The smallest-action moduli space with
     $\kappa=1/8$ for $\Psi_{1}$ is summed at a point on seam, which
     means the gluing parameter is $S^{1}$. But the $S^{1}$ stabilizer
     results in there being no effective gluing parameter. The moduli
     space $M_{g,1/8}$ for $g$ close to this end is again a single
     point.

     The formal dimension of $M_{g,1/8}$ is $-1$ in both the cases
     above, because of the $S^{1}$ stabilizer. The formal dimension of
     the parametrized moduli space $M_{G,1/8}$ is therefore $0$. We
     have seen that it contains a 1-dimensional subset consisting of
     $S^{1}$ reducibles, and it may perhaps contain isolated
     irreducibles also. We conclude that there is a 1-dimensional part
     $M^{1}_{G}$ consisting of solutions with $S^{1}$ stabilizer and
     that it maps to $G$ in way that is a diffeomorphism near the two
     ends. The moduli space consists of solutions with $\kappa=1/8$,
     which means there can be no bubbles, so $M^{1}_{G}$ is proper
     over $G$, and therefore has exactly two ends. We have seen that
     $r$ maps the two ends to $0$ and $\pi/2$. If the action is bigger
     than $1/8$, then the difference is at least $1/2$, leading to
     other components of formal dimension $6$ or more.
 \end{proof}

We now return to the chain map $U$ in \eqref{eq:U}. Recall that we aim
to show that $U$ is chain-homotopic to the identity. Let $Z$ be the
complement of $\check B_{10}$ in $b\cup t \cup q$, and let $Z^{+}$ be
$Z$ equipped with cylindrical ends on the two copies of $K_{2}$ and an
additional cylindrical end on $\check S_{10} = \partial \check
B_{10}$. Lemma~\ref{lem:push-me-pull-you} provides a gluing
interpretation of $U$ as the count of isolated solutions of $Z^{+}$.
The orbifold $\check S_{10}$ on the boundary of $Z$ is has singular
set the 2-component unlink in $S^{3}$, so is the boundary of the
orbifold $\check D$ whose singular set is 2-disks in a 4-ball. By
gluing, we see that the count of solutions on $Z^{+}$ is also equal to
the count of solutions on the cobordism $Z\cup\check D$, from $K_{2}$
to $K_{2}$, when the neck is stretched along $\check S_{10}$. However,
$Z\cup\check D$ is diffeomorphic to the product cobordism. So the
count of solutions, with any choice of metric, is a chain map which is
chain-homotopic to $1$. It follows that $U\sim 1$. This completes the
construction of the second chain homotopy $k_{2}$ in
\eqref{eq:second-chain-homotopies} and proof that the corresponding
chain map \eqref{eq:second-CH-fmla} is an isomorphism on
$\CLsharp(K_{2})$.

\subsection{The second chain homotopy  \texorpdfstring{%
$k_{i}$ for $i=0,1$}{k\_i for i=0, 1}}

The construction of the other two chain homotopies $k_{0}$ and $k_{1}$
in \eqref{eq:second-chain-homotopies}, and the verification that the
maps in \eqref{eq:second-CH-fmla} are isomorphisms, follow a similar
pattern to $k_{2}$. Indeed, the case of $k_{1}$ is essentially
identical. We briefly indicate the key step in the verification for
$k_{0}$.

We construct as before a family of Riemannian metrics parametrized by
a pentagon $P$, and we have a chain-homotopy formula
\begin{equation}\label{eq:U0}
             t j_{1} + j_{0} q + d k'_{0} + k'_{0} d = U_{0}
\end{equation}
in which the three terms $t j_{1}$, $j_{0} q$ and $U_{0}$ arise from
counting isolated solutions in 1-parameter families of metrics
corresponding to three of the five boundary edges of $P$. What needs
to be done is to show that $U_{0}$ is chain-homotopic to the identity
on $\CLsharp(L_{0})$.

Adapting the construction from the $i=2$ case in the natural way, we
now consider the orbifold ball $\check B_{21}$ that is the regular
neighborhood of the union of the two disks $\Delta_{2}$ and
$\Delta_{1}$. In Figure~\ref{fig:mobiusshadedwithfeetfoamonedisk}, the
disk $\Delta_{2}$ is not shown but lies to the left in the figure.
Recall that the disk $\Delta_{0}$ is part of the foam, but
$\Delta_{1}$ is not. The boundary $\check S_{21}=\partial \check
B_{21}$ is an orbifold corresponding to a web $J\subset S^{3}$. In the
case $i=2$, the corresponding web was the unlink $U_{2}$, but now $J$
(depicted in Figure~\ref{fig:web-for-k0} has two extra edges, labeled
$a$ and $b$: the edge $b$ is where the sphere $\partial B_{21}$ meets
the disk $\Delta_{0}$ in the figure, and the arc $a$ is where
$\partial B_{21}$ meets the translate of $\Delta_{0}$ that is three
steps to the left.

On $\check B_{21}$ and the corresponding cylindrical-end bifold, there
is again an open $1$-parameter family $G$ of Riemannian metrics whose
ends correspond to pulling out a regular neighborhood of either
$\Delta_{2}$ or $\Delta_{2}$. We use the same notation $M_{G}$ as
before for the moduli space of solutions on the cylindrical-end
manifold, over this family $G$. The key step in showing that $U_{2}$
was chain homotopic to the identity in the $i=2$ was
Lemma~\ref{lem:push-me-pull-you}. The following lemma adapts this for
$U_{0}$, and completes the argument.

\begin{figure}
    \begin{center}
        \includegraphics[scale=.30]{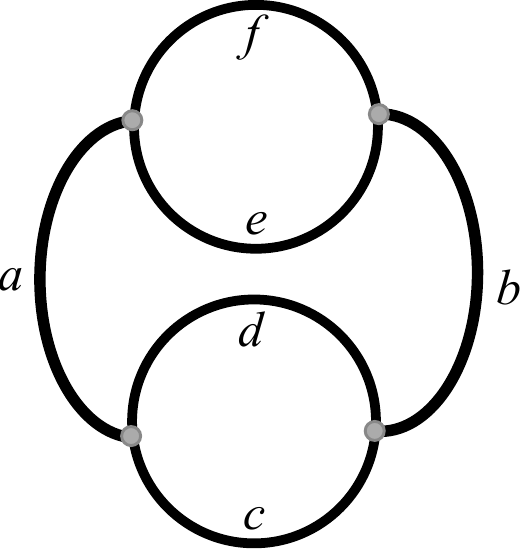}
    \end{center}
    \caption{\label{fig:web-for-k0}
    The web $J$ arising as the boundary of $\check B_{21}$, for the chain
    homotopy $k_{0}$.}
\end{figure}

\begin{figure}
    \begin{center}
        \includegraphics[scale=.30]{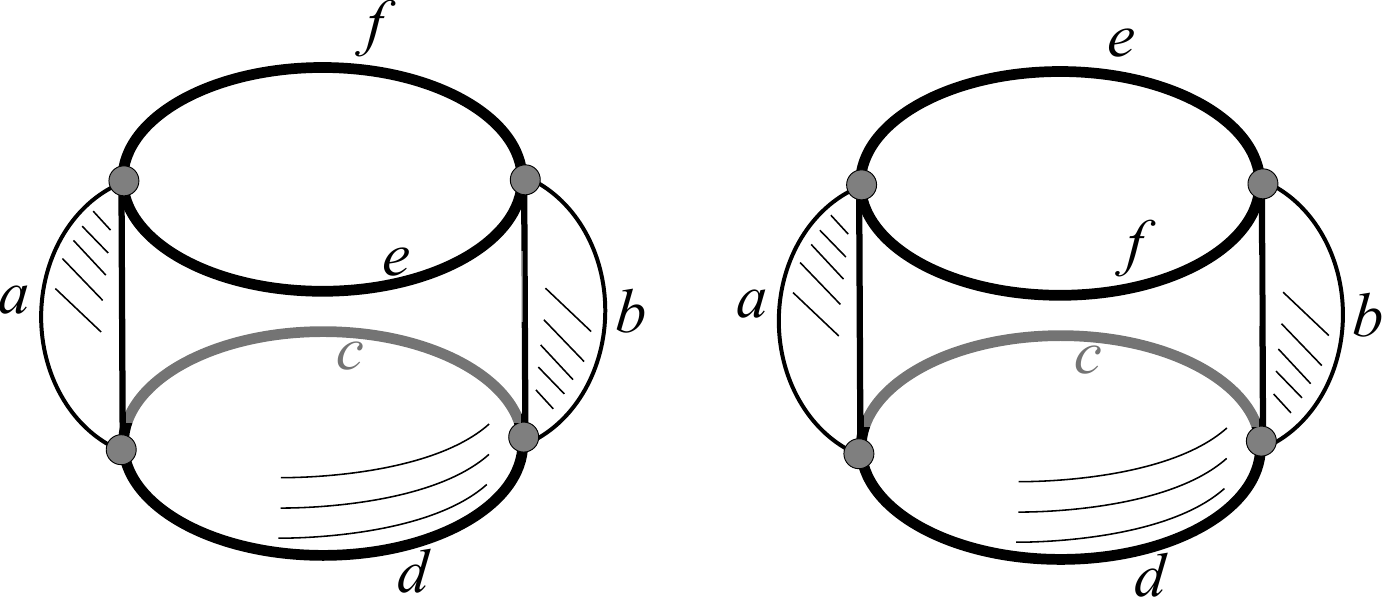}
    \end{center}
    \caption{\label{fig:foams-for-K-SU3-2} The foams $W'_{0}$ and
    $W''_{0}$ from the verification of the chain homotopy formula for
    $k_{0}$. They are isomorphic foams, and both have boundary $J$,
    but the identifications with $J$ differ in the two cases.}
\end{figure}

\begin{lemma}\label{lem:push-me-pull-you-k0}
The statement of Lemma~\ref{lem:push-me-pull-you} continues to hold
verbatim with $\check S_{21} = \partial \check B_{21}$ replacing
$\check S_{10}=\partial B_{10}$ from the previous version.
\end{lemma}

\begin{proof}
The first assertion is that the representation variety of the orbifold
$\check S_{21} = (S^{3}, J)$ is a closed interval. The web $J$ is
shown in Figure~\ref{fig:web-for-k0}. We can describe a flat bifold
connection by specifying a point in $\CP^{2}$ for each edge, with the
constraint that these points be orthogonal when two edges meet at a
vertex. Up to the action of $\PU(3)$, the edges $c$ and $d$ must be
assigned the first two basis vectors $p_{1}$ and $p_{2}$ in $\CP^{2}$
while the edges $a$ and $b$ are assigned the third basis vector
$p_{3}$. The edge $e$ is assigned a point $q$ in the projective line
orthogonal to $p_{3}$, and the edge $f$ is assigned the point $q'$ on
the same line and orthogonal to $q$. Using the remaining symmetry in
the picture, we can take $q$ to lie on a chosen closed geodesic
joining $p_{1}$ to $p_{2}$. The representation variety $\Rep(S^{3},
J)$ is in bijection with this geodesic.

The bifold $\check B_{21}$ has a description parallel to the
\eqref{eq:B10}. It is again the complement of an arc in the bifold
$(B^{4}, \Psi_{1})$, but this time the relevant arc $\gamma$ joins two
points on the seam of $\Psi_{1}$. The interior of $\gamma$ is disjoint
from $\Psi_{1}$, and we have
   \begin{equation}\label{eq:B21}
            \check B_{21} = (\gamma^{c}, \gamma^{c} \cap \Psi_{1}).
    \end{equation}
So the singular set is $\Psi_{1}$ with the neighborhood of two seam
points removed, which we write as
    \begin{equation}\label{eq:W0-foam}
        W_{0} = \gamma^{c} \cap \Psi_{1}.
    \end{equation}
     The limit points of the $1$-parameter family of metrics $G$ now
     correspond to pulling out a neighborhood of either $\Delta_{2}$
     or $\Delta_{1}$ and give rise to two sum decompositions of the
     foam $W_{0}$:
     \begin{equation}\label{eq:push-me-pull-you-2}
     \begin{aligned}
         W_{0} &= \Psi_{1} \#_{s,s'} W'_{0} \\
                  W_{0} &= \Psi_{1} \#_{s,s'} W''_{0} .\\
     \end{aligned}
     \end{equation}
     In both cases, the sum is made at a seam. The foams $W'_{0}$ and
     $W''_{0}$ are shown in Figure~\ref{fig:foams-for-K-SU3-2}. They
     are isomorphic foams, but the isomorphism is not the identity on
     the boundary: it interchanges the two edges $e$ and $f$ of the
     web $J$ as shown in the figure. The moduli space of flat
     connections on $(B^{4}, W'_{0})$ and $(B^{4}, W''_{0})$ each
     consist of a single point; but in our description of the
     representation variety $\Rep(\check S_{21})$ as a closed
     interval, these two flat connections map to the opposite ends of
     the interval. With this understood, the description of the moduli
     space $M_{G}$ for $\check B_{21}$ can be completed as before,
     using gluing and the decompositions
     \eqref{eq:push-me-pull-you-2}, completing the proof of the lemma.
\end{proof}

This lemma, combined with the previous arguments from the case $i=2$,
establishes the chain-homotopies $k_{i}$ for all $i$, with the desired
property, that the chain maps \eqref{eq:second-CH-fmla} are
isomorphisms. This completes the proof of the exactness of the
triangles.
\end{proof}

\section{The edge decomposition and planar webs}
\label{sec:edge-decomp}

\subsection{The edge decomposition}
\label{subsec:edge}

Recall that to each edge $e$ of web $K$ we have associated an operator
\[ u_{e} = \sigma_{1}(e) \] acting on $\Lsharp(K)$, and that the
relation
\[
          u_{e}(u_{e}^{2}+1) =0
\]
holds (Lemma~\ref{lem:Xie-rel-mod-2}. Since the two factors $u$ and
$u^{2}+1$ are coprime, we have a decomposition into generalized
eigenspaces for the eigenvalues $0$ and $1$:
\[
\begin{aligned}
     \Lsharp(K) &= \ker(u_{e}) \oplus \ker
     (u_{e}^{2} + 1) \\
       &= \ker (u_{e}) \oplus \im (u_{e}).
\end{aligned}
\]
The situation here is very similar to what happens in the $\SO(3)$
theory when $\Jsharp$ is deformed using a local coefficient system on
the configurations space (see \cite[section 5.1]{KM-deformation}), and
we can pursue the consequences of this decomposition in just the same
way.

The operators $u_{e}$ all commute, so there is a decomposition of
$\Lsharp(K)$ into generalized eigenspaces. We write
\[
            \Edges(K) = \{\text{edges of $K$}\},
\] 
and given a subset $s\subset
\Edges(K)$, we define
\[
        V(K;s) = \left(\bigcap_{e\in s} \ker(u_{e})\right) 
                                \cap \left(\bigcap_{e\notin s}
                                \im(u_{e})\right) ,
\]
and we have a decomposition of $\Lsharp(K)$ which we call the
\emph{edge decomposition}:
\begin{equation}
    \label{eq:espace-decomp}
    \Lsharp(K) = \bigoplus_{s\subset \Edges(K)} V(K;s).
\end{equation}

Recall that a subset $s\subset \Edges(K)$ is \emph{$1$-set}, or a
\emph{perfect matching} if each vertex of $K$ is incident to exactly
one element of $s$. A $2$-set is any collection of edges whose
complement is a $1$-set.

\begin{lemma}\label{lem:only-1-sets}
    The summand $V(K;s) \subset \Lsharp(K)$ is zero if\/ $s$ is not a
    $1$-set.
\end{lemma}

\begin{proof}
    Let $u_{1}$, $u_{2}$, $u_{3}$ be the operators on $\Lsharp(K)$
    arising from three edges incident at a single vertex. (We allow
    that these edges may not be distinct if $K$ has a loop.) Let
    $\lambda_{1}, \lambda_{2}, \lambda_{3}$ each be either $0$ or $1$,
    and let $V\subset\Lsharp(K)$ be the simultaneous generalized
    eigenspace for $u_{i}$ with eigenvalue $\lambda_{i}$ ($i=1,2,3$).
    The lemma is equivalent to the assertion that $V$ is zero unless
    $(\lambda_{1},\lambda_{2}, \lambda_{3})$ is $(1,1,0)$ or a
    permutation thereof. But Lemma~\ref{lem:dot-migration} tells us
    that, if $V$ is non-zero, we have
    \[
    \begin{aligned}
    \lambda_{1}+\lambda_{2}+\lambda_{3}&=0 \\
    \lambda_{1}\lambda_{2} + \lambda_{2}\lambda_{3} +
    \lambda_{3}\lambda_{1}&=1,
    \end{aligned}
    \]
    and these relations leave no other possibilities open.
\end{proof}

\begin{corollary}\label{cor:edge-decomp}
    There is a direct sum decomposition of $\Lsharp(K)$ as
    \[
            \Lsharp(K) = \bigoplus_{\text{\normalfont $1$-sets $s$}}
            V(K;s).
    \]\qed
\end{corollary}

\begin{examples}
    For the unknot $K$ with its single edge $e$, there are two 1-sets
    $s$, namely $\{e\}$ and $\{\emptyset\}$. The corresponding
    summands $V(K;s)$ have dimension $1$ and $2$ respectively, these
    being the kernel and image of the rank-2 endomorphism $u_{e}$. For
    the theta graph, there are three $1$-sets, each consisting of a
    single edge, and the corresponding summands each have rank $2$.
    These facts follow from Proposition~\ref{prop:unknot} and
    Proposition~\ref{prop:theta-web}.
\end{examples}

\subsection{Planar webs and Tait colorings}

We now turn to the calculation of the dimension of $\Lsharp(K)$ in the
case that $K$ is planar. This is the content of
Theorem~\ref{thm:planar-tait} in the introduction.

Given a Tait coloring of a web $K$, the edges of each color are
$1$-sets, and the edges of any two distinct colors are a $2$-set,
comprising therefore a collection of disjoint simple closed curves,
$C$. Furthermore, a $1$-sets that arises in this way from a particular
color in a Tait coloring is always an \emph{even} $1$-set. Here
``even'' means that the number of vertices in each connected component
of the complementary $2$-set is even. Equivalently, it means that the
relative $\Z/2$ homology class which the $1$-set defines in $H_{1}(Y,
C; \Z/2)$ is zero. Given an even $1$-set $s$, we can look for all Tait
colorings of $K$ for which the edges of the first color are $s$. If
$s$ is even, the number of such Tait colorings is $2^{n(s)}$ where
$n(s)$ is the number of connected components of the 2-set. If $s$ is
odd, the number of Tait colorings is zero. So the number of Tait
colorings of $K$ is
\[
     \sum_{s \in \{\text{\normalfont even $1$-sets}\}} 2^{n(s)}.
\]
To prove Theorem~\ref{thm:planar-tait}, it is therefore enough to
establish this same formula for the dimension of $\Lsharp(K)$ in the
case that $K$ is planar. The required formula follows immediately from
Corollary~\ref{cor:edge-decomp} and the following proposition (which
is the present counterpart of \cite[Proposition~5.17]{KM-deformation}.

\begin{proposition}
    Let $K$ be a planar web and let $s$ be an even $1$-set. Then $V(K;
    s)$ dimension $2^{n}$, where $n$ is the number of
    components in the complementary $2$-set $C$. If $s$ is not even,
    then $V(K;s)=0$. 
\end{proposition}

\begin{proof}
    Consider first the case that $s$ is empty, so that $C$ is a planar
    unlink with $n$ components. We wish to see that $V(C;\emptyset)$
    has dimension $2^{n}$. By the excision result,
    Proposition~\ref{prop:excision}, and the naturality of the
    operators $u_{e}$ with respect to the excision isomorphism, we
    have
    \begin{equation}\label{eq:Vs-2n}
         V(C;\emptyset)
         \cong \bigotimes_{1}^{n} V(C_{i};\emptyset),
    \end{equation}
    where each $C_{i}$ is an unknot. The formula for the dimension,
    $2^{n}$, follows from this, since we know that a single unknot
    has $\dim V(C_{i};\emptyset) =2$.

    Next we have the following result, which applies not just to
    planar webs.

    \begin{lemma}\label{lem:modify-s}
        Let $(Y, K)$ be a web in a 3-manifold $Y$ and let $s$ be a
        1-set for $K$. Write $K=C\cup s$, where $C$ is the link formed
        by the $2$-set. Let $K'\subset Y$ be another web differing
        only in the $1$-set: so $K'=C\cup s'$. Suppose that $s$ and
        $s'$ have the same homology class in $H_{1}(Y, C; \Z/2)$. Then
        \[
            V(K;s) \cong V(K'; s').
        \]
    \end{lemma}

    Assuming the lemma for the moment, we complete the proof of the
    proposition (and hence of Theorem~\ref{thm:planar-tait} also). If
    $s$ is even, then we can apply the lemma with $s'=\emptyset$, and
    reduce to the calculation \eqref{eq:Vs-2n} just above, to
    establish the proposition in the even case. If $s$ is odd, then we
    can use the lemma to reduce to the case that at least one
    component of $C$ is incident with exactly one edge of $s'$. For
    such a web, the representation variety is empty (there is a
    ``embedded bridge''), so $\Lsharp(K')$ is zero in such a case, as
    is the subspace $V(K';s')$ and hence $V(K;s)$ as claimed. We turn
    to the proof of the Lemma next.
\end{proof}

\begin{proof}[Proof of Lemma~\ref{lem:modify-s}]
If $s$ and $s'$ are homologous in $H_{1}(Y,C;\Z/2)$, then $s$ and $s'$
are related to each other by isotopies combined with a sequence of
modifications of the following elementary types.
\begin{enumerate}
    \item \label{it:s-birth-death}
    $s'$ is obtained from $s$ by the birth of a single unknotted
    circle in a ball disjoint from $K$, or by the death of such a
    circle;
    \item \label{it:s-surger} $s'$ is obtained from $s$ by surgery in ball, just as
    $K_{0}$ and $K_{1}$ are related to each other in
    Figure~\ref{fig:L-octahedron};
    \item \label{it:s-bigon}
    $s'$ is obtained from $s$ by the addition or removal of a single
    edge inside a ball which meets $C$ in an arc, as illustrated in
    Figure~\ref{fig:s-bigon}.
    \end{enumerate}    
We will show that a single modification of any of these sorts leaves
$V(K;s)$ unchanged. In the first case, the addition of a single
unknotted circle results in tensor product (by excision) where the new
factor is $V(U_{1}, \{e\})$, where $U_{1}$ is the unknotted circle and
$e$ its edge. From our calculation for the unknot, we know that
$V(U_{1}, \{e\})$ is $1$-dimensional, so this established case
\ref{it:s-birth-death}.

\begin{figure}
    \begin{center}
        \includegraphics[scale=0.50]{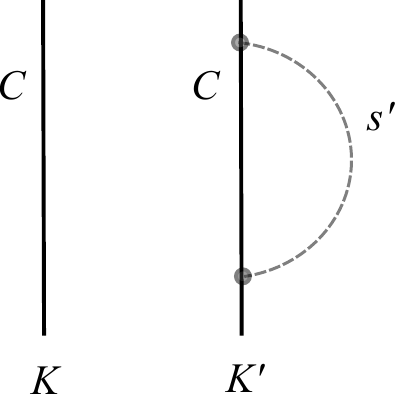}
    \end{center}
    \caption{\label{fig:s-bigon} Changing the $1$-set by adding an
    extra edge, to obtain a new web $K'$ from $K$. The subset $C$ is
    the union of the edges of a $2$-set, and is the same in both
    cases.}
\end{figure}

In case \ref{it:s-surger}, $V(K;s)$ and $V(K';s')$ appear in an exact
triangle which is the front face of the octahedron
\ref{fig:L-octahedron} and in which the third web is $L_{2}$ in the
figure. The maps in the exact triangle commute with the operators
$u_{e}$ for any edge that extends past the boundary of the ball. So
there is an exact triangle involving $V(K_{0}; s)$ and $V(K_{1}; s')$
and in which the third term is a summand of $\Lsharp(L_{2})$ which is
contained in the simultaneous kernel of the four operators $u_{e}$, as
$e$ runs through the four edges of $L_{2}$ meeting the boundary of the
ball in the figure. However, because of the additional edge inside the
ball, there is no $1$-set of $L_{2}$ which includes these four edges.
The corresponding summand of $\Lsharp(L_{2})$ is therefore zero, so
$V(K_{0}; s)$ and $V(K_{1}; s')$ are isomorphic.

Finally we consider the case illustrated in Figure~\ref{fig:s-bigon}.
In this case, we redraw $K'$ as shown in
Figure~\ref{fig:bigon-1-set-triangle}. There it appears as one term in
an exact triangle. The other two webs in the triangle are: first, the
web $K''$ which is the disjoint union of $K$ and an unknotted circle;
and second, a web $L$ isotopic to $K$. Let $s''\subset K''$ be the
1-set formed by $s$ and the additional unknotted circle. We have
$V(K;s) \cong V(K'';s'')$ as an instance of case
\item{it:s-birth-death}. As summands in the triangle, we have an exact
triangle in which two terms are $V(K';s')$ and $V(K'';s'')$ and in
which the third term is a summand of $\Lsharp(L)$ comprising elements
which are in the $\lambda=1$ generalized eigenspace of the operators
$u_{1}$, $u_{2}$ located at the points $p_{1}$, $p_{2}$ and also in
the $\lambda=0$ eigenspace for the operator $p_{3}$. However, in $L$,
unlike in $K'$, these three points lie on the same edge and define the
same operator on $\Lsharp(L)$. So this summand of $\Lsharp(L)$ is
zero, and the exact triangle gives an isomorphism between $V(K';s')$
and $V(K'';s'')\cong V(K,s)$. This completes the proof of the lemma.
\end{proof}

\begin{figure}
    \begin{center}
        \includegraphics[scale=0.50]{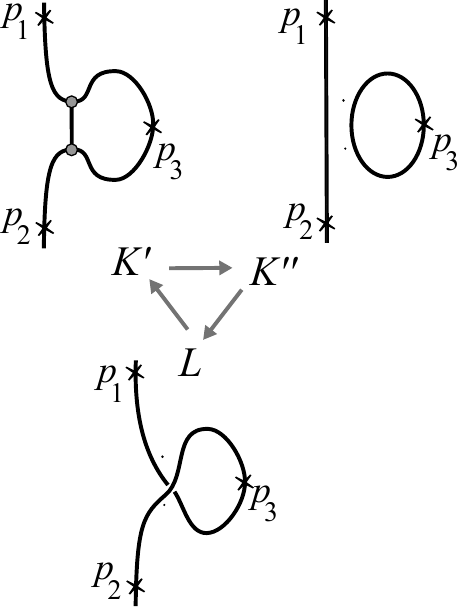}
    \end{center}
    \caption{\label{fig:bigon-1-set-triangle}
   The exact triangle for the proof of case \ref{it:s-birth-death}.
   The web $K''$ is the union of $K$ and an extra unknotted circle.}
\end{figure}

\section{Absolute \texorpdfstring{$\Z/2$}{Z/2} gradings}
\label{sec:absolute}

\subsection{Framings and mod 2 gradings}

For a web $K$ in $Y$, we have seen that $\Lsharp(Y,K)$ has a
relative $\Z/2$ grading. We wish to see what extra data is needed to
specify an absolute $\Z/2$ grading, so as to make $\Lsharp(Y,K)$ a
$\Z/2$-graded vector space.

To begin, let a perturbation be chosen so that we have an instanton
Floer complex $\CLsharp(Y,K)$, and let $\alpha$ be a generator: a
critical point of the perturbed Chern-Simons functional. Choose a
cobordism $(X,\Sigma)$ from $(S^{3},\emptyset)$ to $(Y,K)$ together
with a path of basepoints where the atom is to be attached, and
consider the moduli space which we will simply denote
$M(X,\Sigma,\alpha)$ on the cylindrical-end manifold, for solutions
asymptotic to $\alpha$ on $(Y,K)^{\sharp}$ and to the unique critical
point on $(S^{3})^{\sharp}$. Write $\delta$ for the formal dimension
mod $2$:
\[
   \delta(X,\Sigma,\alpha) = \dim M(X,\Sigma,\alpha) \bmod 2.
\]

So far, we have a quantity that depends on both $\alpha$ and the
choice of $(X,\Sigma)$. Corollary~\ref{cor:dimension-parity} tells
us how the dimension mod $2$ depends on $\Sigma$ through its
self-intersection number in the case of a closed foam. To define a
self-intersection number in the case of foam $\Sigma$ with boundary a
web $K$, we need a suitable notion of a ``framing'' for $K\subset Y$.

The correct notion of a
framing for a web can be read off from the definition of
$\Sigma\ccdot\Sigma$ for closed foams
\cite[Definition~2.5]{KM-Tait}. Recall that an embedded web
$K\subset Y$ has the property that the tangent directions to the three
edges at a vertex are distinct. Since the space of triangles on
$S^{2}$ deformation-retracts onto the space of equilateral triangles
on great circles, we can equally well require that the tangent
directions lie in a 2-plane in addition to being distinct. We will
impose this stronger restriction forthwith. At each
vertex $v$ of $K$, we therefore have a distinguished (unoriented)
normal line,
\begin{equation}\label{eq:Wv}
                W_{v} \subset T_{v}Y,
\end{equation}
perpendicular to the tangents to all three edges (depending on an
unimportant choice of Riemannian metric). For foams in a
$4$-manifold $X$, we similarly impose the condition that at all 
points $p$ of a seam, the tangent 2-planes to the three incident facets
are distinct and lie in a 3-dimensional subspace of the 4-dimensional
tangent space:
\begin{equation}\label{eq:Wp}
                   W_{p} \subset T_{p}X.
\end{equation}
The normals are automatically compatible at the tetrahedral points
where seams meet, so the lines $W_{p}$ define a line subbundle of $TX$
over the
graph formed by the seams. The restriction of this line subbundle to
the boundary $Y$ is the collection of lines \eqref{eq:Wv} at the
vertices of the web, provided that the foam is orthogonal to the
boundary.

\begin{definition}\label{def:K-framing}
    A \emph{semi-framing} $\phi$ of the web $K$ is a choice of a line
    subbundle $W=W_{\phi}\subset TY|_{K}$ which is normal to $K$ along
    every edge and (consequently) coincides with $W_{v}$ at each
    vertex $v$.
\end{definition}

We emphasize that the requirement that $W$ coincide with the normal
line $W_{v}$ at the vertices means that the choice of $W$ is a choice
along the edges only. Note that there is no orientability condition on
$W$. So for example if $K$ is a circle in an oriented 3-ball then the
semi-framings are naturally indexed by $(1/2)\Z$ up to isotopy, with
the integer semi-framings corresponding to the orientable subbundles
$W$. For a general web $K$ with edge set $E$, there is a transitive
action of $((1/2)\Z)^{E}$ on the set of isotopy classes of
semi-framings. If $\phi$ and $\phi'$ are two semi-framings then there
is total difference,
\[
        \Delta(\phi',\phi) \in (1/2)\Z,
\]
defined by summing over the edges.

\begin{definition}
    We say that semi-framings $\phi'$ and $\phi$ belong to the same
    \emph{parity class} if $\Delta(\phi',\phi)$ is an integer.
\end{definition}

We can specify a parity class by providing data at the vertices.

\begin{definition}
    Let an orientation $o_{v}$ of the normal line $W_{v}$ be given at each
    vertex $v$ of $K$. Then a semi-framing $\phi$ of $K$ is
\emph{consonant} with the orientations if the line bundle
    $W_{\phi}$ is orientable in such a way that the orientation agrees
    with $o_{v}$ at each vertex.
\end{definition}

Being consonant with the orientations $o_{v}$ determines $\phi$ up to
the action of $\Z^{E}$, so we have:

\begin{lemma}\label{lem:cons-parity}
    If two orientations $\phi'$ and $\phi$ are both consonant with
    given orientations $o_{v}$ at the vertices, then $\phi$ and
    $\phi'$ belong to the same parity class. If $\phi$ is consonant
    with orientations $o_{v}$ and $\phi'$ is consonant with
    orientations which differ at exactly one vertex, then $\phi$ and
    $\phi'$ belong to different parity classes.
\end{lemma}

\begin{proof}
    Only the last part needs comment. Changing $o_{v}$ at one vertex
    $v$ requires changing $\phi$ by a half-integer along each of the
    three edges incident at $v$. The parity class changes because
    three is odd.
\end{proof}

Returning now to a foam $\Sigma$ with boundary $K$, we can define a
relative self-intersection number of $\Sigma$ with a semi-framing
$\phi$ on the boundary, as follows. Rephrasing the definition from
\cite{KM-Tait} slightly, this self-intersection number is the
obstruction to extending $W_{\phi}$ to all of $\Sigma$. More
precisely, let us first extend $W_{\phi}$ (uniquely) along the seams
of $\Sigma$ as the unoriented common normal line to the incident
facets, as above \eqref{eq:Wp}. Then let us remove a disk $D_{f}$ from
the interior of each facet $f$ of $\Sigma$, and let us extend $W$ to
the interiors of the facets minus the disks $D_{f}$. On the boundary
of each disk $D_{f}$, there is an obstruction to extending $W$ across
the disk, which is a half integer $q_{f}(W)$ in line with the remarks
above. We define the relative self-intersection number as
\begin{equation}\label{eq:Q-Sigma}
            Q(\Sigma, \phi) = \sum_{f} q_{f}(W) \in (1/2)\Z.
\end{equation}
From the definition, it follows that the parity of the integer $2
Q(\Sigma,\phi)$ depends only on the parity class of $\phi$.

\begin{definition}
    Let $(Y,K)$ be given and a choice of semi-framing $\phi$. After a
    choice of generic perturbation, let $\alpha$ be a critical point,
    a generator for $\CLsharp(Y,K)$. We then define the absolute
    $\Z/2$ grading of $\alpha$ with respect to this choice of $\phi$
    to be
    \[
             \gr_{\phi}(\alpha) =   \delta(X,\Sigma,\alpha) + 2
             Q(\Sigma,\phi) \bmod 2.
    \]
    Here $(X,\Sigma)$ is a bifold cobordism from $S^{3}$ to $(Y,K)$ and
    $\delta$ is the formal dimension of the moduli spaces on
    $(X,\Sigma)^{\sharp}$ mod $2$ as above. This depends only on
    $\alpha$ and the parity class of $\phi$. When $\Lsharp(K)$ has
    been given an absolute $\Z/2$ grading in this way, we shall
    sometimes write it as $\Lsharp_{*}(K) = \Lsharp_{0}(K) \oplus
    \Lsharp_{1}(K)$.
\end{definition}

The definition of $Q(\Sigma, \phi)$  can be extended in the obvious
way for  a foam
cobordism $\Sigma$ with incoming and outgoing boundaries carrying
semi-framings $\phi_{0}$, $\phi_{1}$. In this  case
it is additive for composite cobordisms and coincides with the
definition of $\Sigma\ccdot\Sigma$ from \cite{KM-Tait} for closed
foams. It follows that this mod 2 invariant of foam cobordisms
determines whether the maps on $\Lsharp$ induced by a cobordism are
even or odd for the corresponding absolute $\Z/2$ gradings:

\begin{lemma}
 Let $(X,\Sigma)$ be a foam cobordism, with semi-framings $\phi_{0}$
 and $\phi_{1}$ on the boundaries $(Y_{0}, K_{0})$ and
 $(Y_{1},K_{1})$. Use these semi-framings to define the absolute
 $\Z/2$ gradings for $\Lsharp_{*}(Y_{i}, K_{i})$. Then the map
 $\Lsharp(X,\Sigma)$ is even or odd with respect to these $\Z/2$
 gradings according to the parity of the relative self-intersection
 number, $2 Q(\Sigma,\phi_{0}, \phi_{1})$.\qed
\end{lemma}

\subsection{Planar webs and degrees of maps}

Let $K\subset \R^{3}$ be a spatial web. Compactify $\R^{3}$ as $S^{3}$
and use the framed point at infinity for the atom, to construct
$\Lsharp(K)$. Let $\Pi\subset \R^{3}$ be a plane, and let
\[
       \pi  : \R^{3}\to \Pi
\]
be the orthogonal projection. If $K$ is in general position, then the
only singularities of the map $\pi : K\to \Pi$ are transverse
crossings at interior points of edges. In particular, we can require
that, at every vertex $v$, the kernel of $d\pi_{v}$ is not orthogonal
to the common normal $W_{v}$. This implies that the tangents to the
three edgelets in the image are distinct in $T_{\pi(v)}\Pi$. We refer
to $\pi$ or the image $\pi(K)$ as a \emph{regular planar diagram} of
the web $K$. A special case is a web that is actually planar.

\begin{definition}
    Let $K\subset \R^{3}$ be a spatial web and $\pi : K \to \Pi$ a
    regular planar diagram. The associated \emph{diagram semi-framing} of
    $K$ is then the semi-framing $\phi$ for which
    the line subbundle is $W_{\phi} = (\ker d\pi)|_{K}$.
\end{definition}

From the definition, the parity class of the diagram semi-framing is
consonant with the vertex orientations $o_{v}$ which are obtained from an
orientation of $\Pi$ via the projection $\pi$.

Now consider the case of webs $K$ that are actually planar, lying in
the plane $\Pi$. Let $K_{0}$ and $K_{1}$ be two such webs, both
equipped with their diagram semi-framings. Let us say that a foam
cobordism $\Sigma \subset [0,1]\times \R^{3}$ is
\emph{three-dimensional} if it lies in the subspace $[0,1]\times \Pi$.

\begin{lemma}\label{lem:3d-foams-even}
    If the foam $\Sigma$ is three-dimensional, then the resulting map
    $\Lsharp(\Sigma): \Lsharp_{*}(K_{0})\to \Lsharp_{*}(K_{1})$ has
    even degree, i.e.~preserves the $\Z/2$ gradings.
\end{lemma}

\begin{proof}
    The line bundle $\ker d\pi$ which defines the diagram
    semi-framings $\phi_{0}$, $\phi_{1}$ of $K_{0}$ and $K_{1}$
    extends to a line bundle $\ker d (\mathrm{id} \times \pi)$ along
    the foam $\Sigma$, which shows that
    $Q(\Sigma,\phi_{0},\phi_{1})=0$.
\end{proof}

Either using this lemma, or simply by examining the dimension of
explicit moduli spaces, we can determine the mod 2 grading in the
simplest cases:

\begin{lemma}\label{lem:unknot-theta-even}
    For the unknot and for the theta web with its planar embedding,
    and using the diagram semi-framing,
    the homology $\Lsharp_{*}(K)$ is non-zero only in even grading.
\end{lemma}

\begin{proof}
    For the empty web, the homology is in even grading from the
    definition. For the unknot and the theta web, $\Lsharp_{*}(K)$ is
    generated by 3-dimensional cobordisms from the empty web
    (including decorations by dots), so the
    Lemma~\ref{lem:3d-foams-even} can be applied.
\end{proof}

With Lemma~\ref{lem:3d-foams-even}
as a starting point,  we can determine which of the
maps in the octahedral diagram have even degree and which are odd
(Figure~\ref{fig:L-octahedron}).

\begin{lemma}
    Let $K_{i}$, $L_{i}$, $i=1,2,3$, be webs in an octahedral diagram,
    equipped with diagram semi-framings as implied by the figures.
    Then the cobordism maps in the octahedron have even or odd grading
    as indicated in Figure~\ref{fig:L-octahedron-odd-even}.
\end{lemma}

\begin{figure}
    \begin{center}
        \includegraphics[scale=.46]{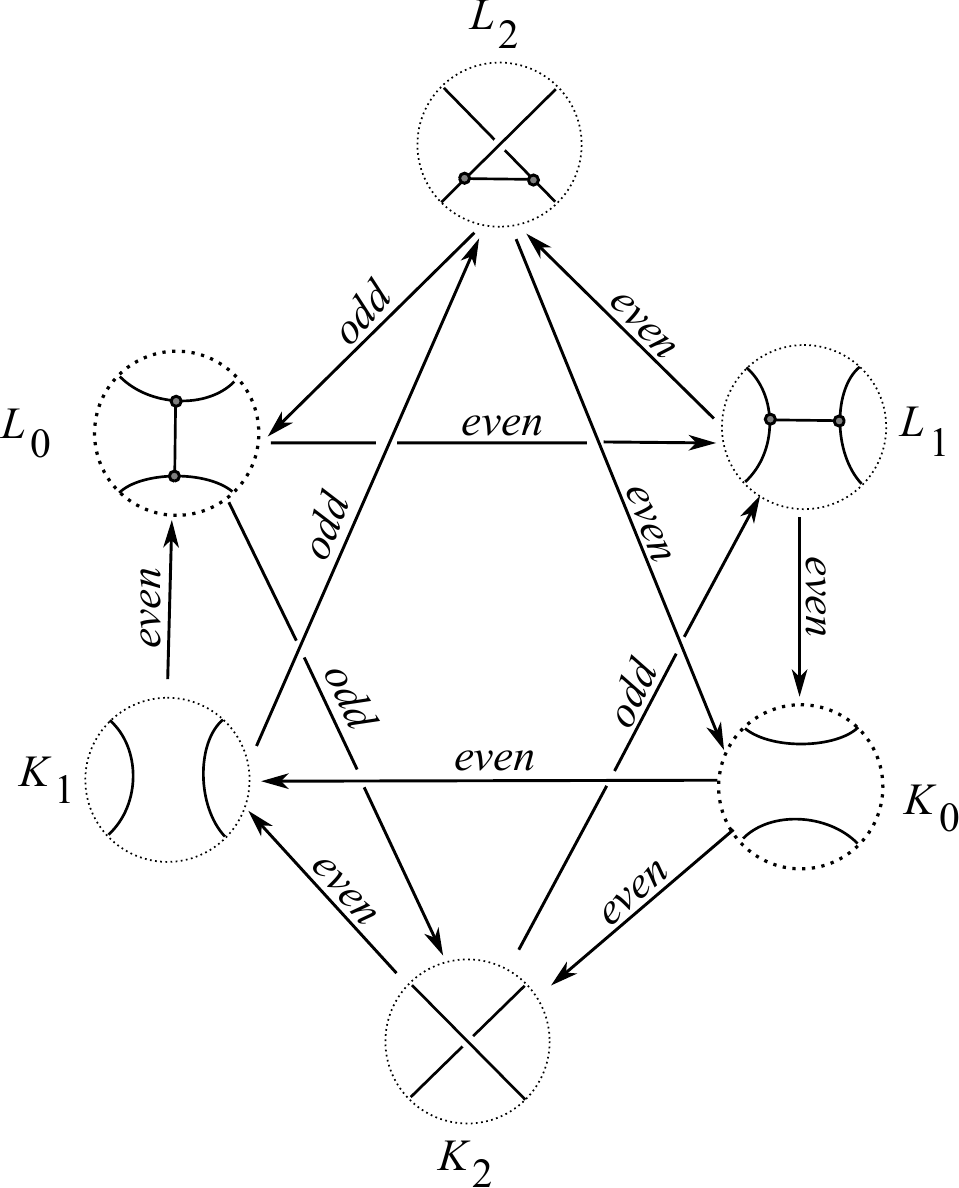}
    \end{center}
    \caption{\label{fig:L-octahedron-odd-even}
    The mod $2$ degrees of the maps when the webs are semi-framed in the
    parity class determined by the planar diagrams.}
\end{figure}

\begin{proof}
    Although the figures in the octahedral diagram only show the parts
    of the web lying in a ball (with the understanding that the webs
    are identical outside), the self-intersection numbers $Q$ of the
    various foams depend only on the non-trivial parts of the
    cobordism, so the question of determining the parity of the
    cobordism maps is well-defined in this context. The diagrams for
    $L_{0}$, $L_{1}$, $K_{0}$ and $K_{1}$ are all planar, and the
    cobordisms between them are 3-dimensional, so by
    Lemma~\ref{lem:3d-foams-even}, these maps are even. This leaves
    only the cobordism maps involving either $L_{2}$ or $K_{2}$.

    In a skein exact triangle, the composite of the three maps always
    has odd degree. This can be seen from the proof of exactness: the
    construction of the second chain homotopies $k_{i}$ in
    section~\ref{subsec:second-chain} shows that the moduli spaces
    along the triple-composite cobordism have formal dimension which
    is one less than the corresponding moduli spaces on the product
    cobordism.

    Looking at the exact triangle involving $L_{0}$, $L_{1}$ and
    $L_{2}$, we therefore learn that exactly on of the two maps
    $L_{1}\to L_{2}$ or $L_{2}\to L_{0}$ has odd degree. Let us
    complete the pictures to closed webs by adding two arcs on the
    left and right, so that $L_{0}$ is a theta web and $L_{1}$ has
    trivial homology because it has a bridge. In this case, the map
    $L_{2}\to L_{0}$ must be an isomorphism, by exactness. Both are
    theta webs, but their planar diagrams give them semi-framings in
    opposite parity classes. The map between them, being an
    isomorphism, must therefore be odd. It follows also that $L_{1}\to
    L_{2}$ is even.

    With the same two arcs added on the outside, both $K_{0}$ and
    $K_{2}$ are unknots, and they have diagram semi-framings in the
    same parity class. From the exact triangle involving $L_{1}$, we
    see that the map $K_{0}\to K_{2}$ is an isomorphism, so it must be
    even. The map $K_{2}\to L_{1}$ must therefore be odd. A similar
    argument involving the triangle $L_{0}$, $K_{1}$ and $K_{2}$
    determines the parity of $K_{2}\to K_{1}$ and $L_{0}\to K_{2}$.
    The parity of the remaining maps is then easily obtained. 
\end{proof}

\subsection{The Euler characteristic}
\label{subsec:euler}

Given a web $K$ with a semi-framing, we have a $\Z/2$ graded vector
space $\Lsharp_{*}(K)$, and we can consider the Euler characteristic
$\chi(\Lsharp_{*}(K)) \in
\Z$.

\begin{lemma}
    The integer $\chi(\Lsharp_{*}(K))$ for semi-framed webs $K$ is
    invariant under isotopy, is multiplicative for disjoint split
    unions, and satisfies the following additional relations for webs
    $K$ when the semi-framing is determined by the planar diagram:
    \begin{enumerate}
    \item \label{item:crossing-change}
    $\mathfig{0.079}{figures/Graph-X} =
    \mathfig{0.079}{figures/Graph-Xdag}$
    \item  \label{item:skein}
    $\mathfig{0.079}{figures/Graph-X} =
    \mathfig{0.079}{figures/Graph-Res1} -
    \mathfig{0.079}{figures/Graph-I}$
    \item  \label{item:whole-twist}
    $\mathfig{0.079}{figures/Graph-arc-twist} =
    \mathfig{0.079}{figures/Graph-arc}$
    \item  \label{item:vertex-twist}
    $\mathfig{0.079}{figures/Graph-Y-twisted} =
   - \mathfig{0.079}{figures/Graph-Y-untwisted}$
   \item \label{item:unknot-3}
   $\mathfig{0.079}{figures/Graph-O} =
   3$
    \end{enumerate}
    These relations completely determine the invariant. Only the
    overall sign of the invariant depends on the semi-framing.
\end{lemma}

\begin{proof}
    The relation \ref{item:skein} follows from the exact triangle
    $(K_{2}, K_{1}, L_{0})$ in Figure~\ref{fig:L-octahedron-odd-even},
    now that we know that it is the map $L_{0}\to K_{2}$ that is odd.
    From the exact triangle $(K_{2}, L_{1}, K_{0})$, we similarly
    obtain the following relation for the Euler characteristics:
    \[
         \mathfig{0.079}{figures/Graph-X} =
    -\mathfig{0.079}{figures/Graph-H} +
    \mathfig{0.079}{figures/Graph-Res0}
    \]
    If we rotate the three diagrams in this relation by a quarter turn
    and compare with the relation \ref{item:skein}, we deduce the
    crossing-change relation \ref{item:crossing-change}.

    The twist in the strand in \ref{item:whole-twist} changes the
    semi-framing by an integer, and does not change the parity class,
    hence the equality there. The twist in \ref{item:vertex-twist} on
    the other hand, changes the parity class of the diagram
    semi-framing, so changes the sign of the Euler number. Finally item
    \ref{item:unknot-3} follows from the fact that $\Lsharp_{*}$ has
    rank $3$ in this case and is supported in even grading, by
    Lemma~\ref{lem:unknot-theta-even}.

    These properties characterize this invariant of semi-framed webs,
    because \ref{item:skein} can be used to reduce a general web to a
    knot or link (i.e a web without vertices), and
    \ref{item:crossing-change} can then be used to reduce to the case
    of an unlink. Indeed, the relation \ref{item:vertex-twist} is not
    needed, as it can be deduced from the others.
\end{proof}

The crossing-change relation \ref{item:crossing-change} in the lemma
means that this invariant of semi-framed webs does not depend on the
spatial embedding. This invariant can be found elsewhere in the
literature. One connection is with the \emph{Yamada polynomial} of a
thickened spatial graph \cite{Yamada}. The Yamada polynomial $R[K](A)$
is a finite Laurent series in $A$ associated to a thickened graph in
$3$-space. When restricted to trivalent graphs, and evaluated at
$A=1$, it is an integer invariant $\mathcal{Y}(K)=R[K](1)$ that
satisfies exactly the same relations as $\chi(\Lsharp_{*}(K))$, with
the exception of a change of sign in the skein relation:
 \begin{equation}\label{eq:Y-skein}
 \mathcal{Y} \mathfig{0.079}{figures/Graph-X} =
    \mathcal{Y}\mathfig{0.079}{figures/Graph-Res1} +
    \mathcal{Y}\mathfig{0.079}{figures/Graph-I}.
\end{equation}    
Correcting for the sign, we see from the lemma that
\[
           \chi(\Lsharp_{*}(K)) = (-1)^{n/2}  \mathcal{Y}(K) 
\]
where the even integer $n$ is the number of vertices.

This same integer invariant can be evaluated as a \emph{signed} count
of Tait colorings, as follows. Let $K$ be a web, and let local
orientations $o_{v}$ at the vertices be given by specifying at each
vertex $v$ a cyclic order for the edgelets at that vertex. The local
orientations $o_{v}$ determine a consonant parity class of
semi-framings, by Lemma~\ref{lem:cons-parity}, and hence a
well-defined $\Z/2$ grading on $\Lsharp_{*}(K)$. Given a Tait coloring
$t$ of $K$, the order of the colors $\{1,2,3\}$ at each vertex also
determines a cyclic order of the edgelets. We attach a sign
$\epsilon_{v}(t)=\pm 1$ to each vertex according to whether this
cyclic order agrees with the order determined by $o_{v}$. We then
define an overall sign as the product: \[\epsilon(t)=\prod_{v}
\epsilon_{v}(t). \] The \emph{signed Tait count} for $K$ with the
given local orientations $o_{v}$ is then defined as
\begin{equation}\label{eq:sTait}
         \mathrm{sTait}(K) = \sum_{t} \epsilon(t),
\end{equation}
where the sum is over all Tait colorings $t$. It is not hard to verify
that $\mathrm{sTait}(K)$ satisfies the same relations as
$\mathcal{Y}(K)$, with the only non-trivial one being the skein
relation \eqref{eq:Y-skein}.
So we have
\begin{equation}\label{eq:Yamada-signed-sum}
                \mathcal{Y}(K) = \mathrm{sTait}(K) .
\end{equation}
We draw out the conclusion as a separate proposition:

\begin{proposition}
For a semi-framed web $K$, the Euler characteristic of
$\Lsharp_{*}(K)$ is equal to the signed count of Tait colorings, up to
an overall sign $(-1)^{n/2}$:
\[
\chi(\Lsharp_{*}(K)) = (-1)^{n/2}\, \mathrm{sTait}(K) .
\]
\qed
\end{proposition}

If we restrict out attention to \emph{planar} webs (equipped with
their diagram semi-framings), then there is slightly different set of
relations that characterizes the Euler number. Chasing the octahedral
diagram in Figure~\ref{fig:L-octahedron-odd-even} we obtain the
following relation among the Euler numbers of the four planar diagrams
in the middle of the figure:
\[
                \mathfig{0.079}{figures/Graph-H} -
                \mathfig{0.079}{figures/Graph-I} +
                \mathfig{0.079}{figures/Graph-Res1} 
   - \mathfig{0.079}{figures/Graph-Res0} =0.
\]
This ``Tutte relation'' is also satisfied by the invariant
$\mathrm{Tait}(K)$ which counts Tait colorings of $K$ (without sign).
Indeed, this relation completely characterizes $\mathrm{Tait}(K)$ when
combined with the normalization provided by the value on unlinks the
relation
\[
             \mathfig{0.079}{figures/Lollipop}=0.
\]
See \cite{Tutte,Agol-Krushkal}. So for planar webs we have
\begin{equation}\label{eq:chi-Tait-planar}
      \chi(\Lsharp_{*}(K)) = \mathrm{Tait}(K).
\end{equation}
On the other hand, from Theorem~\ref{thm:planar-tait}, we know that
$\mathrm{Tait}(K)$ is also the \emph{dimension} of $\Lsharp_{*}(K)$ in
the planar case. So we have the following corollary.

\begin{corollary}
     For a planar web with its diagram semi-framing, the homology
     $\Lsharp_{*}(K)$ is supported in the even grading. \qed
\end{corollary}

Because we have a formula for $\chi$ in terms of
the signed count or in terms of the absolute count of Tait colorings,
it is apparent that, for planar webs, all the Tait colorings
have the same sign (as can be checked directly):
\[
           \epsilon(t) = (-1)^{n/2}.
\]
In non-planar examples, the signs attached to the Tait colorings in
the sum \eqref{eq:sTait} need not all be equal. An example of this
occurs with the web $K=K_{3,3}$, the complete bipartite graph with $6$
vertices. For this case, for a standard spatial embedding such that a
planar projection has only one crossing, the rank of $\Lsharp(K)$ is
12, which is also the number of Tait colorings. But the Euler
characteristic of $\Lsharp(K)$ is zero: it is easy to verify that the
12 Tait colorings fall into two classes of 6 having opposite signs.
See section~\ref{subsec:K33}.

\section{Further results on foam evaluation}

\subsection{Neck-cutting and bubble bursting}

The following neck-cutting relation is illustrated in
Figure~\ref{fig:neck-cutting}.

\begin{figure}
    \begin{center}
        \includegraphics[scale=0.40]{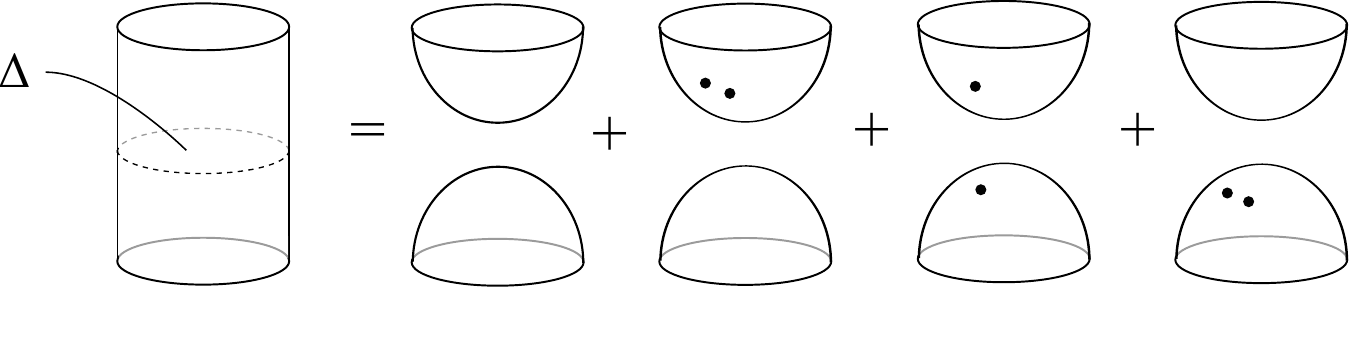}
    \end{center}
    \caption{\label{fig:neck-cutting}
   The neck-cutting relation. The
   disk $\Delta$ is not part of the foam.}
\end{figure}

\begin{proposition}[Neck-cutting]\label{prop:neck-cutting}
    Let $(X,\Sigma)$ be a cobordism defining a morphism in
    $\Cat^{\sharp}$. Suppose that $X$ contains an embedded disk
    $\Delta$ whose boundary lies in the interior of a facet of
    $\Sigma$, which it meets transversely, and whose interior is
    disjoint from $\Sigma$. Suppose that the trivialization of the
    normal bundle to $\Delta$ at the boundary which $\Sigma$
    determines extends to a trivialization over the disk. Let
    $\Sigma'$ be the foam obtained by surgering $\Sigma$ along
    $\Delta$, replacing the annular neighborhood of $\partial \Delta$
    in $\Sigma$ with two parallel copies of $\Delta$. Let
    $\Sigma'(k_{1}, k_{2})$ be obtained from $\Sigma'$ by adding
    $k_{i}$ dots to the $i$'th copy of $\Delta$. Then 
    \begin{equation}\label{eq:neck-cutting}
        \Lsharp(X,\Sigma) = \Lsharp(X, \Sigma'(0,0)) +
        \Lsharp(X,\Sigma'(0,2)) + \Lsharp(X,\Sigma'(1,1)) + \Lsharp(X,\Sigma'(2,0)). 
     \end{equation}
\end{proposition}

\begin{proof}
    The four morphisms in $\Cat^{\sharp}$ that appear in the formula
    are obtained by local modifications inside a ball, so the formula
    fits into the framework of the excision principle,
    Proposition~\ref{prop:excision-local}, with five local pieces
    $\mathbf{P}_{0}$, $\mathbf{P}_{1}$, \dots, $\mathbf{P}_{4}$. The
    first, $\mathbf{P}_{0}$, is the bifold corresponding to an annulus
    with boundary $U_{2}$, and the other four are all pairs of disks
    with dots, with the same boundary $U_{2}$. We can apply
    Corollary~\ref{cor:excision-unlink} directly, which leads us to
    consider the sum of evaluations of the closed foams
    \[
               \sum_{i_{0}}^{4}\Lsharp(\mathbf{P}_{i} \cup D(l_{1},
               l_{2})).
    \]
    This sum is zero for all $l_{1}$, $l_{2}$, as follows from
    Proposition~\ref{prop:sphere-with-dots}. The formula
    \eqref{eq:neck-cutting} therefore follows.     
\end{proof}

As one of several possible applications of the neck-cutting relation,
we single out this one as particularly useful. See \cite[Proposition
6.3]{KM-Tait}.

\begin{proposition}[Bubble-bursting]\label{prop:bubble-bursting}
    Let $D$ be an embedded disk in the interior of a facet of a foam
    $\Sigma\subset X$. Let $\gamma$ be the boundary of $D$, and let
    $\Sigma'$ be the foam $\Sigma' = \Sigma\cup D'$, where $D'$ is a
    second disk meeting $\Sigma$ along the circle $\gamma$, so that
    $D\cup D'$ bounds a $3$-ball. Let $\Sigma'(k)$ denote $\Sigma'$ with
    $k$ dots on $D'$. Let $\Sigma(k)$ denote $\Sigma$ with $k$ dots on
    $D$. Then we have
     \[
               \Lsharp(\Sigma'(k)) = \Lsharp(\Sigma(k-1))
     \]
      for $k=1$ or $2$, and $\Lsharp(\Sigma'(0))=0$. Furthermore
      $\Lsharp(\Sigma'(k+2))=\Lsharp(\Sigma'(k))$ for $k\ge 1$. 
\end{proposition}

\begin{proof}
    Let $\gamma_{1}$ be a circle parallel to $\gamma$ in
    $\Sigma\setminus D$ lying outside the ball bounded by $D$ and
    $D'$. Apply the neck-cutting relation to the surgery of $\Sigma'$
    along an auxiliary disk $\Delta$ with boundary $\gamma_{1}$. The
    surgered foam is the disjoint union of a foam isotopic to $\Sigma$
    with theta foam. The neck-cutting relations provides the relation
    \[
    \begin{aligned}
    \Lsharp(\Sigma'(k)) =
            \Lsharp(\Sigma(0))& \cdot \Lsharp(\Theta(k,0,0)) +
            \Lsharp(\Sigma(2))\cdot \Lsharp(\Theta(k,0,0)) \\ & \null +
            \Lsharp(\Sigma(1))\cdot \Lsharp(\Theta(k,0,1)) +
            \Lsharp(\Sigma(0))\cdot \Lsharp(\Theta(k,0,2)) .
            \end{aligned}
    \]
    The result follows by examining this formula in the cases
    $k=0,1,2$ using Proposition~\ref{prop:theta-evaluation}. The last
    sentence follows from Proposition~\ref{lem:Xie-rel-mod-2}.
\end{proof}

\subsection{Evaluation of some standard closed surfaces}

We consider the effect of changing a foam $\Sigma$ by forming a
connected sum with a standard surface contained in $S^{4}$. We begin
with a torus.

\begin{proposition}
    \label{prop:torus-sum} Let $\check X = (X, \Sigma)$ be a bifold
    cobordism, and let $\bX$ be $\check X$ with any decoration by
    dots. Let $\Sigma'$ be obtained as the internal connected sum
    $\Sigma\# T^{2}$ with a standard $2$-torus, at a point $x$ on a
    facet of $\Sigma$. Let $\bX'$ be the corresponding decorated
    bifold. Then, as linear maps, we have
    \[
                 \Lsharp( \bX' ) = \Lsharp(\bX(\mu^{2}+1)), 
    \]
    where, on the right-hand side, $\bX(\mu^{2})$ for example denotes
    $\bX$ with decoration by two additional dots on the facet of
    $\Sigma$ where $x$ is.
\end{proposition}

\begin{proof}
 This is a consequence of neck-cutting. Surgery on a disk $\Delta$
 with boundary on an essential curve in the $T^{2}$ changes $\bX'$
 back to $\bX$. In the neck-cutting formula \eqref{eq:neck-cutting},
 three of the terms contribute the $\mu^{2}$ term, while the first
 term on the right of \eqref{eq:neck-cutting} contributes the $+1$.    
\end{proof}

Next we examine the connected sum with the $R\subset S^{4}$, the
standard copy of $\RP^{2}$ with $R\ccdot R=2$ considered in
section~\ref{subsec:RP2-foams}.

\begin{proposition}
\label{prop:R-sum}
Let  $\check X = (X, \Sigma)$ be a bifold cobordism, and let $\bX$
    be $\check X$ with any decoration by dots. Let $\Sigma'$ be
    obtained as the internal connected sum $\Sigma\# R$ with $R$, the
    standard $\RP^{2}$ with $R\ccdot R =2$. Then we have
    \[
     \Lsharp( \bX' ) = \Lsharp(\bX).
    \]    
\end{proposition}

\begin{proof}
    This can be proved by a standard connected-sum argument, by
    stretching the neck where the sum is made. According to
    Lemma~\ref{lem:smallest-action-R}, there is a unique flat
    connection on the bifold $(S^{4}, R)$ which is an unobstructed
    solution with stabilizer $U(2)$. The group $U(2)$ is also the
    stabilizer of the flat connection on $(S^{3}, U(1))$, so when the
    neck is stretched, the local model for the moduli space on $\bX'$
    is the same as the moduli space on $\bX$. 
\end{proof}

An indirect argument now allows us to analyze the case of a sum with
the mirror-image copy of $\RP^{2}$, namely the standard copy
$R_{-}\subset S^{4}$ with self-intersection $-2$ (so that the branched
cover is $\CP^{2}$).

\begin{proposition}
\label{prop:R-minus-sum}
Let  $\check X$ and $\bX$
    be as above. Let $\Sigma'$ be
    obtained as the internal connected sum $\Sigma\# R_{-}$. Then we have
    \[
     \Lsharp( \bX' ) = \Lsharp(\bX(\mu^{2}+1)),
    \]
    where again $\mu$ is a dot on the facet of\/ $\Sigma$ where the sum
    is made.
\end{proposition}

\begin{proof}
       Because of Proposition~\ref{prop:R-sum}, we can equivalently
       try and compute for the foam $\Sigma\# R \# R \# R_{-}$, which
       is a sum of $\Sigma\# R$ and Klein bottle. Because $\Sigma \#
       R$ has a non-orientable facet where the sum is made, the 
       internal connected sum of $\Sigma\# R$ with a Klein bottle is
       the same as the sum with a torus, $\Sigma\# R \# T^{2}$, by an
       ambient isotopy. This
       last foam we can compute using Proposition~\ref{prop:R-sum} and
       Proposition~\ref{prop:torus-sum}, which gives the result.
\end{proof}

As a special case of these formulae for connected sums, we can compute
the evaluation $\Lsharp$ for standard closed surface in $S^{4}$,
either a standard orientable surface of genus $g$ formed as a sum of
unknotted tori,
\[
               \Sigma_{g} = \#_{g} T^{2},
\]
or a sum of copies of $R=R_{+}$ and $R_{-}$,
\[
                S_{a,b} =  ( \#_{a} R_{+}) \; \# \; (\#_{b} R_{-} ).
\]
We already know that for the sphere $\Sigma_{0}$ decorated with $k$
dots, the evaluation is zero for $k=0$ or $k$ odd, and is $1$ for even
$k\ge 2$. The above three propositions easily reduce the general
case to the case of $\Sigma_{0}$, and we obtain:

\begin{corollary}
    For the orientable surface $\Sigma_{g}$ with $g\ge 1$, or for the
    surface $S_{a,b}$ with $b\ge 1$, decorated with $k$ dots,
    the evaluation of $\Lsharp$ is
    $1$ for $k=0$ and zero otherwise. For the surface $S_{a,0}$ with
    $k$ dots, the evaluation is $1$ for even $k\ge 2$ and zero
    otherwise.\qed
\end{corollary}

\subsection{Bigons, triangles and squares}

Neck cutting and bubble bursting, and our calculations for the sphere,
the theta web and the tetrahedral web
have applications to the calculation of $\Lsharp$ for webs $K$
containing a bigon, a triangle, or a square. These results, carried
over from the $\SO(3)$ case in \cite{KM-Tait}, have algebraic
counterparts for the web homology defined via foam evaluation in
\cite{Khovanov-Robert} (see Propositions~3.13--3.15 in
\cite{Khovanov-Robert}). In the case of the bigon and square
relation, these relations appear in \cite{Khovanov-SL3}.

The proofs for $\Lsharp$ can be written
so as to follow the
$\SO(3)$ case from \cite{KM-Tait} with little modification, but can
be simplified a little in the present context, using the exact
triangle and our knowledge of planar webs. (The exact triangle for
$\Jsharp$ was not used in the proof of the corresponding results in
\cite{KM-Tait}.) We briefly summarize these results.

\begin{figure}
    \begin{center}
        \includegraphics[scale=0.50]{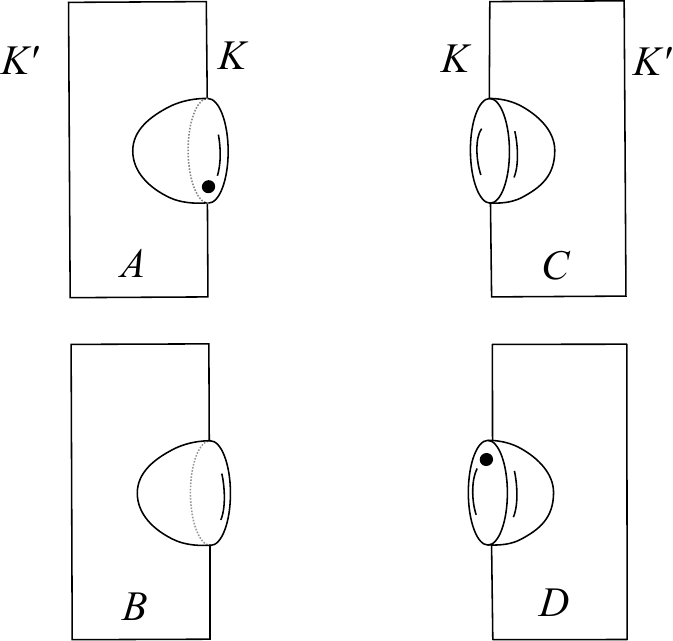}
    \end{center}
    \caption{\label{fig:bigon-maps}
   Four morphisms for the bigon relation.}
\end{figure}

Recall that a web $K$ contains a \emph{bigon} if two edges of $K$ are
arcs with common endpoints bounding a disk. There is then a web $K'$
obtained by collapsing the disk to a single edge and forgetting the
two vertices. Figure~\ref{fig:bigon-maps} illustrates four morphisms
between $K$ and $K'$, two of which include decorations with dots. The
morphism $A, B, C, D$ give linear maps
\[
\begin{aligned}
    a, b : \Lsharp(K') &\to \Lsharp(K) \\
       c, d : \Lsharp(K) &\to \Lsharp(K') \\
\end{aligned}
\]
Using the bubble bursting relation, and following the proof of the
corresponding result \cite[Proposition~6.5]{KM-Tait} for the
$\SO(3)$ case, we have:

\begin{proposition}[Bigon removal]
    The dimension of $\Lsharp(K)$ is twice that of $\Lsharp(K')$. The
    above morphisms provide mutually inverse isomorphisms
    \[
            a \oplus b : \Lsharp(K') \oplus \Lsharp(K') \to \Lsharp(K)
    \]
    and
    \[
                (c,d) : \Lsharp(K) \to \Lsharp(K') \oplus \Lsharp(K')
                .
    \]
\end{proposition}

\begin{proof}
    As in \cite{KM-Tait}, the fact that $(c,d)\comp (a \oplus b)$ is
    the identity follows from the bubble-bursting relations. To
    complete the proof we will show that $\dim \Lsharp(K) = 2 \dim
    \Lsharp(K')$ using the exact triangle.
    \begin{figure}
    \begin{center}
        \includegraphics[scale=0.50]{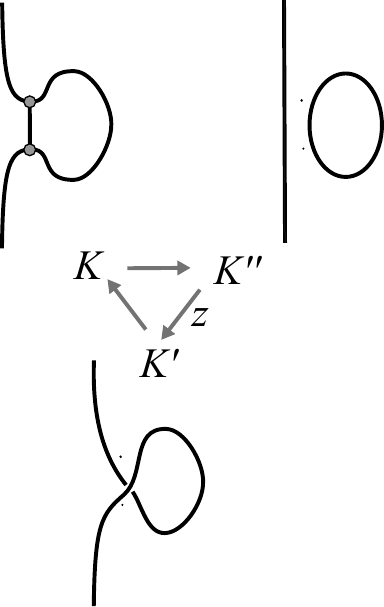}
    \end{center}
    \caption{\label{fig:bigon-exact-triangle}
   The exact triangle for the bigon relation.}
\end{figure}
   Figure~\ref{fig:bigon-exact-triangle} shows the exact triangle
   involving $K$, $K'$ and a third web $K''$ which is the disjoint
   union of $K'$ and an extra unknotted circle. By excision, we know
   that $\dim \Lsharp(K'') = 3 \dim \Lsharp(K')$. In the triangle, the
   map $z: \Lsharp(K'') \to \Lsharp(K')$ is surjective, as we can see by
   precomposing with the cobordism $K'\to K''$ which covers the extra
   circle with a disk: the composite is the identity morphism. So the
   kernel of $z$ has dimension $2 \dim \Lsharp(K')$ and is isomorphic
   to $\Lsharp(K)$.
\end{proof}

Next we state the result when $K'$ is obtained from $K$ by removing a
triangle, as described in Figure~\ref{fig:triangle-move}.

\begin{proposition}[Triangle removal]
    When $K'$ is obtained from $K$ by removing a triangle, as in
    Figure~\ref{fig:triangle-move}, then $\Lsharp(K)$ and $\Lsharp(K')$,
    are isomorphic. Mutually inverse isomorphisms are provided by the
    foams indicated in Figure~\ref{fig:triangle-foams}, each of which
    has a single tetrahedral point.
\end{proposition}

    \begin{figure}
    \begin{center}
        \includegraphics[scale=0.40]{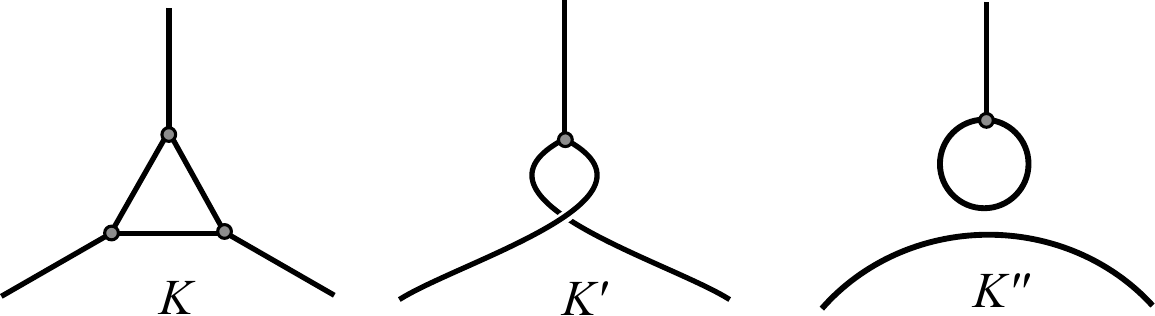}
    \end{center}
    \caption{\label{fig:triangle-move}
   The triangle-removal relation.}
\end{figure}

\begin{figure}
    \begin{center}
        \includegraphics[scale=0.4]{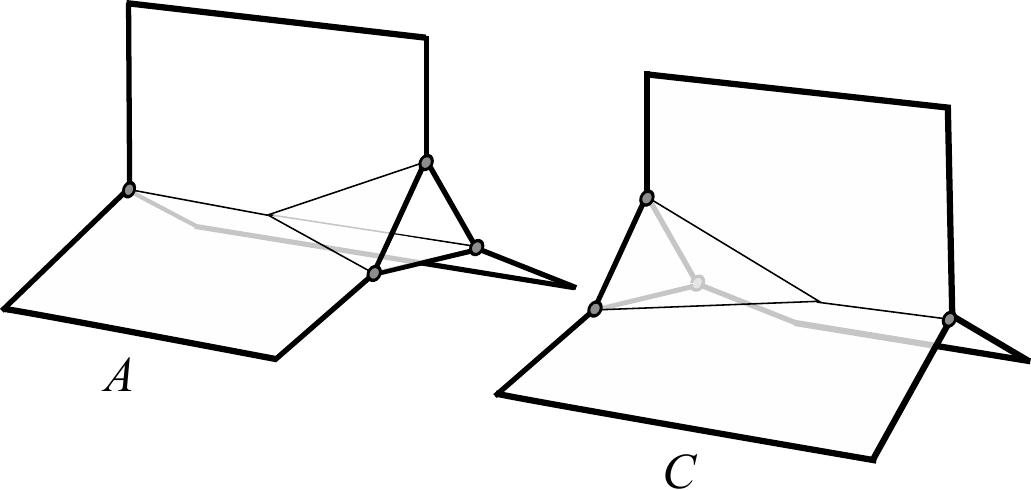}
    \end{center}
    \caption{\label{fig:triangle-foams}
   Foams $A$, $C$ as morphisms between $K$ and $K'$ for triangle
   removal.}
\end{figure}

\begin{proof}
    There is an exact triangle in which the third web $K''$ is as
    shown in the figure. An isotopy has been applied to the projection
    of $K'$ introducing a crossing, to match the description of the
    exact triangle.) Since the representation variety is empty, we
    have $\Lsharp(K'')=0$, and the map $\Lsharp(K)\to \Lsharp(K')$ is
    an isomorphism. There is a similar exact triangle with the
    crossing in $K'$ reversed and maps in the opposite directions,
    giving an isomorphism $\Lsharp(K')\to \Lsharp(K)$.

    The morphism $C'$ from $K$ to $K'$ is not the same as the morphism
    $C$ described in Figure~\ref{fig:triangle-foams}. But it is formed
    from $C$ by making a sum with $\Psi_{2}$ at the tetrahedral point,
    so Proposition~\ref{prop:sum-with-Psi-2-3} tells us that $C$ and
    $C'$ give the same map on $\Lsharp$. So the foam $C$ (and
    similarly $A$) in Figure~\ref{fig:triangle-move} gives an
    isomorphism.

    It remains only to show that the isomorphisms provided by the
    foams $A$ and $C$ are mutually inverse, and for this it is enough
    to look at the foam $A\cup C$ as a morphism from $K'$ to $K$'.
    This can be demonstrated by using excision to reduce to a local
    calculation and applying the known results for the foam
    evaluations in Propositions~\ref{prop:theta-evaluation} and
    \ref{prop:tetr-evaluation}. A model for this argument is the proof
    of the corresponding result in \cite{KM-Tait} (Proposition~6.6).
\end{proof}

Finally we consider the case that $K$ contains a square.

\begin{proposition}[Square removal]
    Suppose the web $K$ contains a square, and let $K'$ and $K''$ be
    obtained from $K$ as shown in Figure~\ref{fig:square}. Then  we
    have
    \[
                 \Lsharp(K) \cong \Lsharp(K') \oplus \Lsharp(K'').
    \]
\end{proposition}

    \begin{figure}
    \begin{center}
        \includegraphics[scale=0.40]{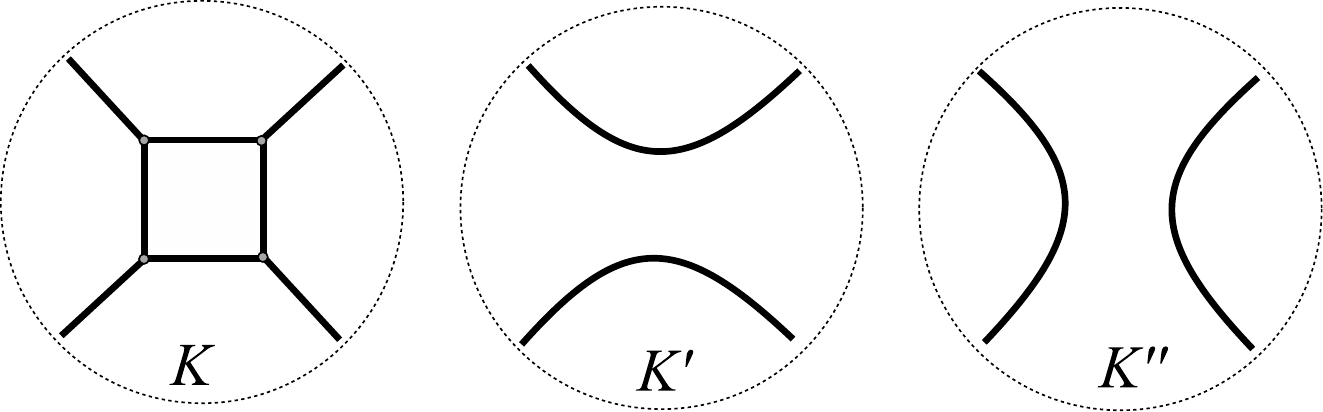}
    \end{center}
    \caption{\label{fig:square}
   The square relation.}
\end{figure}

\begin{proof}
    The proof for the $\SO(3)$ case has a formal aspect (which carries
    over essentially unchanged to the $\SU(3)$ case), but also a
    hands-on examination of a representation variety, used in
    \cite{KM-Tait} in the proof of Lemma~5.12 of that paper, where it
    is shown that the dimension of $\Jsharp(L_{4})$ is at most $24$
    when $L_{4}$ is the web formed by the edges of a cube. We need the
    corresponding result for $\Lsharp$. In \cite{KM-Tait}, the result
    for the $\SO(3)$ case was proved by showing that one can perturb
    the Chern-Simons functional so that the critical set is Morse-Bott
    and consists of four copies of the (real) flag manifold $\SO(3)/
    V_{4}$. Since the flag manifold itself has a Morse function with 6
    critical points, this gives a model for the complex which computes
    $\Jsharp(L_{4})$ with exactly $24$ generators, so providing the
    required bound. An argument of this sort can be carried out also
    in the $\SU(3)$ case for $\Lsharp$, to provide an upper bound
    $\dim\Lsharp(L_{4})\le 24$, completing the proof. (Indeed, in the
    $\SU(3)$ case it turns out that all 24 critical points are in the
    same mod $2$ grading, so we even get an equality rather than an
    upper bound.) However, we can avoid the need for this somewhat
    delicate argument in the present context, because $L_{4}$ is a
    planar web, and the equality $\dim\Lsharp(L_{4})=24$ follows from
    Theorem~\ref{thm:planar-tait}, because $L_{4}$ has $24$ Tait
    colorings.
\end{proof}

\section{Calculations for some non-planar webs}
\label{sec:examples}

\subsection{Calculation for Hopf links}

The ``linked handcuffs'' is the spatial web $\mathit{LHC}$ consisting
of a Hopf link with an extra edge joining the two components, as shown
in Figure~\ref{fig:hopf-cuffs-triangle}. The following proposition
describes $\Lsharp(\mathit{LHC})$ as both a vector space and a module
for the edge operators.

\begin{figure}
    \begin{center}
        \includegraphics[scale=0.5]{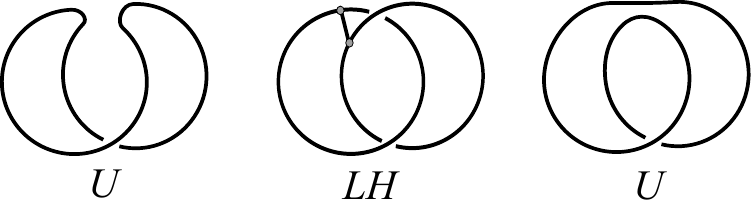}
    \end{center}
    \caption{\label{fig:hopf-cuffs-triangle}
   An exact triangle for computing $\Lsharp$ for the linked handcuffs,
   $\mathit{LHC}$.}
\end{figure}

\begin{proposition}
    \label{prop:LHC} As an $\F$-vector space, the $\SU(3)$ bifold
    homology $\Lsharp(\mathit{LHC})$ for the linked handcuffs has
    dimension $4$. It is a module for the polynomial algebra
    $\F[u_{1}, u_{2}, v]$, where $u_{1}$ and $u_{2}$ act by mapping to
    the edge operators corresponding to the two components of the Hopf
    link, and $v$ acts via the edge operator corresponding to the edge
    joining them. As such, we have
\[
        \Lsharp(\mathit{LHC} ) \cong M \oplus M\{1\}
\]
where $M$ is the $2$-dimensional cyclic module
\begin{equation}\label{eq:module-M}
            M = \F[u_{1}, u_{2}, v] \big/ \langle v,\, u_{1}+u_{2},\,
            u_{1}^{2}+1 \rangle
\end{equation}
and the notation $\{1\}$ indicates that the two copies of $M$ lie in
the two different relative $\Z/2$ gradings.
\end{proposition}

\begin{proof}
    The linked handcuffs appear in an exact triangle in which the
    other two webs are unknots $U$, as shown in
    Figure~\ref{fig:hopf-cuffs-triangle}. The connecting homomorphism
    in the exact triangle is provided by a cobordism from the unknot
    to the unknot which has the topology of the connect sum of a
    product annulus with a standard copy of $\RP^{2}$, embedded with
    self-intersection number $-2$. By
    Proposition~\ref{prop:R-minus-sum} the connecting homomorphism is
    multiplication by $(u^{2}+1)$. Since we have
    \[
            \Lsharp(U) =\frac{ \F[u]}{u} \oplus \frac{ \F[u]}
            {u^{2}+1},
    \]
    the connecting homomorphism has rank $1$, and its kernel and
    cokernel are both $\F[u]/(u^{2}+1)$. In the cobordisms between $U$
    and $\mathit{LHC}$ (in both directions), the dot operator $u$ for
    the unknot and is intertwined with both of the two dot operators
    $u_{1}$, $u_{2}$ for $\mathit{LHC}$, so the exact triangle
    presents $\Lsharp(\mathit{LHC})$ as an extension:
    \begin{equation}\label{eq:LHC-is-twice-M}
              0 \to M \to \Lsharp(\mathit{LHC}) \to M \to 0
    \end{equation}
    where $M$ is as described in the statement of the proposition.

    Let The groups $\Lsharp(U)$ and $\Lsharp(\mathit{LHC})$ both be
    $\Z/2$-graded using the semi-framings that come from the diagrams.
    Comparing with Figure~\ref{fig:L-octahedron-odd-even}, we see that
    one of the maps in the sequence in
    Figure~\ref{fig:hopf-cuffs-triangle} is even and one is odd. It
    follows that the extension is trivial and
    \[
        \Lsharp(\mathit{LHC} ) \cong M \oplus M\{1\}
     \]
     as claimed.
    \end{proof}

\begin{remark}
    The ordinary (unlinked) handcuffs $\mathit{HC}$ consists of a
    planar 2-component unlink with a straight edge joining the two
    components. The representation variety for $\mathit{HC}$ is empty
    on account of the bridge, so $\Lsharp(\mathit{HC})=0$. The web
    $\mathit{LHC}$ is obtained from $\mathit{HC}$ by a crossing
    change: they are the same abstract graph, with different
    embeddings. We see, therefore, that $\Lsharp(K)$ is not an
    invariant of abstract trivalent graphs, but does depend on their
    embedding.
\end{remark}
 
\begin{figure}
    \begin{center}
        \includegraphics[scale=0.5]{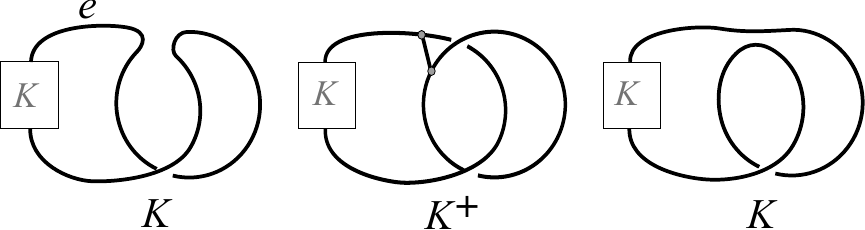}
    \end{center}
    \caption{\label{fig:add-earing-triangle}
   The web $K^{+}$ obtained from $K$ by adding an earing on an edge
   $e$.}
\end{figure}

The calculation of $\Lsharp(\mathit{LHC})$ can be partly generalized
to compute the $\Lsharp(K^{+})$ in terms of $\Lsharp(K)$, where
$K^{+}$ is the web obtained by adding an ``earing'' to $K$, as shown
in Figure~\ref{fig:add-earing-triangle}. From the exact triangle in the
figure, we see (as in the case that $K$ is the unknot) that
\[
       \Lsharp(K^{+}) \cong V \oplus V\{1\}
\]
as an $\F[u]$-module, where $u$ is the operator associated to the edge
$e$ and $V\subset \Lsharp(K)$ is the kernel of $u^{2}+1$. In the case
that $K$ is a knot, this tells us that $\dim \Lsharp(K^+) = 2 \dim
\Lsharp(K)-2$.

We turn next to the Hopf link $H$. The following proposition
determines $\Lsharp(H)$ as a vector space and as a  module. 

\begin{proposition}
    \label{prop:hopf} For the Hopf link $H$, let $u_{1}$ and $u_{2}$
    be the two edge operators. As a $\Z/2$ graded vector space
    $\Lsharp(H)$ has dimension $9$ and is concentrated in even
    grading. As a module for the algebra $\F[u_{1}, u_{2}]$, we have 
    \[
           \Lsharp(H) \cong \Lsharp(U) \oplus \Lsharp(\theta).
    \]
\end{proposition}

\begin{figure}
    \begin{center}
        \includegraphics[scale=0.66]{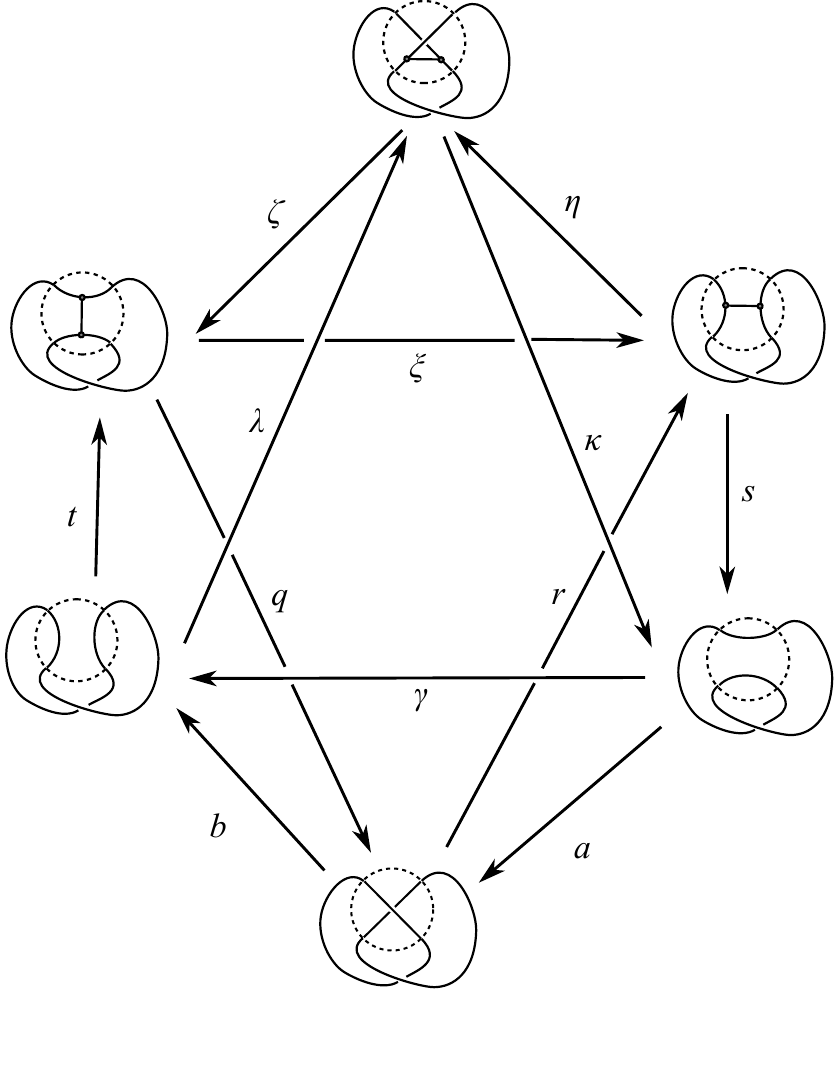}
    \end{center}
    \caption{\label{fig:octahedron-for-hopf}
   An instance of the octahedral diagram, for computation of the
   Hopf link.}
\end{figure}

\begin{proof}
 The Hopf link appears at the bottom of the octahedron shown in
 Figure~\ref{fig:octahedron-for-hopf}. The web at the top of the
 diagram is the handcuffs, which has trivial instanton homology
 because of the bridge. From the exact triangle
 $\{\lambda,\kappa,\gamma\}$, it follows that $\gamma$ is an
 isomorphism. Furthermore the map $a$ has a left inverse,
 $\gamma^{-1}b$, because $ba=\gamma$. It follows that the exact
 triangle $\{s,r,a\}$ has $s=0$ and also splits. The other two webs in
 this triangle are an unknot and a theta web. This gives the direct
 sum decomposition in the proposition. The theta web that appears has,
 from its diagram, a semi-framing of opposite parity from the one
 arising from a planar embedding, so the homology of this semi-framed
 theta web is in odd grading, while the homology of the unknot is in
 even grading. The map $a$ is even while the map $r$ is odd, so the
 homology of the Hopf link is all in even grading.
\end{proof}

\begin{remark}
    The Hopf link and the two-component unlink both have $\Lsharp(K)$
    of dimension $9$, but they have different module structures. In
    particular, consider the submodule $V(K;s)$ where $s$ consists of
    both components of the link. This is the intersection of the
    kernels of $(u_{1}^{2}+1)$ and $(u_{2}^{2}+1)$. In the case of the
    unlink, this is the cyclic module $\F[u_{1},
    u_{2}]/(u_{1}^{2}+1)(u_{2}^{2}+1)$. In the case of the Hopf link
    however, the proposition tells us that this module is not cyclic
    but is the direct sum of two copies of the module $M$ from
    \eqref{eq:module-M}.
\end{remark}

\subsection{Calculation for the trefoil}

The homology for the trefoil can be calculated using the octahedral
diagram, starting from our existing calculations of the linked
handcuffs, the Hopf link and the theta web.
\begin{figure}
    \begin{center}
        \includegraphics[scale=0.66]{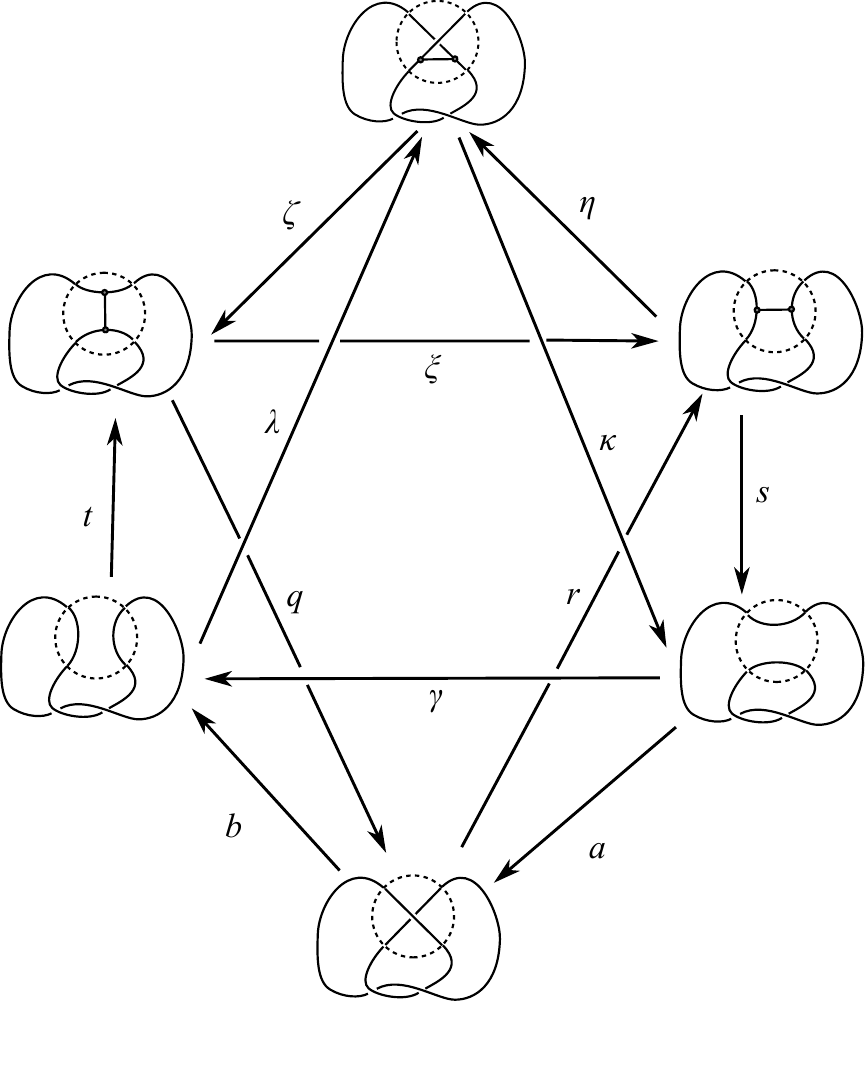}
    \end{center}
    \caption{\label{fig:octahedron-for-trefoil}
   An instance of the octahedral diagram, for computation of the
   trefoil.}
\end{figure}

\begin{proposition}
    Let $T$ be the trefoil knot (with either handedness). Then
    $\Lsharp(T)$, as a module over $\F[u]$ where $u$ is the edge
    operator, has the form
    \[
        \Lsharp(T) = N \oplus M \oplus M \oplus M\{1\},
    \]
    where $N$ is the $1$ dimensional module with $u=0$, and $M$ is the
    $2$-dimensional module $\F[u]/(u^{2}+1)$. The modules $N$ and $M$
    in this decomposition lie in even grading, and $M\{1\}$
    is shifted, in odd grading.
\end{proposition}

\begin{proof}
    From our calculation of the linked handcuffs earlier, this result
    is equivalent to the statement
    \[
            \Lsharp(T) = \Lsharp(U) \oplus \Lsharp(\mathit{LHC}),
    \]
    where $\Lsharp(\mathit{LHC})$ is regarded as an $\F[u]$ module
    using either of the two edge operators belonging to the cuffs.
    (The two operators are equal.)
    These three homology groups appear in an exact triangle which is
    one of the triangles in the octahedral diagram shown in
    Figure~\ref{fig:octahedron-for-trefoil}, where the right-handed
    trefoil appears in the bottom corner, and $U$ and $\mathit{LHC}$
    appear on the left. There are no non-trivial extensions to
    consider, so it will suffice to show that the connecting
    homomorphism $t$ is zero. In the diagram, $t= \zeta\comp\lambda$,
    so we will be done if $\lambda=0$. But $\lambda$ belongs to the
    exact triangle involving also $\kappa$ and $\gamma$ in the figure,
    and since the dimensions of the three vector spaces in this
    triangle are $3$, $6$ and $9$, the map $\lambda$ must be zero, as
    in the calculation of $\Lsharp(H)$ in Proposition~\ref{prop:hopf}.
\end{proof}

\subsection{The tangled handcuffs}
\label{subsec:tangled}

The web shown on the left in Figure~\ref{fig:tangled-cuffs} is the
tangled handcuffs, $\Th$. 

\begin{figure}
    \begin{center}
        \includegraphics[scale=0.4]{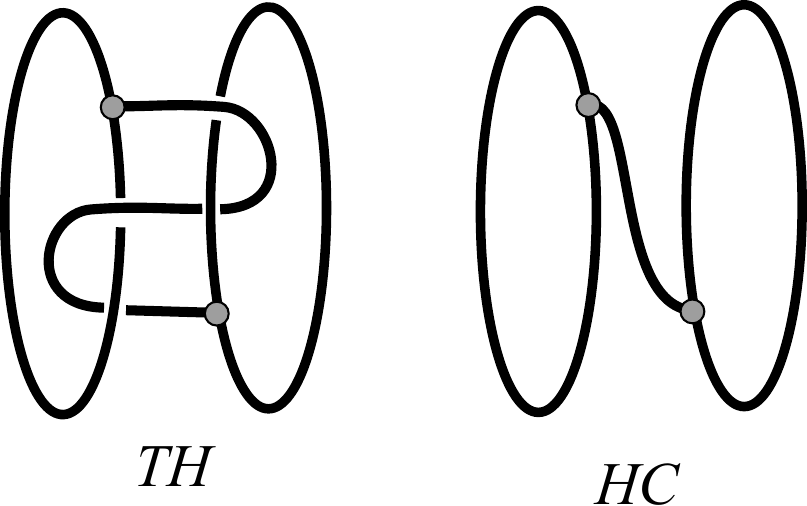}
    \end{center}
    \caption{\label{fig:tangled-cuffs}
   The ``tangled handcuffs'' and the untangled handcuffs.}
\end{figure}

\begin{proposition}
    For the tangled handcuffs, we have $\Lsharp(\Th) =0$. 
\end{proposition}

\begin{proof}
    There is only one choice of $1$-set $s$ for $\Th$, namely the
    single edge which forms the tangled ``chain'' joining the two
    cuffs, so we have $\Lsharp(\Th) = V(\Th; s)$. By
    Lemma~\ref{lem:modify-s}, we can modify the edges of the $1$-set,
    keeping the relative homology class unchanged, without altering
    $V(\Th, s)$. It follows that $\Lsharp(\Th) \cong
    \Lsharp(\mathit{HC})$, where the latter is the untangled
    handcuffs. But for $\mathit{HC}$, the representation variety is
    empty, and $\Lsharp$ must vanish.
\end{proof}

\begin{remarks}
    Essentially the same argument is used in \cite{KM-deformation}, to
    show that a deformation of $\Jsharp$ vanishes. On the other hand,
    $\Jsharp$ itself is non-trivial for $\Th$ on account of a general
    non-vanishing theorem for spatial webs without an embedded bridge
    \cite{KM-Tait}. We see here that such a non-vanishing theorem
    cannot hold of $\Lsharp$. At the same time, this is the first
    example presented in this paper where $\Lsharp$ and $\Jsharp$ have
    different dimensions.
\end{remarks}

\subsection{The bipartite graph \texorpdfstring{$K_{3,3}$}{K(3,3)}}
\label{subsec:K33}

The complete bipartite graph $K_{3,3}$ is the simplest trivalent graph
that is not abstractly isomorphic to a planar graph. Of course, it has
many spatial embeddings as a web in $\R^{3}$, but we refer to the one
pictured as $L_{2}$ at the top of
Figure~\ref{fig:L-octahedron-for-K33}.

\begin{figure}
    \begin{center}
        \includegraphics[scale=.46]{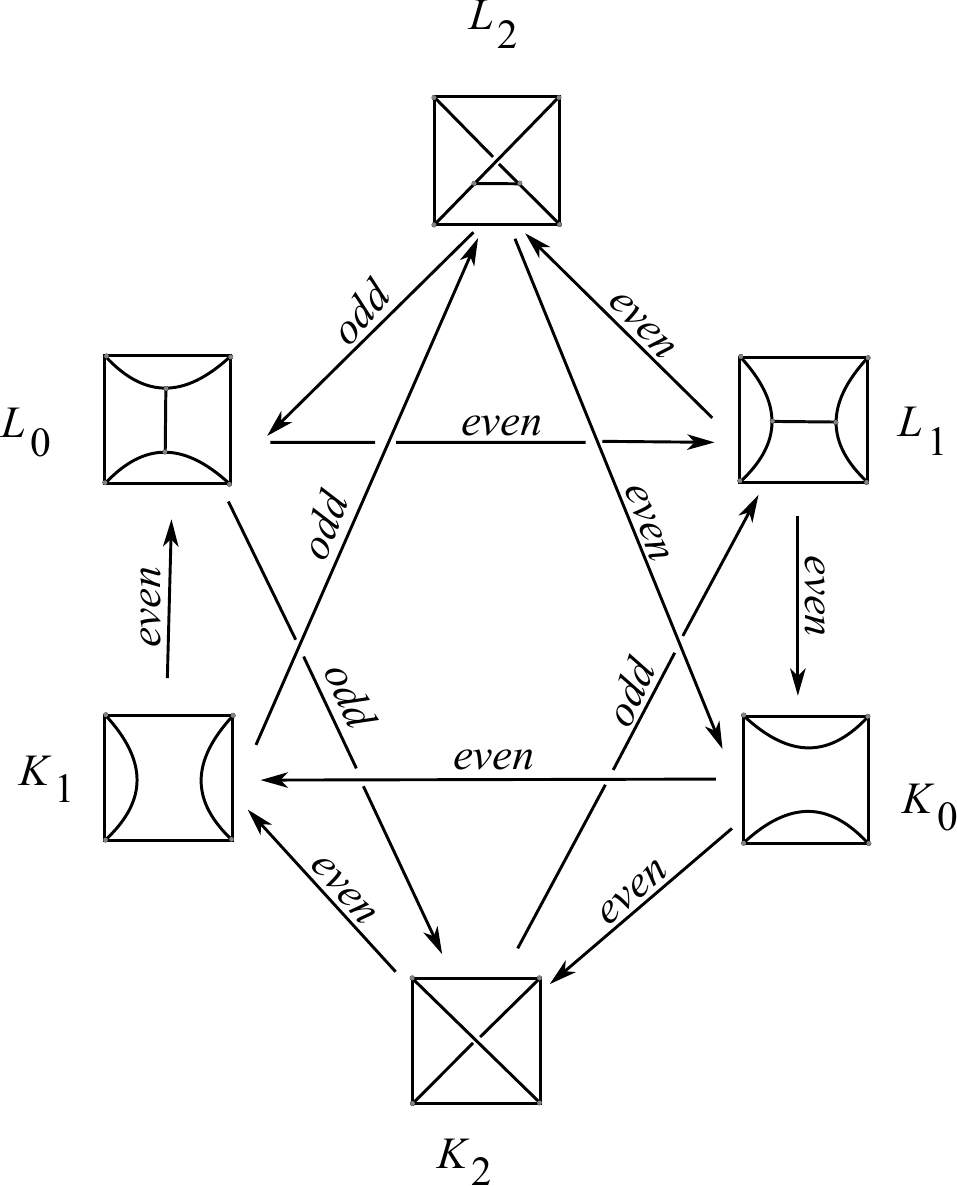}
    \end{center}
    \caption{\label{fig:L-octahedron-for-K33}
    Computing $\Lsharp$ for the spatial embedding of the complete
    bipartite graph $K_{3,3}$, shown in the position $L_{2}$.}
\end{figure}

\begin{proposition}
    For the graph $K_{33}$ embedded as a web as shown in
    Figure~\ref{fig:L-octahedron-for-K33}, the dimension of $\Lsharp$
    is $12$ and the Euler characteristic is $0$.
\end{proposition}

\begin{proof}
Except for $L_{2}$, all webs in the octahedron in
Figure~\ref{fig:L-octahedron-for-K33} are planar. In the case of
$K_{2}$, the diagram shown is not the planar embedding, but the
semi-framing of the diagram is in the same parity class as the planar
one. So $\Lsharp$ is supported in even degrees for all webs in the
picture except (perhaps) for $L_{2}$. The latter is an embedding of
the complete bipartite graph $K_{3,3}$.

The map $L_{0}\to L_{1}$ in the diagram has even degree, but it is the
composite of two odd-degree maps, via $K_{2}$. So the map $L_{0}\to
L_{1}$ is zero. It follows that the exact triangle at the top of the
octahedron becomes a short exact sequence (notation $\Lsharp$ is
implied but omitted):
\[
  0 \to L_{1} \to L_{2} \to L_{0} \to 0,
\]
Since one map in this short exact sequence is
even and the other is odd, and since $L_{1}$ and $L_{2}$ both have
$\Lsharp$ supported in even degree, we have
\[
       \Lsharp(L_{2}) = \Lsharp(L_{1}) \oplus \Lsharp(L_{0})\{1\}.
\]
Both $L_{1}$ and $L_{0}$ are tetrahedral webs, and $\Lsharp$ has
dimension $6$ for both. It follows that $\Lsharp(L_{2})$ has rank $12$
and Euler characteristic $0$.    
\end{proof}

\begin{remark}
    It is not hard to verify that the representation variety for
    $(S^{3}, K_{3,3})^{\sharp}$ consists of two copies of the flag
    manifold $\SU(3)/ T^{2}$. The flag manifold has ordinary homology
    of rank $6$, all in even gradings. The above proposition shows
    that $\Lsharp$ in this case can be identified with the ordinary
    homology of the representation variety, but with one of the copies
    of the flag manifold shifted so that its homology is in odd
    grading.
\end{remark}

\subsection{The Kinoshita theta graph}

The Kinoshita theta graph $K$ \cite{Kinoshita} is a spatial embedding
of the theta graph that is knotted (i.e.~not isotopic to a planar
embedding), but has the Brunnian property, that the deletion of any
one edge leaves an unknot. See Figure~\ref{fig:Kinoshita-theta}.
Knowing only this, if we look at the direct sum decomposition,
\[
       \Lsharp(K) = \bigoplus_{s} V(K;s),
\]
we find that there are three $1$-sets and each corresponding summand
is $2$-dimensional, because it coincides with the kernel of
$(u^{2}+1)$ for the unknot. So $\Lsharp(K)$ has dimension $6$. Since
this is also the number of Tait colorings, we see that the Euler
number must also be $\pm 6$, so $\Lsharp(K)$ is supported in just one
grading (which depends on the choice of semi-framing).

\begin{figure}
    \begin{center}
        \includegraphics[scale=.75]{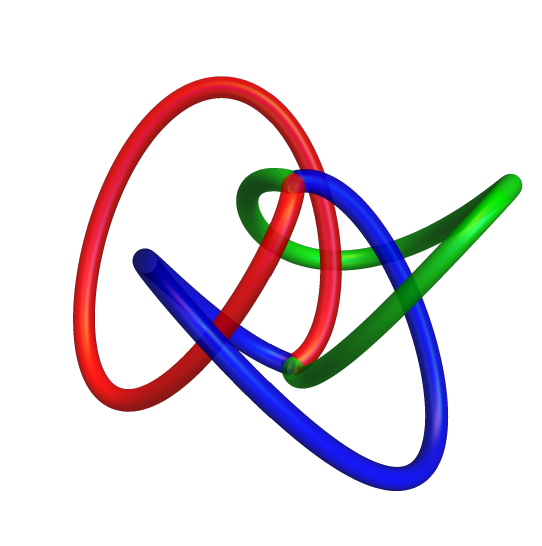}
    \end{center}
    \caption{\label{fig:Kinoshita-theta}
    The Kinoshita theta graph.}
\end{figure}

The discussion so far applies to any Brunnian theta graph, but for the
particular case of the Kinoshita graph it is interesting to also look
at the representation variety $\Rep^{\sharp}(K)$. The orbifold
fundamental group of $(S^{3}, K)$ is a finite subgroup $\Gamma$ of
$\SO(4)$, and the orbifold is a quotient of $S^{3}$. The group
$\Gamma$ can most easily be described in terms of its double cover
$\tilde\Gamma$ in $\Spin(4) =\SU(2)\times \SU(2)$, which is a product,
\[
     \tilde\Gamma = 2I \times Q_{8},
\]
where the factors are the binary icosahedral group and the quaternion
group of order $8$. From this description it is straightforward to
verify that there are exactly three conjugacy classes of
representations $\Gamma\to\SU(3)$ sending elements of order $2$ to
non-trivial elements (as required for a bifold connection). One of
these three representations is abelian and factors through the map
$\Gamma\to V_{4} = Q_{8}/\{\pm 1\}$. The other two factor through the
map $\Gamma\to I$. The group $I$ is $A_{5}$ which has two irreducible
representations in $\SU(3)$. These conjugacy classes give rise to the
three connected components of $\Rep^{\sharp}(K)$, which are one flag
manifold $\SU(3)/T$ and two copies of $\PSU(3)$. The instanton
homology $\Lsharp(K)$ coincides with $\Lsharp$ for the unknotted theta
graph and can be identified with the homology of the flag manifold, as
a vector space.

\section{Further discussion}
\label{sec:discuss}

\subsection{Relation to Khovanov-Rozansky homology}

The Khovanov-Rozansky $\sl_{N}$ homology of an oriented knot or link
$K$, originally defined over a field of characteristic zero, is a
finite-dimensional vector space $\KR_{N}(K)$ carrying an operator
$x_{e}$ for each edge $e$, satisfying $x_{e}^{N}=0$. A suitable
deformation $\KR_{3}(K)'$ can be constructed in which the operators
satisfy $x_{e}^{3}+x_{e}=0$. Equivalent constructions, valid for
coefficients in any field, were given in \cite{Queffelec-Rose} and
\cite{Robert-Wagner}. So it is natural to compare our $\Lsharp(K)$
with the deformed $\KR_{3}(K;\Z/2)'$ in this case.

We can pursue this comparison via two closely related routes. On the
one hand, the exact triangles satisfied by $\Lsharp$ can be extended
so as to compute $\Lsharp(K)$ from a cube of resolutions (the
resolutions being trivalent graphs). On the gauge theory side, the
construction closely parallels the constructions in \cite{KM-unknot}.
The result is a spectral sequence whose $E_{2}$ page we can expect to
agree with $\KR_{3}(K;\Z/2)'$ and which abuts to $\Lsharp(K)$.

On the other hand, both $\Lsharp(K)$ and $\KR_{3}(K;\Z/2)'$ are
simplified because of the factorization $x^{3}+x=x(x+1)^{2}$. So, for
a knot $K$ for example, we have $\Lsharp(K) = \F \oplus
V(K,\emptyset)$ in the notation of section~\ref{subsec:edge}.
Similarly, the deformed Khovanov-Rozansky homology has a decomposition
\[
           \KR_{3}(K;\Z/2)' = \F \oplus \KR_{2}(K;\Z/2),
\]
where in the second summand the usual operator $x$ has been replaced
by $x+1$. We therefore expect there to be a spectral sequence from the
ordinary Khovanov homology $\KR_{2}$ for knots and links abutting to
$V(K,\emptyset)$. Starting from the exact triangles for $\Lsharp$, one
can verify with a little care that $V(K,\emptyset)$ satisfies the same
skein exact triangles as $\KR_{2}(K;\Z/2)$, and the required spectral
sequence should be constructed as in \cite{KM-unknot} again.

At this point, we can notice that there are two different instanton spectral
sequences whose $E_{2}$ page is $\KR_{2}(K;\Z/2)$. There is the
spectral sequence in \cite{KM-unknot} converging to the $\SU(2)$
singular instanton homology $\Isharp(K;\Z/2)$. Then there is the
spectral sequence described above, converging to the summand
$V(K,\emptyset)$ of the $\SU(3)$ instanton homology $\Lsharp(K)$. It
is natural to ask:

\begin{question}
Can we directly compare $\Isharp(K;\Z/2)$ with
$V(K,\emptyset)$ for a knot or link $K$? Are they even isomorphic?
\end{question}

\noindent
The authors do not have an approach to answering this question.

\subsection{Using \texorpdfstring{$\Th$}{TH} as an atom}

In section~\ref{subsec:atoms}, we described the ``atom'' which we use
in the construction of $\Lsharp(K)$. We can replace this choice of
atom with any other orbifold of our choice as long as the
representation variety consists of just one, irreducible
representation. Such a choice is the bifold $(S^{3}, \Th)$, where
$\Th$ is the tangled handcuffs described in
section~\ref{subsec:tangled}. Let us consider the $\SU(3)$ instanton
homology $L^{\Th}(K)$ for bifolds $K$ that arises when our atom
$H_{3}$ from section~\ref{subsec:atoms} is replaced with $\Th$.

One key difference is that the proof of the excision property,
Proposition~\ref{prop:excision}, no longer works as before: it is not
clear to the authors whether $L^{\Th}(K)$ is multiplicative for split
unions of webs in $S^{3}$ for example. On the other hand, the proof of
the exact triangles is independent of the choice of atom, so we have
(for example) an octahedral diagram for $L^{\Th}$, just as in
section~\ref{sec:exact-triangle}.

The discussion of the canonical $\Z/2$ grading and the Euler
characteristic from section~\ref{sec:absolute} also needs no change.
In particular, the Euler characteristic of $L^{\Th}$ is again the
count of Tait colorings with signs. Because there is no trifold locus
now that we have dispensed with the trifold atom $H_{3}$, the $\Z/2$
grading on the instanton homology $L^{\Th}(K)$ can be lifted to a
relative $\Z/6$ grading. The authors have not examined this further,
though is apparent for the unknot that the homology has rank $3$ and
is supported with rank 1 in each of the even gradings mod $6$.

The last important difference between $\Lsharp$ and $L^{\Th}$ is that
the edge operators $x_{e}$ for the latter satisfy a different cubic
relation.

\begin{lemma}
    The operators $x_{e}$ on $L^{\Th}(K)$ for a web $K$ satisfy
    $x_{e}^{3}=0$.
\end{lemma}

\begin{proof}
The argument in \cite{Xie} does not rely on excision and shows that
there is a relation $x_{e}^{3} + \nu x_{e}=0$, where $\nu$ is the
operator associated to the characteristic class $c_{2}(P)$ for the
basepoint bundle $P$ at a non-singular point $x_{0}$ of the bifold. We
must show $\nu=0$ on $L^{\Th}(K)$ for all $K$.

We consider the trivial
cobordism $[0,1]\times Z$, where $Z=(Y,K) \# (S^{3}, \Th)$ in the
usual way, and consider $x_{0}$ as a point near $\{1/2\}\times \Th$.
Let $B$ a neighborhood of $x_{0}$ that meets $[0,1]\times\Th$ in
subset $[1/4,3/4]\times \Th$ so that the boundary of $B$ is $(S^{3},
\Th \sqcup \Th)$. By pulling out $B$, we see that it is enough to show
that $\nu=0$ on $L^{\Th}(K)$ in the special case that $K=\Th$.

When $K=\Th$, we consider next the first two relations in
Lemma~\ref{lem:dot-migration}, where $e_{1}$ is the tangled ``chain''
in $\Th$ and $e_{2}=e_{3}$ are both the same ``cuff'', meeting $e_{1}$
at a trivalent vertex therefore. Because we considering $L^{\Th}$
rather then $\Lsharp$, the right-hand side of the second relation is
not $1$ but the operator $\nu$, as the proof of the lemma shows. From
the first relation of Lemma~\ref{lem:dot-migration}, we learn that
$\sigma_{1}(e_{1})=0$, and from the second relation we then see
$\sigma_{1}(e_{2})^{2}=\nu$. On the other hand, the representation
variety of $(S^{3}, \Th \sqcup \Th)$ is isomorphic to $\PSU(3)$ whose
$\Z/2$ homology is non-zero in degrees $0,2,3,5$ modulo $6$. So
$L^{\Th}(\Th)$ is potentially non-zero only in these relative degrees
modulo $6$. The operator $\sigma_{1}(e_{2})$ has degree $2$, and must
have square zero because of the gradings. This shows that $\nu=0$ as
claimed.
\end{proof}

\begin{remark}
    The proof of the lemma above is more elaborate than would have
    been the case if we had an excision property, which would have
    allowed us to base the proof on the case of the empty, for which
    $L^{\Th}$ is supported in a single grading mod $6$, where it is
    apparent that the degree-4 operator $\nu$ must be zero.
\end{remark}

Unlike the case of $\Lsharp$ where we have the edge decomposition from
section~\ref{subsec:edge} resulting from the factorization of
$x^{3}+x$, there is no direct sum decomposition of $L^{\Th}(K)$ in
general, and we cannot readily compute it for planar graphs. It has
much in common with the $\SO(3)$ homology, $\Jsharp(K)$, and one
should ask whether it shares the following property:

\begin{question}
    Is $L^{\Th}(K)$ always non-zero for webs $K$ which do not have an
    embedded bridge in the sense of \cite{KM-Tait}?
\end{question}

Many of the calculations that we have presented for $\Lsharp$ can be
adapted readily for $L^{\Th}$, with only the module structure changing.
For example, by exploiting the $\Z/6$ grading and the relation
$x^{3}=0$, we can see that $L^{\Th}(\theta)$ is isomorphic to the
ordinary cohomology ring for the flag manifold $\SU(3)/T$. The
calculations for the Hopf link, the linked handcuffs, the trefoil and
$K_{3,3}$ in the previous section are all presented in a way that
adapts for $L^{\Th}$, though the module structure is different. On the
other hand, we do not have a calculation of $L^{\Th}$ for the tangled
handcuffs $L^{\Th}(\Th)$ or for the Kinoshita theta web, because those
calculations for $\Lsharp$ depended on the edge decomposition.

The authors do not know of an example where $L^{\Th}(K)$ and
$\Jsharp(K)$ have different dimensions, but this may simply reflect
the fact that our toolkits for calculation are similarly limited.

\subsection{Comparison with previous work}

In \cite{KM-yaft}, a general framework was constructed for gauge
theory with arbitrary compact structure group $G$, which was applied
to the case of three-dimensional orbifolds $(Y,K)$ in which the
singular set $K$ was an oriented embedded link. (More generally there,
the geometrical setup studied connections defined on the complement of
$K$, without the requirement that the local geometry be of orbifold
type: the case of orbifolds arose as a special case.) A suitable
``atom'' was found only in the case of $\SU(N)$, which restricted the
applicability to the construction when the eventual goal was to define
instanton Floer homology groups.

Despite the generality of the earlier work, the construction of
$\Lsharp(Y,K)$ (restricted perhaps to knots and links) is not a
special case of the results of \cite{KM-yaft}. To explain why this is
so, we recall the earlier framework for the case of simply-connected
compact group $G$. Let $\mathfrak{t}$ be the Lie algebra of the
maximal torus, and let $W\subset \mathfrak{t}$ be the (closed)
fundamental Weyl chamber with respect to some choice of positive
roots. Let $W_{1}\subset W$ be the fundamental \emph{Weyl alcove},
defined as
\[
          W_{1} = \{\, \alpha \in W \mid  \theta(\alpha) \le 1 \,\}.
\]
In the case of $\SU(N)$, with standard choices, $W_{1}$ is the set of
diagonal matrices
\[
       \alpha = i\, \mathrm{diag}(\lambda_{1}, \dots, \lambda_{N}),
\]
where $\lambda_{j}\ge \lambda_{j+1}$ for all $j$ and
$\lambda_{1}-\lambda_{N}\le 1$. The image of $W_{1}$ under the
exponential map $\alpha\mapsto \exp(2\pi \alpha)$ meets every
conjugacy class in $G$, so when studying (for example) flat
connections on $Y\setminus K$ we can treat the general case by
choosing $\Phi \in W_{1}$ and considering connections for the which
the monodromy on the oriented meridian of $K$ at each component lies
in the conjugacy class of $\exp(2\pi\,\Phi)$. As $\Phi$ varies in
$W_{1}$, there are finitely many possibilities that arise for the
stabilizer $G_{\Phi}\subset G$ of the Lie algebra element $\Phi$ under
the adjoint action. For each of these possible stabilizers $G_{p}$,
there it turns out that there is a unique $\Phi_{p}$ for which the
dimension formula has a \emph{monotone} property \cite{KM-yaft}.

In \cite{KM-yaft}, the framework is not quite this general: it is
required that $\Phi$ does not lie on the ``far wall'' of the alcove,
meaning that
\[
           \Phi \in W'_{1} =  \{\, \alpha \in W \mid  \theta(\alpha) <
           1 \, \}.
\]
This extra constraint (the strict inequality for $\theta(\alpha)$)
excludes the case that leads to $\Lsharp$.

\begin{figure}
    \begin{center}
        \includegraphics[scale=.46]{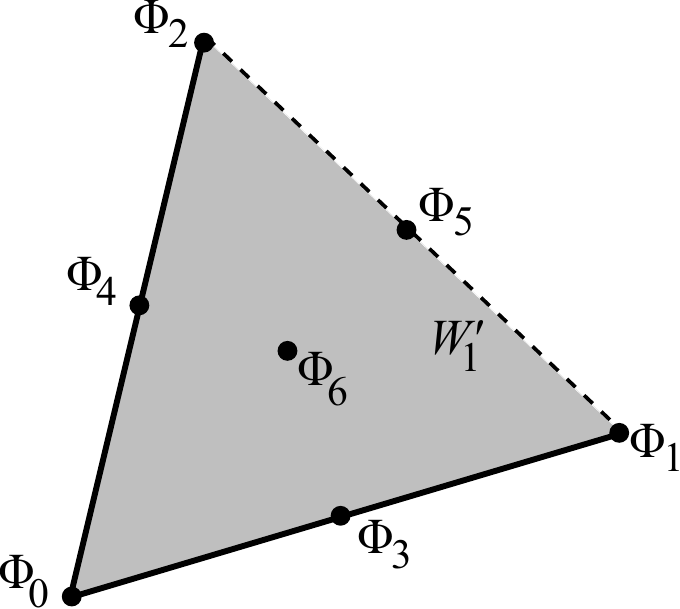}
    \end{center}
    \caption{\label{fig:alcove}
    The alcove $W'_{1}\subset W_{1}$ and the various monotone cases
    for $\Phi$. The bifold bundles of this paper correspond to
    $\Phi_{5}$.}
\end{figure}

To see this in detail, we can list the elements $\Phi\in W_{1}$ which
lead to a monotone gauge theory. In the standard Weyl alcove these are
as follows (see Figure~\ref{fig:alcove}):
\begin{itemize}
\item    $\Phi_{0} = i \, (0,0,0)$;
\item    $\Phi_{1} = i \, (1/3,1/3,-2/3)$;
\item    $\Phi_{2} = i \, (2/3,-1/3,-1/3)$;
\item    $\Phi_{3} = i \, (1/6,1/6,-1/3)$;
\item    $\Phi_{4} = i \, (1/3,-1/6,-1/6)$;
\item    $\Phi_{5} = i \, (1/2,0,-1/2)$;
\item    $\Phi_{6} = i \, (1/3,0,-1/3)$.
\end{itemize}
The group elements $\exp(2\pi\,\Phi_{k})$ for $k=0,1,2$ are the three
central elements of $\SU(3)$, and are less interesting in the current
context: the resulting connections in the adjoint $\PSU(3)$ bundle
extend smoothly over the locus $K$. The element $\exp(2
\pi\,\Phi_{6})$ has eigenvalues the three cube roots of unity and is
the monodromy that is relevant in the description of the atom $H_{3}$
used in this paper. Of the remaining three, $\Phi_{3}$ and $\Phi_{4}$
belong to the framework of \cite{KM-yaft}. The stabilizers $G_{\Phi}$
of $\exp(2\pi\,\Phi)$ are two different copies of $U(2)$ in $\SU(3)$.
Significantly, the case of $\Phi_{5}$ does \emph{not} fall into the
earlier framework, because $\theta(\Phi_{5})=1$. The exponential
$\exp(2\pi \,\Phi_{5})$ is the diagonal matrix $(-1,1,-1)$ which is
precisely the bifold holonomy used for $\Lsharp$. Its stabilizer is a
third copy of $U(2)$in $\SU(3)$. Note that this stabilizer is larger
than the stabilizer of the Lie algebra element $\Phi_{5}$ itself.

The distinction between $W_{1}$ and $W'_{1}$ arises also for
connections in dimension two, on a punctured Riemann surface. When
$\Phi\in W'_{1}$, such flat connections have an interpretation as
stable bundles with parabolic structure, by an extension of the
theorem of Mehta and Seshadri \cite{Mehta-Seshadri}. In this context,
the restriction $\theta(\Phi)<1$ is discussed in
\cite{Teleman-Woodward}, where it is pointed out that the
Mehta-Seshadri
correspondence breaks down when $\theta(\Phi)=1$.

Nevertheless, there is an important case in which the cases of
$\Phi_{3}$, $\Phi_{4}$ and $\Phi_{5}$ become equivalent. The group
elements $\exp(2\pi\,\Phi)$ in these three cases differ by central
elements of $\SU(3)$. Consider the case that $K$ is an oriented knot
in a 3-manifold $Y$ or an oriented embedded surface in a 4-manifold
$X$. Write $Z$ for $X$ or $Y$ in either case. Let $m\cong S^{1}$ be an
oriented of $K$ in $Z\setminus K$, and suppose that the map
$H^{1}(Z\setminus K) \to H^{1}(m)$ is surjective, either with integer
coefficients or (slightly more generally) with $\Z/3$ coefficients. If
$K$ has more than one component with meridians $m_{i}$, we require
that there is an element in $H^{1}(Z\setminus K)$ that restricts to
the oriented generator of each. In this case, there is a flat complex
line bundle $\xi\to Z\setminus K$ whose holonomy around every meridian
is $e^{2\pi i / 3}$. Given a rank-3 vector bundle $E$ with connection
and structure group $\SU(3)$ on $Z\setminus K$, we can form another by
tensoring with $\xi$:
\[
          E \mapsto E\otimes \xi.
\]
This operation carries the holonomy parameter $\exp(2\pi\,\Phi_{5})$
(the bifold case relevant to $\Lsharp$) to $\exp(2\pi\,\Phi_{4})$. It
provides a an isomorphism between the bifold configuration space
$\bonf(Z,K)$ and the configuration space for $\Phi_{4}$ from
\cite{KM-yaft}. Similarly, tensoring with $\xi^{-1}$ takes us to
$\Phi_{3}$.

The hypothesis on the existence of the flat bundle $\xi$ is always
satisfied if $K$ is an oriented classical knot or link. So in this
case, with suitable atom, $\Lsharp(K)$ is isomorphic to the Floer
homology defined in \cite{KM-yaft} for rank 3 and holonomy parameter
$\Phi_{4}$ or $\Phi_{3}$. Specifically then, with $\Z/2$ coefficients,
the homology group introduced as $\mathit{FI}^{3}_{*}(K)$ in
\cite[Definition 4.1]{KM-yaft} for oriented classical knots is
isomorphic to $\Lsharp(K)$. Already for elementary cobordisms between
knots however, we may introduce non-orientable surfaces, and $\xi$ may
not exist on the complement of the surface.

When $S\subset X$ is an oriented surface in a closed 4-manifold, the
flat line bundle $\xi$ will not exist if $[S]$ is a primitive class in
homology. It is interesting to note that in this situation, with the
extra hypothesis that $b^{+}(X)>0$, a generic perturbation ensures
that the moduli spaces of \cite{KM-yaft} contain no reducible
solutions. This does not hold for the bifold solutions in
$\bonf(X,S)$. To see this, observe first that we can expect there to
be non-singular $\SO(3)$ solutions on $X$ with Stiefel-Whitney class
$w_{2}$ dual to $[S]$ (by the results of \cite{Taubes-non-empty}), and
that these will continue to exist when $X$ is equipped with a bifold
metric which is singular along $S$ \cite{S-Wang-thesis}. These
$\SO(3)$ connections can be interpreted as bifold $\SU(2)$ connections
with monodromy $-1$ at the singular set. They then become $\SU(3)$
bifold connections via the inclusion of $\SU(2)$ in $\SU(3)$. In this
way, we see that there will always be reducible $\SU(3)$ bifold
solutions in $\bonf(X,S)$, with stabilizer of order $2$ and monodromy
$\SU(2)$. See the remark following the proof of
Lemma~\ref{lem:involution-fixed}.

\appendix

\section{\texorpdfstring{}{Appendix: }%
The second Chern class operator}

\subsection{Background}\label{app:background}

We give a proof here of Lemma~\ref{lem:nu-is-1}, which states that the
operator $\nu$ is $1$ on $\Lsharp(\check Y)$, for any bifold $\check
Y$. As noted there, it is enough to treat the case that $\check Y =
S^{3}$. The definition of $\nu$ and $\Lsharp$ means that the statement
to be proved involves a 4-dimensional moduli space of instantons on
the orbifold $\R\times H_{3}$, where $H_{3}$ is the trifold atom whose
singular locus is a Hopf link $H$. Specifically, there is a unique
critical point $\fa$ in $\bonf^{w}(H_{3})$ and a homotopy class of
paths from $\fa$ to $\fa$ such that the corresponding moduli space of
trajectories has dimension $4$. From the dimension formula, the action
of these solutions is $\kappa=1/3$, so we denote the moduli space by
$M_{1/3}(\fa,\fa)$.

The moduli space is
non-compact first because of the translation action and second because
action can bubble off at a point on the singular locus of the
orbifold. On the trifold, the smallest amount of action that is lost
in bubbling is $1/2$, so the weak limit in any bubbling will be the
 flat connection on the cylinder. So the ideal solutions that occur as
 Uhlenbeck limits have the form $([A_{0}], q)$, where $[A_{0}]$ is the
 flat connection and $q\in \R\times H$ is a point on the singular
 locus of the orbifold cylinder. Let us write
 \begin{equation}\label{eq:M-Uhl}
       M^{\Uhl}_{1/3}(\fa,\fa) =   M_{1/3}(\fa,\fa) \; \cup \;
       (\R\times H)
 \end{equation}
for this Uhlenbeck completion. To describe a compact space, we still
need to account for the action or $\R$ by translations, so we write
\[
          \hat M = M^{\Uhl}_{1/3}(\fa,\fa) \; \cup \; \{ \, -\infty,
          \infty \},
\]
where the points $\pm \infty$ denote the translation-invariant weak
limits as the action slides off to $\pm \infty$ respectively on the
cylinder.

 If $x$ is a basepoint in $\R\times H_{3}$ away from the singular
 locus $\R\times H$, then we have the associated basepoint bundle on
 the moduli space. This lifts to an $\SU(3)$ bundle, and we write this
 bundle as
\[\mathbf{E}_{x} \to M_{1/3}(\fa, \fa).\]
The basepoint bundle extends across points of the Uhlenbeck completion
where bubbling occurs, because the convergence at $x$ is always
strong, so we have a bundle
\[\mathbf{E}^{\Uhl}_{x} \to M^{\Uhl}_{1/3}(\fa, \fa).\]
The basepoint bundle is also trivial on the two ends of the moduli
space corresponding to the translation action so it extends to a
bundle $\hat{\mathbf{E}}_{x}$ on the compact space:
\[
        \hat{\mathbf{E}}_{x} \to \hat M.
\]
Our task is to
compute
\begin{equation}\label{eq:nu-eval}
           \nu = c_{2}(\hat{\mathbf{E}})[\hat M].
\end{equation}
Although
Lemma~\ref{lem:nu-is-1} only describes the second Chern class operator
mod $2$, we work here with integer coefficients, and will show:

\begin{proposition}\label{prop:nu-is-3}
    When the moduli space $M_{1/3}(\fa,\fa)$ is given its complex
    orientation, the value of $\nu$ in \eqref{eq:nu-eval} is $3$.
\end{proposition}

The space $\R\times H_{3}$ is globally a quotient of $\R\times S^{3}$,
which is conformally $\R^{4}\setminus \{0\}$. Understanding this
4-dimensional moduli space explicitly therefore involves understanding
instantons on $\R^{4}$ which are invariant under the action of a
finite group. This we can do using the ADHM construction of
instantons.

In using the ADHM construction in this sort of context, we are
following the strategy employed by Daemi-Scaduto in
\cite{Daemi-Scaduto} and by Austin in \cite{Austin}. See also
\cite{Kronheimer-Nakajima}.

Our calculation of $\nu$ is essentially the same result as is proved
in \cite{Xie}, but using a different atom. The two arguments have
something in common, because while we use the ADHM construction, the
closely-related Fourier-Mukai transform is used in \cite{Xie}. We
present the calculation for $\SU(3)$ only, but it is quite apparent
that the same argument shows that $\nu=N$ for the $\SU(N)$ case, with
a suitable atom, a result also obtained in \cite{Xie}. We discuss this
briefly at the end.

\subsection{The ADHM construction on the orbifold}

Our convention and notations for the ADHM correspondence mostly
follows \cite{Donaldson-Kronheimer}. Let $U$ denote Euclidean $\R^{4}$
or $\C^{2}$ with complex coordinates $z_{1}$ and $z_{2}$. We consider
bundles $(E,A)$ with connection, having rank $N$ and structure group
$\SU(N)$, such that the curvature of $A$ is anti-self-dual with
Chern-Weil integral $k\in \mathbb{N}$. We also consider such
instantons $(E,A)$ equipped with a framing, by which we mean a special
unitary isomorphism $\phi: E_{\infty}\to \C^{N}$, where $E_{\infty}$
is the fiber over infinity for the extension of $(E,A)$ to $S^{4}$
(which is provided by Uhlenbeck's theorem on the removal of
singularities). The ADHM construction provides a one-to-one
correspondence between isomorphism classes of such framed instantons
on $U$ and equivalence classes of ``ADHM data'' $(\cH, \tau_{1},
\tau_{2}, \pi,\sigma)$. Here $\cH$ is a $k$-dimensional complex vector
space with inner product, and the other items are linear maps,
\[
\begin{aligned}
\tau_{i} :\null &\cH \to \cH \\
\pi : \null &\cH \to E_{\infty}\\
\sigma:\null & E_{\infty} \to \cH.
\end{aligned}
\]
These are required to satisfy the ADHM equations,
\[
\begin{aligned}
\relax    [\tau_{1},\tau_{2}] + \sigma\pi &= 0\\
          [\tau_{1},\tau_{1}^{*}] +[\tau_{2},\tau_{2}^{*}]
          + \sigma\sigma^{*}  - \pi^{*} \pi& = 0,
\end{aligned}
\]
and also a non-degeneracy condition (see below). Equivalence for ADHM
data is defined by the action of the unitary group $U(\cH)$.

Given ADHM data and a point $z\in\C^{2}$, let
\[
      \alpha_{z}= 
      \begin{pmatrix}
        \tau_{1} -z_{1} \\ \tau_{2} -z_{2} \\ \pi
      \end{pmatrix} , \qquad\qquad
      \beta_{z}= 
      \begin{pmatrix}
        -\tau_{2} +z_{2} & \tau_{1} -z_{1} & \sigma.
      \end{pmatrix}
\]
The ADHM equations
can be stated as \[ \beta_{0}\alpha_{0}=0, \qquad
\alpha_{0}\alpha_{0}^{*}- \beta_{0}\beta_{0}^{*}
= 0,\] (with the same holding automatically for other $z\ne 0$). The
non-degeneracy condition requires that $\alpha_{z}$ is injective and
$\beta_{z}$ is surjective, for all $z\in \C^{2}$. The instanton bundle
$E$ is recovered from the ADHM data by
\[
   E_{z} = \ker(\beta_{z}) / \im (\alpha_{z}).
\]
In the inverse construction, the inner product space $\cH$ arises from
$(E,A)$ as the $L^{2}$ kernel of the coupled Dirac operator,
$D^{-}_{A}: \Gamma(S^{-}\otimes E) \to \Gamma(S^{+}\otimes E)$.

When considering the naturality of this construction, one should
introduce the 1-dimensional vector space $\Lambda^{2}U$ and write
\[
\begin{aligned}
(\tau_{1}, \tau_{2}):\null \cH &\to \cH \otimes U \\
(- \tau_{2}, \tau_{1}):\null \cH \otimes U &\to \cH \otimes \Lambda^{2}U \\
\pi : \null \cH &\to E_{\infty}\\
\sigma:\null  E_{\infty} &\to \cH \otimes \Lambda^{2} U.
\end{aligned}
\]
With this in mind, we have
\[
\cH \stackrel{\alpha_{z}}\longrightarrow \left(\cH \otimes U\right) \oplus
E_{\infty} \stackrel{\beta_{z}}\longrightarrow \cH \otimes \Lambda^{2}
U.
\]

For now, let us consider the case of rank-3 bundles ($N=3$). Let $G$
be the finite group of order 27 described by
\eqref{eq:central-extension-3} and $V_{9}=C_{3}\times C_{3}$ its
abelian quotient. We seek framed instantons $(E,A,\phi)$ over $U$
which are invariant under an action of $G$, where $G$ acts on
$E_{\infty}=\C^{3}$ by the representation $\rho$ from
section~\ref{subsec:atoms} and acts on $U=\C^{2}$ via its abelian
quotient $V_{9}$ by
\begin{equation}\label{eq:G-on-U}
       \Gg\mapsto 
       \begin{pmatrix}
        \omega & 0 \\ 0 & \omega^{-1}
       \end{pmatrix} ,\qquad 
       \Gh\mapsto 
       \begin{pmatrix}
        \omega & 0 \\ 0 & 1
       \end{pmatrix}.
\end{equation}
Since the action of $G$ is irreducible, we can drop the framing $\phi$
from our discussion, because it is unique up to isomorphism. Let
$M_{k}(U)^{G}$ denote this moduli space of $G$-invariant instantons of
charge $k$.

The quotient of $U\setminus 0$ by the action of $V_{9}$ is conformally
equivalent to the product $\R\times H_{3}$, where $H_{3}$ is the
trifold atom. If $\fa$ denotes the unique flat connection in
$\bonf^{w}(H_{3})$, then the moduli space $M_{k}(U)^{G}$ can be
identified with the moduli space of trajectories $M_{\kappa}(\fa,\fa)$
on $\R\times H_{3}$ with action $\kappa=k/9$. The dimension of this
moduli space is $12\kappa$, and we are concerned with the
4-dimensional moduli space. So the integer $k=\dim\cH$ needs to be
$3$.

So we examine the ADHM data $(\cH, \tau_{1},\tau_{2},\pi,\sigma)$ for
instantons in $M_{3}(U)^{G}$, where the vector space $\cH$ now
acquires an action of $G$. The space $\cH$ is identified with the
$L^{2}$ kernel of the Dirac operator on $U$ coupled to $(E,A)$, and
since the commutator $\bar\gamma=[\Gg, \Gh]$ acts trivially on $U$ and
acts as the scalar $\omega$ on $E_{\infty}$, it follows that the
commutator also acts by $\omega$ on $\cH$. This forces the action of
$G$ on $\cH$ to be by the same representation $\rho$ as on
$E_{\infty}$, given by equation \eqref{eq:rep-rho} in
section~\ref{subsec:atoms}.

Choose a unitary isomorphism between $E_{\infty}$ and $\cH$,
respecting the actions of $G$. Since $\pi: \cH\to E_{\infty}$ is
$G$-equivariant, it is a scalar multiples of the identity as an
endomorphism of $\cH$. The map $\sigma$ is naturally a map
$E_{\infty}\to \cH \otimes \Lambda^{2} U$, and again these are
isomorphic representations, so we choose a linear isomorphism of
$G$-spaces, $S:E_{\infty}\to \cH \otimes \Lambda^{2} U$, and we have
\[
        \pi = a \, 1_{\cH} , \qquad \sigma = b \, S,
\]
for some scalars $a$ and $b$. Similarly, the vector space $U=\C^{2}$
decomposes as $U_{1} \oplus U_{2}$, where these are two characters of
$G$ (via the abelian quotient $V_{9}$) corresponding to the top-left
and bottom-right entries of the matrices \eqref{eq:G-on-U}. The
representations $\cH \otimes U_{i}$ are both isomorphic to $\cH$,
which is irreducible, so let us choose unitary isomorphisms of
$G$-spaces,
\[
            F_{1}:  \cH \to \cH\otimes U_{1}, \qquad  F_{2}:  \cH \to
            \cH \otimes U_{2}.
\]
These also define isomorphisms,
\[
       F_{1} : \cH\otimes U_{2} \to \cH \otimes \Lambda^{2} U ,
       \qquad F_{2} : \cH\otimes U_{1} \to \cH \otimes \Lambda^{2} U .
\]
With this in mind, we have
\[
      \tau_{1} = t_{1} F_{1}, \qquad \tau_{2} = t_{2} F_{2},
\]
for complex scalars $t_{1}$, $t_{2}$.
The ADHM equations now become,
\begin{equation}\label{eq:scalar-ADHM}
        t_{1} t_{2} + C ab = 0, \quad |a|^{2} - |b|^{2} = 0,
\end{equation}
where the constant $C\in\C$ is the scalar $S = C [F_{1}, F_{2}]$. We
are free to scale $S$, so we can take it that $C=1$ to simplify the
exposition. The unitary transformations of $\cH$ that commute with $G$
are just the scalars, so the equivalence classes of ADHM data are the
orbits of the circle $S^{1}$ acting by
\[
       (t_{1}, t_{2}, a, b) \mapsto (t_{1}, t_{2}, \lambda a ,
       \lambda^{-1} b).
\]
The non-degeneracy condition is the condition $t_{1} t_{2}\ne 0$.

To within the action of $S^{1}$, the equations determine $a$ and $b$
in terms of $t_{1}$ and $t_{2}$. So a solution to the ADHM
equation is determined by $(t_{1}, t_{2})$ and we can set $a=a(t_{1},
t_{2})$ and $b=b(t_{1}, t_{2})$, where
\begin{equation}\label{eq:ab-ADHM}
\begin{aligned}
(t_{1}, t_{2}) 
        &= |t_{1} t_{2}|^{1/2} e^{i\theta_{1}} \\
        b(t_{1}, t_{2}) 
        &= |t_{1} t_{2}|^{1/2} e^{i\theta_{2}} \\
        \end{aligned}
\end{equation}
where $t_{j} = |t_{j}| e^{i\theta_{j}}$. We summarize this description
of the moduli space in a proposition.

\begin{proposition}\label{prop:M13}
The moduli space $M_{1/3}(\fa,\fa)
\cong M_{3}(U)^{G}$ is identified with $\C^{*}\times \C^{*}$ via
complex
coordinates $(t_{1}, t_{2})$, so
\begin{equation}\label{eq:M1/3-description}
         M_{1/3}(\fa,\fa) \cong \R \times (S^{3}\setminus H)
\end{equation}
where $H$ is a Hopf link. The solution corresponding to $(t_{1},
t_{2})$ in this description arise from the ADHM matrices $\alpha_{z}$
and $\beta_{z}$ that can be written as
\begin{equation}\label{eq:alpha-beta}
      \alpha_{z}= 
      \begin{pmatrix}
        t_{1}F_{1} - z_{1} \\ t_{2}F_{2} -z_{2} \\ a(t_{1}, t_{2})
      \end{pmatrix} , \qquad
      \beta_{z}= 
      \begin{pmatrix}
        -t_{2}F_{2} +z_{2} & t_{1}F_{1} -z_{1} & b(t_{1}, t_{2}) S
      \end{pmatrix},
\end{equation}
where we interpret $z_{i}$ as elements of $U_{i}$, and $a$ and $b$ are
given by \eqref{eq:ab-ADHM}. \qed
\end{proposition}

The ADHM construction also provides a description of the Uhlenbeck
completion of the instanton moduli space. If we drop the condition
that $t_{1}t_{2}$ is non-zero, and ask only that $t_{1}$ and $t_{2}$
are not both zero, we obtain a description of  $M^{\Uhl}_{1/3}(\fa,
\fa)$, as
\[
    M^{\Uhl}_{1/3}(\fa, \fa) \cong \C^{2} \setminus \{0\}
\]
via the same coordinates $t_{1}$, $t_{2}$. The translation action on
$M_{1/3}(\fa, \fa)$ corresponds to the action of positive scalars on
the coordinates $t_{i}$.  Putting in the two limit
points of the translations, we obtain a description of the
compactified moduli space
\[
        \hat M = S^{4}.
\]

\subsection{The basepoint bundle and its second Chern class}

Consider now the basepoint bundle $\mathbf{E}_{z_{*}}$ over the moduli
space, associated with the basepoint $z_{*}=(1,1) \in U$. As noted in
section~\ref{app:background}, the bundle $\mathbf{E}_{z_{*}}$ extends
first as a bundle $\mathbf{E}^{\Uhl}_{z_{*}}$ over the Uhlenbeck
completion $M^{\Uhl}_{1/3}(\fa,\fa)$ and then as a bundle
$\hat{\mathbf{E}}_{z_{*}}$ over the compact space $\hat M$.

The ADHM description of the moduli space also provides a description
of the basepoint bundle. Let us indicate the dependence on $t=(t_{1},
t_{2})$ in the ADHM matrices by writing the matrices in
\eqref{eq:alpha-beta} as $\alpha(t,z)$ and $\beta(t,z)$. With
$z_{*}=(1,1)$, let $\bbE$ be the vector bundle over $\C^{2}$ whose
fiber at $t$ is
\[
          \bbE(t) = \ker(\beta(t,z_{*})) / \im (\alpha(t, z_{*})).
\]
Over $\C^{*}\times\C^{*} = M_{1/3}(\fa,\fa)$, this describes the
basepoint bundle $\mathbf{E}_{z_{*}}$ via the ADHM construction, and
this extends to an isomorphism between $\bbE$ and
$\mathbf{E}^{\Uhl}_{z_{*}}$ over the Uhlenbeck completion
$\C^{2}\setminus \{0\} = M^{\Uhl}_{1/3}(\fa, \fa)$. 
To understand the behavior at the two additional points $\pm\infty$
in the compactification $\hat M$, we need to examine the ends more
carefully.

Let us introduce the projective space $\CP^{4}$ with coordinates
$[t_{1}, t_{2}, a, b, u]$ and the map
\[
\begin{gathered}
w : M_{1/3}(\fa,\fa) \to \CP^{4} \\
    (t_{1}, t_{2}) \mapsto (t_{1}, t_{2}, a(t_{1}, t_{2}), b(t_{1},
    t_{2}), 1),
\end{gathered}
\]
taking $t_{1}, t_{2}$ still to be coordinates on $M_{1/3}(\fa,\fa) =
\C^{*}\times \C^{*} \subset \C^{2}$. This map does not extend
continuously to $\hat M = \C^{2}\cup \{\infty\}$, but it does extend
continuously to a larger compactification $\tilde M = \CP^{2}$, by the
(well-defined) map
\[
        \tilde w:   [t_{1}, t_{2}, u] \mapsto
        [t_{1}, t_{2}, a(t_{1}, t_{2}), b(t_{1},
    t_{2}), u].
\]
Let
\[
      p : \tilde M \to \hat M
\]
be the quotient map which collapses $\CP^{1}$ to a point.

In the coordinates $[t_{1}, t_{2}, a, b, u]$ of $\CP^{4}$, let us
introduce homogeneous versions of the maps $\alpha$, $\beta$ as
      \begin{equation}\label{eq:alpha-beta-hom}
      \alpha = 
      \begin{pmatrix}
        t_{1}F_{1} - u \\ t_{2}F_{2} -u \\ a
      \end{pmatrix} ,\qquad
      \beta= 
      \begin{pmatrix}
        -t_{2}F_{2} +u & t_{1}F_{1} -u & b  S
      \end{pmatrix},
\end{equation}         
which we interpret over projective space as bundle maps,
\[
\cH\times \cO(-1) \stackrel{\alpha}\longrightarrow \cH \otimes U \oplus
E_{\infty} \stackrel{\beta}\longrightarrow \cH \otimes \Lambda^{2}
U \otimes \cO(1).
\]
The composite $\beta\alpha$ is zero only over the locus $Q =
\{t_{1}t_{2}=ab\}$ in $\CP^{4}$. Over $Q$, the map $\alpha$ is
injective and $\beta$ is surjective except at the points $[0,0,1,0]$
and $[0,1,0,0]$ respectively. Let $Q'\subset Q$ denote the complement
of these two points. Over $Q'$, the ADHM description provides us with
a bundle
\[
        \mathbb{E} = \ker(\beta)/\im(\alpha).
\]

\begin{lemma}
    The bundles $p^{*}(\hat {\mathbf{E}}_{z_{*}})$ and $\tilde
    w^{*}(\mathbb{E})$ over $\tilde M = \CP^{2}$ are isomorphic.
\end{lemma}

\begin{proof}
    Consider the open manifold $\hat M \setminus \{\infty\} = \R^{4}$
    and let $M^{+} \cong B^{4}$ be obtained by attaching a 3-sphere at
    infinity. So $M^{+}$ is the real oriented blow up of $\hat M =
    S^{4}$ at the point at infinity. Invariance under scaling the
    variables, or equivalently invariance under translation of the
    cylinder $\R\times H_{3}$, provides an isomorphism between the
    basepoint bundle $\mathbf{E}_{z=(1,1)}$ on $M = M_{1/3}(\fa,\fa)$
    restricted to the 3-sphere $\| t\| = 1/\epsilon$ and the basepoint
    bundle $\mathbf{E}_{z=(\epsilon,\epsilon)}$ restricted to the
    3-sphere $\|t\|=1$. If we regard $M$ as the moduli space
    $M_{3}(U)^{G}$, then the fiber at $0\in U$ is defined, and we see
    that the basepoint bundle $\mathbf{E}_{z=(1,1)}$ extends over the
    3-sphere $\partial M^{+}$ and can be identified there with the
    restriction of $\mathbf{E}_{z=0}$ to the 3-sphere $\|t\|=1$.

    The bundle $\hat{\mathbf{E}}_{z=(1,1)}$ is obtained from the
    bundle $\mathbf{E}_{z=(1,1)}$ as an identification space, by a
    trivialization of the bundle on $\partial M^{+}$. Equivalently, it
    is determined by a trivialization $\psi$ of $\mathbf{E}_{z=0}$
    over the sphere $\|t\|=1$. The finite group $G$ acts on $U$ and
    acts on the fibers of basepoint bundle $\mathbf{E}_{z=0}$. The
    trivialization $\psi$ respects the action of $G$, and since the
    action is irreducible on the fibers, this condition characterizes
    $\psi$ uniquely.

    The bundle over $\tilde M = \CP^{2}$ is similarly obtained from a
    bundle over the 4-ball $M^{+}$, as $\CP^{2}$ is obtained from
    $B^{4}$ by quotienting by the Hopf fibration on the boundary
    $S^{3}$. The ADHM construction describes the instanton bundles on
    $U$, not just on $U\setminus \{0\}$, so the appropriate bundle on
    $M^{+}$ is as before and carries the same action of $G$ on the
    fibers over the boundary. The bundle $\tilde w^{*}(\mathbb{E})$ is
    described by trivializing the bundle on $\partial M^{+}$ along the
    fibers of the Hopf fibration. Since the trivializations are the
    same in both cases, this identifies $\tilde w^{*}(\mathbb{E})$
    with $p^{*}(\hat {\mathbf{E}}_{z_{*}})$.
\end{proof}

Given the lemma, we can compute $\nu$ as follows. Since the map $p$ is
an isomorphism on fourth homology, we have $\nu = c_{2}(p^{*}(\hat
{\mathbf{E}}_{z_{*}})[\tilde M]$, which is equal to $c_{2}(\tilde
w^{*}( \mathbb{E}))[\tilde M]$ by the lemma. The image of the
fundamental class of $[\tilde M]$ under $\tilde w$ has degree $1$,
because $\tilde w$ is homotopic to the inclusion of a standard copy of
$\CP^{2}$. We can compute the Chern classes of $\mathbb{E}$ from its
definition as
\[
\begin{aligned}
\mathbb{E} &=  \bigl(\cH \otimes U \oplus
E_{\infty}\bigr) \;\ominus\; \bigl( \cH\otimes \cO(-1) \oplus \cH\otimes
\cO(1)\bigr)  \\
&= \C^{9} \;\ominus\;  \bigl(  \cO(-1) \oplus \cO(1) \bigr)\otimes \C^{3}.
\end{aligned}
\]
This description is valid over all of $\CP^{4}$, and gives the total
Chern class of $\bbE$ as
\begin{equation}\label{eq:total-chern}
          c(\bbE) = (1-h^{2})^{-3}.
\end{equation}
This means that $c_{1}=0$ and $c_{2}=3 h^{2}$, where $h$ is the
generator. Hence \[\nu
= 3 h^{2}[ w(\tilde M)]=3\] as required. This completes the proof of
Proposition~\ref{prop:nu-is-3}.

\subsection{The generalization to \texorpdfstring{$\SU(N)$}{SU(N)}}

For every $N\ge 2$ there is an orbifold $H_{N}$ generalizing the
$H_{3}$ that appears here, as the quotient of $S^{3}\subset \C^{2}$ by
the abelian group $C_{N}\times C_{N}$. This product of cyclic groups
has a central extension $G_{N}$ by the group $C_{N}$, and $G_{N}$ has
an irreducible representation in $\SU(N)$ by the generalizations of
the matrices in \eqref{eq:rep-rho}:
\[
\begin{aligned}
\rho(g)=\epsilon&\begin{pmatrix}
1 & 0 & 0 & \cdots & 0 \\
0 & \zeta & 0 & \cdots & 0 \\
0 & 0 & \zeta^2 & \cdots & 0 \\
& & & \ddots & 0 \\
0 & 0 & 0 & 0 & \zeta^{N-1}
\end{pmatrix} \\[1.0ex]
\rho(h)=\epsilon&\begin{pmatrix}
0 & 0 & 0 & \cdots & 1 \\
1 & 0 & 0 & \cdots & 0 \\
0 & 1 & 0 & \cdots & 0 \\
& & & \ddots & 0 \\
0 & 0 & 0 & 1 & 0
\end{pmatrix}\\[1.0ex]
\rho(\gamma)=&\begin{pmatrix}
\zeta & 0 & 0 & \cdots & 0 \\
0 & \zeta & 0 & \cdots & 0 \\
0 & 0 & \zeta & \cdots & 0 \\
& & & \ddots & 0 \\
0 & 0 & 0 & 0 & \zeta
\end{pmatrix}
\end{aligned}
\]
where $\zeta=e^{2\pi i /N}$ and $\epsilon= 1$ if $N$ is odd and an
$N$th root of $-1$ if $N$ is even. See \cite{KM-yaft}. This
representation of $G_{N}$ determines an orbifold $\SU(N)$ bundle on
the complement of an arc $w$ joining the two components of the Hopf
link, just as in the case of $H_{3}$. It is irreducible and defines
the unique critical point $\fa$ in $\bonf^{w}(H_{N})$. There is a
4-dimensional moduli space $M_{\kappa}(\fa, \fa)$ when $\kappa=1/N$,
and these solutions can be interpreted as $G_{N}$-equivariant
instantons with instanton number $N$ on $\R^{4}$.

The basepoint bundle on this moduli space has a second Chern class
that evaluates to an integer $\nu_{N}$. The ADHM description carries
over with essentially no change: the complex vector space $\cH$ is now
$N$-dimensional and carries the action of $G_{N}$ defined by $\rho$,
and the solutions to the ADHM equations are described by
Proposition~\ref{prop:M13} with ``$N$'' replacing ``$3$'' in the
statement. The calculation of $\nu_{N}$ proceeds via the total Chern
class of $\bbE$ on $\CP^{4}$, which is now given by
\[
          c(\bbE) = (1-h^{2})^{-N}.
\]
This gives $\nu_{N}=N$, generalizing the case $N=3$.

\bibliographystyle{abbrv}
\bibliography{SU3}

\end{document}